\documentclass{amsart}
\usepackage{eurosym}
\usepackage{amsmath}
\usepackage{amsfonts}
\usepackage{graphicx}
\usepackage{caption}
\usepackage{subcaption}
\usepackage{verbatim}

\newtheorem{theorem}{Theorem}

\newtheorem{definition}[theorem]{Definition}

\newtheorem{lemma}[theorem]{Lemma}

\newtheorem{proposition}[theorem]{Proposition}
\newtheorem{remark}[theorem]{Remark}


\begin{document}
\title[]{Self-consistent transfer operators. Invariant measures, convergence
to equilibrium, linear response and control of the statistical properties.}
\author{Stefano Galatolo$^1$}
\email{$^{1}$ stefano.galatolo@unipi.it}
\address{Dipartimento di Matematica, Universita di Pisa, Via \ Buonarroti
1,Pisa - Italy}

\begin{abstract}
We describe a general approach to the theory of self consistent transfer
operators. These operators have been introduced as tools for the study of
the statistical properties of a large number of all to all interacting
dynamical systems subjected to a mean field coupling.

We consider a large class of self consistent transfer operators and prove
general statements about existence of invariant measures, speed of
convergence to equilibrium, statistical stability and linear response.

While most of the results presented in the paper are valid in a weak
coupling regime, the existence results for the invariant measures we show
also hold outside the weak coupling regime.

We apply the general statements to examples of different nature: coupled
continuous maps, coupled expanding maps, coupled systems with additive
noise, systems made of \emph{different maps }coupled by a mean field
interaction and other examples of self consistent transfer operators not
coming from coupled maps.

We also consider the problem of finding the optimal coupling between maps in
order to change the statistical properties of the system in a prescribed way.
\end{abstract}

\maketitle
\tableofcontents

\section{Introduction}

Suppose we have a normed real or complex vector space $B_{w}$ and a
collection of linear operators $L_{\delta ,f}:B_{w}\rightarrow B_{w}$ \
depending on some parameter \ $\delta \in \mathbb{R}$ and $f\in B_{w}$. One
can consider the nonlinear function $\mathcal{L}_{\delta }:B_{w}\rightarrow
B_{w}$ defined by%
\begin{equation}
\mathcal{L}_{\delta }(f):=L_{\delta ,f}f.  \label{1.1}
\end{equation}%
These kind of nonlinear functions have been also called \emph{%
self-consistent transfer operators} (as the operator itself depends on the
point at which it is calculated). In many examples where these operators are
used the family $L_{\delta ,f}$ depends in a Lipschitz way on $f$ and the
Lipschitz constant is proportional to $\delta $. In this context the
parameter $\delta $ represents in some sense a measure of the strength of
the nonlinearity of the function $\mathcal{L}$. These concepts have been
introduced and studied as models to describe the collective behavior of a\
network of interacting\ systems coupled by a mean field interaction. In the
case of operators modeling \ coupled extended systems the parameter $\delta $
often represents the strength of the coupling or of the interaction between
the systems.

These operators, their invariant measures and their properties have been
studied by different techniques in classes of examples. In this paper we
attempt a general approach to the study of this kind of operators, and the
statistical properties of their dynamics. The main goal is to investigate
under which assumptions we can establish some basic important properties of
the self-consistent system as the existence of the invariant measure in a
certain regularity class, the convergence to equilibrium, the statistical
stability and response to perturbation of the system.

We show the flexibility and the effectiveness of the approach applying it to
several kinds of self-consistent operators coming from coupled expanding
maps, coupled random systems and other examples.

We study the behavior of the invariant measures of these operators, their
convergence to equilibrium and their statistical stability mostly in the
"weak coupling" regime, in the sense that most of the results presented hold
for intervals of values of the type $\delta \in \lbrack 0,\overline{\delta }%
] $ for a relatively small $\overline{\delta }$ (with an estimate for the
size of $\overline{\delta }$, hence not only for $\delta \rightarrow 0$).
Some of the results presented however can be applied even for large values
of $\delta $. \ In particular, under suitable assumptions, we prove the
existence of some invariant probability measure for the self-consistent
operator $\mathcal{L}_{\delta }$, also providing estimates on its regularity
imposing no restrictions on the size of $\delta $ (Theorem \ref{existence1}%
), such result then also hold in a \emph{strong coupling }regime, for which
very few results are known. Sufficient assumptions for the uniqueness of the
invariant measure are then shown in the case of weak coupling regime
(Theorem \ref{existence}). Still in the weak coupling regime we study the
attractiveness of the invariant measure as a fixed point of \ $\mathcal{L}%
_{\delta }$, providing exponential convergence to equilibrium results
(Theorem \ref{expco}) and study the response of the invariant measure of the
system to changes in the function defining the mean field coupling
interaction in the zero coupling limit (Theorems \ref{ss} \ and \ref%
{thm:linresp}). We also investigate these questions from an optimal control
point of view. Suppose we have an initial uncoupled system and we want to
introduce a coupling which maximizes certain aspects of the statistical
properties of the coupled system, as for example the average of a given
observable. What is the best coupling to be introduced in order to do so?
This is a problem related to the control of the statistical properties of
chaotic and complex systems.

In this paper we have to deal with several concepts: networks of coupled
systems; self-consistent transfer operators; linear response; random and
deterministic systems; optimal response and control of the statistical
properties. To help the reader, each main section dealing with these
concepts will have an introductory part trying to explain the concept, the
main ideas behind and giving some additional references for its deeper
understanding.

\noindent \textbf{Transfer operators. }An efficient method for the study of
transport and the statistical properties of a dynamical system is to
associate to the system a certain transfer operator describing how the
dynamics acts on suitable spaces of measures or probability distributions on
the phase space. Important properties of the original system are related to
fixed points and other properties of these transfer operators. The transfer
operator which is convenient to associate to a dynamical system is usually a
linear operator. Self-consistent transfer operators are nonlinear operators.

As mentioned before, these operators arise as natural models of extended
systems in which there is a large set of interacting dynamical systems and
we consider the dynamics of each element of the large set (the local
dynamics) as being influenced or perturbed by the state of the other
elements in a mean field coupling. This means that the perturbation we apply
to the dynamics of each local system depends on the distribution of the
states of all the other elements of the large system. This global state will
be represented by a probability measure, representing the probability of
finding a generic local system in a given set of states of the phase space.
If now we consider the transfer operator associated with the dynamics of
each local system we have that this linear operator depends on the current
global state of the system. One can furthermore suppose all the local
systems to be homogeneous and consider again the measure representing the
global state of the system as a representative for the probability of
finding a local system in a given state\footnote{%
We could also consider in a similar way interacting systems of different
types, where instead of a single measure representing the distribution of
the states of the systems in a certain phase space we will have a a vector
of measures representing the states of systems of different type (see
Section \ref{diff} for more details).}. Applying the transfer operator
associated to the local dynamics to see how this probability measure
evolves, we have then a transfer operator depending on a certain measure and
acting on the measure itself. This naturally brings us to the formalization
presented in $($\ref{1.1}$)$. From a formal point of view this give rise to
a nonlinear function to be applied to a certain functional space of
measures. In the weak coupling regime however this nonlinear function is a
small nonlinear perturbation of a linear one, simplifying the situation and
the understanding of the properties of this function.

The use of self-consistent operators for the study of networks of coupled
maps was developed from a mathematical point of view in \cite{K} \ and \cite%
{Bl11}. In Section \ref{expla} we explain some of the heuristics behind the
use of these operators for the study of coupled maps. We refer to \cite{CF}
\ and \cite{Sell} for a further discussion on the scientific context in
which these concepts appear and for an accurate bibliography on the subject
(see also \cite{Fe} and \cite{Kan} for other approaches to maps in a global
coupling). For introductory material we also recommend the reading of the
paper \cite{ST}.

\noindent \textbf{Overview of the main results. \ }In Section \ref{exsec} we
show a set of general assumptions on the family of operators $L_{\delta ,f}$%
, \ ensuring that the nonlinear operator $\mathcal{L}_{\delta }$ has a fixed
point of a certain type and hence the associated system has some invariant
probability measure (see Theorem \ref{existence1}). This result is obtained
by topological methods, applying the Brouwer fixed point theorem to a
suitable sequence of finite dimensional nonlinear operators approximating $%
\mathcal{L}_{\delta }$. The assumptions required for this result are related
to the regularity of the family of linear operators $L_{\delta ,f}$ \ when $%
f $ varies\ in a strong-weak topology, the regularity of its invariant
measures (see the assumptions $(Exi1),(Exi1.b),(Exi2)$ in Theorem \ref%
{existence1}) and the existence of a suitable finite dimensional projection,
allowing to apply a kind of finite element reduction of the problem. The
result also holds outside the weak coupling regime and implies a general
statement for the existence of an invariant measure for systems of
continuous maps in the mean field coupling (see Proposition \ref{KB}). We
also discuss the uniqueness of the invariant probability measure. This will
be proved in the weak coupling regime (see Theorem \ref{existence}). The set
of assumptions for the uniqueness, essentially require that the operators $%
L_{\delta ,f}$ and their fixed points depend on $f$ in a Lipschitz way (see
assumptions $(Exi3)$ in Theorem \ref{existence}). The assumptions required
to apply these results are not difficult to be verified, and in the
following sections we show how to apply this general framework to
interacting random and deterministic systems, together with examples of
different kind.

In section \ref{sec1} we take the same point of view with the goal of
investigating the convergence to equilibrium: the attractiveness of the
invariant measure as a fixed point of $\mathcal{L}_{\delta }$ and in the
weak coupling regime we show assumptions under which we can prove
exponential speed of convergence to equilibrium for a general class of
self-consistent transfer operators (see Theorem \ref{expco}). The
assumptions we require are related to convergence to equilibrium and a
common "one step" Lasota Yorke inequality satisfied by each transfer
operator in the family $L_{\delta ,f}$ (see assumptions $(Con1),...,(Con3)$%
). The assumptions made are in a certain sense natural when considering
suitable coupled dynamical systems like expanding maps or random systems
with additive noise, and in the next sections we apply these general results
to several classes of examples.

In Section \ref{sec:linresp}, \ after an introduction to the concept of
Linear Response and \ some related bibliography, we prove a general
statistical stability result (see Theorem \ref{ss}) and a linear response
result for nonlinear perturbations of linear transfer operators (see Theorem %
\ref{thm:linresp}), describing the first order change in the invariant
measure of the system when an infinitesimal perturbation leading to a
nonlinear operator is applied. We remark that this result is similar in the
statement and in the proof to many other general linear response results
proved for linear transfer operators (see e.g. \cite{GS}).

The methods used to establish the general statements in sections \ref{exsec}%
, \ref{sec1}, \ref{sec:linresp} are related to the classical transfer
operator approach, letting the transfer operator associated with the system
to act on stronger and weaker spaces (in a way similar to the classical
reference \cite{KL}), exploiting the fact that the perturbations we are
interested in applying to our systems are small when considered in a kind of
mixed norm, from the stronger to the weaker space.

We show the flexibility of this general approach applying it to systems of
different kind. In particular we will consider coupled deterministic
expanding maps and random maps with additive noise, coupling identical maps
or different ones in a mean field regime. For these examples we will use
simple spaces of functions as $L^{1},C^{k}$, the Sobolev spaces $%
W^{k,1},W^{k,2}$ or the space of Borel signed measures equipped with the
total variation or the Wasserstein distance.

In Section \ref{contmap} we consider continuous maps on the circle with a
mean field coupling and we prove the existence of an invariant probability
measure for the associated self-consistent transfer operators, providing a
sort of Krylov-Bogoliubov theorem for this kind of extended systems.

In Section \ref{secmap}, after recalling several useful classical results on
expanding maps we show that the self-consistent transfer operator associated
with a network of coupled expanding maps has an invariant measure in a
suitable Sobolev space $W^{k,1}$ and we show an estimate for its Sobolev
norm (see Theorem \ref{existenceexp}). In the small coupling regime we also
show exponential convergence to equilibrium for this kind of systems. This
will allow to apply our general linear response statement and get a linear
response statement for the zero coupling limit of such systems. Similar
results for this kind of systems in the weak coupling regime appear in \cite%
{ST}, the spaces used and the methods of proof however are quite different.

In Section \ref{noise} we consider coupled random maps and we apply our
general framework to this case. More precisely, we consider maps with
additive noise in which at every iterate of the dynamics a certain
deterministic map is applied and then a random i.i.d. perturbation is added.
Due to the regularizing effect of the noise at the level of the associated
transfer operators we do not need to put particular restrictions on the maps
considered. These examples are then particularly interesting for the
applications. After recalling the basic properties of these systems and the
associated transfer operators we define a self-consistent transfer operator
representing the global behavior of a network of coupled random maps. We
prove the existence of invariant measures for this self-consistent operator
and show an estimate for its $C^{k}$ norm which is uniform when varying the
coupling strength. In the case of weak coupling, we also prove exponential
speed of convergence to equilibrium for this globally coupled system. We
then apply the general linear response results to this system, obtaining
again a linear response result for the system in the zero coupling limit.

In Section \ref{strange} we consider a class of self-consistent transfer
operators where the deterministic part of the dynamics is driven by a
certain map whose slope depends on the average of a given observable, in
some sense similar to the examples studied in \cite{Se2}. For these systems
we study the existence, uniqueness of the invariant measures and linear
response, similarly to what is proved for the systems coming from coupled
maps.

In Section \ref{diff} we consider suitable self-consistent transfer
operators to model a \emph{mean field interaction of different maps}. For
simplicity we consider two types of maps. We show that the general framework
we are considering also applies to this case, showing the existence and
uniqueness of the invariant measure in a weak coupling regime.

In Section \ref{opt} we consider the linear response results we proved from
an optimal control point of view. Suppose we want to introduce in the system
a coupling which changes the statistical properties of the dynamics in some
desired way. What is the optimal coupling to be considered? Given some
observable whose average is meant to be optimized and a convex set $P$ of
allowed infinitesimal couplings to be applied, we show conditions under
which the problem has a solution in $P$ and this solution is unique. We
remark that in \cite{MK} the research in this direction of research was
motivated, with the goal of the management of the statistical properties of
complex systems and in this direction several results for probabilistic
cellular automata were shown.

\section{Self-consistent transfer operators for coupled circle maps,
heuristics and formalization \label{expla}}

Since the study of self-consistent transfer operators is strongly motivated
by the applications to systems of globally interacting maps, in this section
we briefly introduce a model representing the dynamics of a large number of
coupled maps in a global mean field interaction and the associated \emph{%
self-consistent transfer operators}. We will see how the formalization of
such interaction leads to the study of a self-consistent transfer operator
of the kind defined at the beginning of the introduction.

We remark that in this paper we only consider discrete time dynamical
systems. In the continuous time case, the models one is lead to consider are
related to the topic of Vlasov-type differential equations, we suggest the
recent surveys \cite{F2}, \cite{go} and the references therein for an
introduction to the subject.

We are now going to define more precisely the self-consistent transfer
operators associated with a set of dynamical systems coupled by a mean field
interaction. \ One can think the set of interacting systems as a continuum,
endowed with a measure, as for instance a swarm of interacting particles
distributed by a certain density in different parts of the space. \ We take
this point of view and we consider the case in which the set of systems we
consider is a measurable space $M$ with a probability $p$. The set $M$ can
be finite or infinite and in each case we can define the self-consistent
transfer operator associated with the system. We remark that one could see
the case where $M$ is infinite as a suitable limit of finite sets and define
the self-consistent transfer operator associated with the global coupling of
infinitely many systems by a suitable limit of finitely many couplings (see 
\cite{Bl11}, \cite{K}, \cite{ST} and Footnote \ref{not1} for further details
on this approach).

Let us fix some notation and terminology: let us consider two metric spaces $%
X,Y$, the spaces of Borel probability measures $PM(X),~PM(Y)$ on $X$ and $Y,$
and a Borel measurable $F:X\rightarrow Y$. We denote the pushforward of $F$
as $L_{F}:PM(X)\rightarrow PM(Y)$, defined by the relation%
\begin{equation*}
\lbrack L_{F}(\mu )](A)=\mu (F^{-1}(A))
\end{equation*}%
for all $\mu \in PM(X)$ and measurable set $A\subseteq Y$. The pushforward
can be extended as a linear function $L_{F}:SM(X)\rightarrow SM(Y)$ from the
vector space of Borel signed measures on $X$ to the same space on $Y$. In
this case $L_{F}$ will be also called as the transfer operator associated
with the function $F$.

We now define a model for the dynamics of a family of dynamical systems
interacting in the mean field. For simplicity we will suppose as a phase
space for each interacting system the unit circle $\mathbb{S}^{1}$ and we
will equip $\mathbb{S}^{1}$ with the Borel $\sigma -$algebra. We consider an
additional metric space $M$ \ equipped with the Borel $\sigma -$algebra and
a probability measure $p\in PM(M)$. \ Let us consider a collection of \emph{%
identical dynamical systems} $(\mathbb{S}^{1},T)_{i}$, with $i\in M$ \ and $%
T:\mathbb{S}^{1}\rightarrow \mathbb{S}^{1}$ being a Borel measurable
function.

The initial state of this collection of interacting systems can be
identified by a point $\mathbf{x}(0)=(x_{i}(0))_{i\in M}\in $ $(\mathbb{S}%
^{1})^{M}$ (we suppose $i\rightarrow x_{i}(0)$ being measurable). Let $%
\mathcal{X\subseteq }(\mathbb{S}^{1})^{M}$ be the set of measurable
functions $M\rightarrow \mathbb{S}^{1}$. We now define the dynamics of the
interacting systems by defining a global map $\mathcal{T}:\mathcal{X}%
\rightarrow \mathcal{X}$ and global trajectory of the system by%
\begin{equation*}
\mathbf{x}(t+1):=\mathcal{T}(\mathbf{x}(t))
\end{equation*}%
\ where $\mathbf{x}(t+1)$ is defined on every coordinate by applying at each
step the local dynamics $T$, \emph{plus a perturbation given by the mean
field interaction with the other systems}, by%
\begin{equation}
x_{i}(t+1)=\Phi _{\delta ,\mathbf{x}(t)}\circ T(x_{i}(t))
\end{equation}%
for all $i\in M$, \ where $\Phi _{\delta ,\mathbf{x}(t)}:\mathbb{S}%
^{1}\rightarrow \mathbb{S}^{1}$ represents the perturbation provided by the
global mean field coupling with strength $\delta \geq 0$, defined in the
following way: let $\pi _{\mathbb{S}^{1}}:\mathbb{R}\rightarrow \mathbb{S}%
^{1}$ be the universal covering projection, let us consider some continuous
function $h:\mathbb{S}^{1}\times \mathbb{S}^{1}\rightarrow \mathbb{R}$,
where $h(x,y)$ represents the way in which the presence of some subsystem in
the state $y\in \mathbb{S}^{1}$ perturbs a certain subsystems in the state $%
x\in \mathbb{S}^{1}$; we define \ $\Phi _{\delta ,\mathbf{x}(t)}$ as 
\begin{equation}
\Phi _{\delta ,\mathbf{x}(t)}(x)=x+\pi _{\mathbb{S}^{1}}(\delta
\int_{M}h(x,x_{j}(t))~dp(j))  \label{diffeo}
\end{equation}%
($j\rightarrow h(x,x_{j}(t))$ can be viewed as a function $:M\rightarrow 
\mathbb{R}$)\footnote{\label{not1} The set $M$ can be finite or infinite. In
the case $M=M_{n}:=\{1,...,n\}$ is finite we consider a\ finite set of
interacting systems. In this case a natural choice is to set $p$ as the
uniform distribution $p_{n}$ giving to each system the same weight $\frac{1}{%
n}$. We remark that in this case $(\ref{diffeo})$ becomes%
\begin{equation}
\Phi _{\delta ,x(t)}(x)=x+\pi _{\mathbb{S}^{1}}(\frac{\delta }{n}%
\sum_{j=1}^{n}h(x,x_{j}(t))).  \label{wnnw}
\end{equation}%
\ One approach to the definition of the dynamics of a system made of
infinitely many globally interacting maps is to start by the case of $n$
interacting maps and then considering the limit for $n\rightarrow \infty $.
The perturbation $\Phi _{\delta ,x(t)}$ induced by the interaction between
the systems is defined as a suitable \ limit of \ref{wnnw}. \ This might
raise some technical problem in selecting states and the assumptions for
which the limit converge. With our approach we might also consider an
infinite space of interacting systems as \ a limit of a finite interacting
family. In this case it is sufficient to see $(M,p)$ as a suitable limit of $%
(M_{n},p_{n})$.}. Consider the function $I_{\mathbf{x},t}:M\rightarrow 
\mathbb{S}^{1}$ defined by 
\begin{equation*}
I_{\mathbf{x},t}(i):=x_{i}(t).
\end{equation*}

We remark that with these definitions, for all $t\in \mathbb{N}$, $I_{%
\mathbf{x},t}$ is also measurable. We say that the global state $\mathbf{x}%
(t)$ of the system is represented by a probability measure $\mu _{\mathbf{x}%
(t)}\in PM(\mathbb{S}^{1})$ if 
\begin{equation*}
\mu _{\mathbf{x}(t)}=[I_{\mathbf{x},t}]_{\ast }(p)
\end{equation*}%
(the pushforward of $p$ by the function $I_{\mathbf{x},t}$). \ Now we see
how the measures representing given initial conditions evolve with the
dynamics.

\begin{lemma}
Let us consider the system $(\mathcal{X},\mathcal{T})$ defined above. \ Let $%
\mu \in PM(\mathbb{S}^{1}),$ \ let us consider 
\begin{equation*}
\Phi _{\delta ,\mu }(x):=x+\pi _{\mathbb{S}^{1}}(\delta \int_{\mathbb{S}%
^{1}}h(x,y)~d\mu (y)).
\end{equation*}%
Suppose the initial condition of the system $\mathbf{x}(0)$\ is represented
by a measure $\mu _{\mathbf{x}(0)},$ then $\mathbf{x}(1)=\mathcal{T}(\mathbf{%
x}(0))$ is represented by \ 
\begin{equation*}
\mu _{\mathbf{x}(1)}=L_{\Phi _{\delta ,\mu _{\mathbf{x}(0)}}\circ T}(\mu _{%
\mathbf{x}(0)}).
\end{equation*}
\end{lemma}

\begin{proof}
Since two probability measures are identical if they act in the same way on
continuous functions, we prove that for all continuous $g:\mathbb{S}%
^{1}\rightarrow \mathbb{R},$ we have $\int g~d\mu _{\mathbf{x}(1)}=\int
g~dL_{\Phi _{\delta ,\mu _{\mathbf{x}(0)}}\circ T}(\mu _{\mathbf{x}(0)}).$
By applying several times the change of variable formula, we have\ 
\begin{eqnarray*}
\int_{\mathbb{S}^{1}}g~d\mu _{\mathbf{x}(1)} &=&\int_{M}g(x_{j}(1))~dp(j) \\
&=&\int_{M}g(\Phi _{\delta ,\mathbf{x}(0)}\circ T(x_{j}(0)))~dp(j) \\
&=&\int_{\mathbb{S}^{1}}g\circ \Phi _{\delta ,\mathbf{x}(0)}\circ T(y)~d\mu
_{\mathbf{x}(0)}(y) \\
&=&\int_{\mathbb{S}^{1}}g~dL_{\Phi _{\delta ,\mathbf{x}(0)}\circ T}(\mu _{%
\mathbf{x}(0)}).
\end{eqnarray*}%
But since $\mathbf{x}(0)$\ is represented by $\mu _{\mathbf{x}(0)}$%
\begin{eqnarray*}
\Phi _{\delta ,\mathbf{x}(0)}(x) &=&x+\pi _{\mathbb{S}^{1}}(\delta
\int_{M}h(x,x_{j}(0))~dp(j)) \\
&=&\Phi _{\delta ,\mu _{\mathbf{x}(0)}}(x)
\end{eqnarray*}%
leading to the statement.
\end{proof}

Hence the measure representing the current state of the system fully
determines the measure which represents the next state of the system,
defining a function between measures%
\begin{equation*}
\mu \rightarrow L_{\Phi _{\delta ,\mu }\circ T}(\mu ).
\end{equation*}%
\ This function is an example of what in the following section we will
consider as a self-consistent transfer operator. \ In the case of coupled
systems $(\mathcal{X},\mathcal{T})$ described above, to describe the
evolution of a certain probability measure representing the global state of
the system we hence apply at each time a transfer operator from a family of
the kind 
\begin{equation*}
L_{\delta ,\mu }:=L_{\Phi _{\delta ,\mu }\circ T}=L_{\Phi _{\delta ,\mu
}}L_{T}.
\end{equation*}%
Each operator $L_{\delta ,\mu }$ can be seen as the transfer operator
associated with the dynamics of a given node of the network of coupled
systems, given that the distribution of the states of the other nodes in the
network is represented by the measure $\mu .$

We remark that the extended system $(\mathcal{X},\mathcal{T})$ above
described can be identified by the choice of the phase space $\mathbb{S}^{1}$%
, the local dynamics $T$, the strength of coupling $\delta $ and the
coupling function $\ h.$ Hence it can be identified as the quadruple $(%
\mathbb{S}^{1},T,\delta ,h)$.

\section{Self-consistent operators, the existence of the invariant measure. 
\label{exsec}}

\noindent \textbf{General standing assumptions and notations. }Motivated by
the class of examples described in the previous section, given a compact
metric space $X$ we consider a family of Markov operators $L_{\delta ,\mu
}:SM(X)\rightarrow SM(X)$ depending on a probability measure $\mu \in PM(X)$
\ and $\delta \geq 0$. In our statements, we will apply the operators $%
L_{\delta ,\mu }$ to different strong and weak spaces of measures which are
subspaces of $SM(X).$ We now introduce the notations and the basic
assumptions to formalize this. \ Let $(B_{w},||~||_{w})$ be a normed vector
subspace \ of $\ SM(X)$. In the paper we will suppose that the weak norm $%
||~||_{w}$ is strong enough so that the function $\mu \rightarrow \mu (X)$
is continuous as a function $:B_{w}\rightarrow \mathbb{R}$ and that \ $||\mu
_{n}-\mu ||_{w}\rightarrow 0$ for a sequence of positive measures $\mu _{n}$
implies that $\mu $ is positive. Let $P_{w}:=B_{w}\cap PM(X)$ the set of
probability measures in $B_{w}.$ We will suppose that $P_{w}$ with the
metric induced by $||~||_{w}$ is a complete metric space.

A \emph{self-consistent transfer operator} in our context will be the given
of a family of Markov linear operators such that $L_{\delta ,\mu }:$ $%
B_{w}\rightarrow B_{w}$ \ for all $\mu \in P_{w}$, some $\delta \geq 0$ and
the dynamical system $(P_{w},\mathcal{L}_{\delta })$ where $\mathcal{L}%
_{\delta }:P_{w}\rightarrow P_{w}$ is defined by%
\begin{equation}
\mathcal{L}_{\delta }(\mu ):=L_{\delta ,\mu }(\mu ).  \label{cupled}
\end{equation}%
In the notation $\mathcal{L}_{\delta }$ \ we emphasize the dependence on $%
\delta $ as in the following we will be interested in the behavior of these
operators \ for certain sets of values of $\delta $ or in the limit $\delta
\rightarrow 0.$ We also point out that here and in the following we will use
the calligraphic notation $\mathcal{L}$ to denote some operator which is not
necessarily linear and the notation $L$ to denote linear operators.

In the following we will apply linear operators on spaces with different
topologies. If $A,B$ are two normed vector spaces and $L:A\rightarrow B$ is
a linear operator we denote the mixed norm $\Vert L\Vert _{A\rightarrow B}$
as 
\begin{equation*}
\Vert L\Vert _{A\rightarrow B}:=\sup_{f\in A,\Vert f\Vert _{A}\leq 1}\Vert
Lf\Vert _{B}.
\end{equation*}

\begin{remark}
In the case where $L_{\delta ,\mu }$ is the transfer operator associated
with a map $T_{\delta ,\mu }:X\rightarrow X$ \ to this dynamical system one
can also associate the skew product dynamical system $(A\times X,F)$ on $%
A\times X$ where $F:A\times X\rightarrow A\times X$ \ is defined by%
\begin{equation*}
F(\mu ,x)=(\mathcal{L}_{\delta }(\mu ),T_{\delta ,\mu }(x))
\end{equation*}%
(see also \cite{Bla}). One can remark that in the case $\mu $ is a fixed
point for $\mathcal{L}_{\delta }$ the associated dynamics will be nontrivial
only on the second coordinate, where \ $T_{\delta ,\mu }$ represents a map
for which $\mu $ is an invariant measure. \ Hence by the classical ergodic
theory results, finding the fixed points of $\mathcal{L}_{\delta }$ gives
important information on the statistical behavior of the second coordinate
of the system $F$.
\end{remark}

We will hence be interested in the dynamics $\mathcal{L}_{\delta }$
considered on a space of measures, and on the properties of its fixed
points. In particular we will be interested in the attractiveness of these
fixed points (which will determine the convergence to equilibrium of the
global system) and to the stability \ or response of these fixed points with
respect to perturbations of the global system.

\noindent \textbf{Standing assumptions 1. }In this section we will use the
following standing assumptions and notations.

Let $B_{w}$ as above, and let $B_{s}$ be a normed vector subspace $%
(B_{s},||~||_{s})\subseteq $ $(B_{w},||~||_{w})$. Suppose $||~||_{s}\geq
||~||_{w}$. We also denote by $P_{s}:=P_{w}\cap B_{s}$ the set of
probability measures in $B_{s}$. We suppose $P_{s}\neq \emptyset $. We will
also suppose that there is $M\geq 0$ such that as $\mu $ varies in $P_{w}$
the family $L_{\delta ,\mu }$ is such that \ $||L_{\delta ,\mu
}||_{B_{w}\rightarrow B_{w}}\leq M$ and $||L_{\delta ,\mu
}||_{B_{s}\rightarrow B_{s}}\leq M$.

We now prove general statements regarding the existence and uniqueness of
regular (and then physically meaningful) invariant measures for
self-consistent transfer operators. We remark that since our transfer
operators are not linear, the normalization of the measure to a probability
one is important in this context. In the case in which we put no
restrictions on the size of the parameter $\delta $ representing the
nonlinearity strength, by a topological reasoning we prove a general result\
on the existence of invariant probability measures (Theorem Theorem \ref%
{existence1}). We then suppose that the parameter $\delta $ is below a
certain threshold, and in this weak coupling regime we also prove some
unique existence result (see Theorem \ref{existence}). We remark that in the
weak coupling regime similar results have been proved in several cases of
extended systems (see e.g. \cite{Bl11}, \cite{KL05}, \cite{JP}), also
showing the uniqueness of the invariant measure in a certain class. It is
known on the other hand that as the coupling strength grows, phase
transitions phenomena can occur, leading to the presence of multiple
invariant measures (see \cite{BKLZ}, and \cite{Se2} for a case not arising
from coupled maps in which the uniqueness of absolutely continuous invariant
measures is lost for all $\delta >0$).

\begin{theorem}
\label{existence1} Suppose that there exists $\pi _{n}:B_{w}\rightarrow
B_{s} $, a linear projection of finite rank $n$ which is a Markov operator
with the following properties: there is $M_{0}\geq 0$ and a decreasing
sequence $a(n)\rightarrow 0$ such that for all $n\geq 0$%
\begin{eqnarray}
||\pi _{n}||_{B_{w}\rightarrow B_{w}} &<&M_{0},  \label{chio1} \\
||\pi _{n}||_{B_{s}\rightarrow B_{s}} &<&M_{0}  \notag
\end{eqnarray}%
and%
\begin{equation}
||\pi _{n}f-f||_{w}\leq a(n)||f||_{s}.  \label{chio}
\end{equation}

Let us suppose that $\pi _{n}(P_{w})\subseteq P_{s}$ and $\pi _{n}(P_{w})$
is bounded in $B_{s}$. Let us fix $\delta \geq 0$ and suppose furthermore
that:

\begin{description}
\item[Exi1] there is $M_{1}\geq 0$ such that $\forall \mu _{1}\in P_{w}$ and 
$f$ $\in P_{w}$ which is a fixed point of $L_{\delta ,\mu _{1}}$ it holds%
\begin{equation*}
||f||_{s}\leq M_{1};
\end{equation*}

\item[Exi1.b] $\forall \mu _{1}\in P_{w},~n\in \mathbb{N}$ and for every $f$ 
$\in P_{w}$ which is a fixed point for the finite rank approximation $\pi
_{n}L_{\delta ,\pi _{n}\mu _{1}}\pi _{n}$ of $\ L_{\delta ,\mu _{1}}$ it
holds%
\begin{equation*}
||f||_{s}\leq M_{1};
\end{equation*}

\item[Exi2] there is $K_{1}\geq 0$ such that $\forall \mu _{1},\mu _{2}\in
P_{w}$%
\begin{equation*}
||L_{\delta ,\mu _{1}}-L_{\delta ,\mu _{2}}||_{B_{s}\rightarrow B_{w}}\leq
\delta K_{1}||\mu _{1}-\mu _{2}||_{w}.
\end{equation*}
\end{description}

Then there is $\mu \in P_{s}$ such that 
\begin{equation*}
\mathcal{L}_{\delta }\mu =\mu .
\end{equation*}
and 
\begin{equation}
||\mu ||_{s}\leq M_{1}.  \label{boundnorm}
\end{equation}
\end{theorem}

To understand the assumptions made we suggest to think of $B_{w}$ as a weak
space, for example $L^{1}$ and of $B_{s}$ as a stronger space in which
regular fixed points of the linear transfer operators $L_{\delta ,\mu }$ are
contained, for example, in the case of transfer operators associated with
expanding maps, one can think of $B_{s}$ as some Sobolev space. The
projection $\pi _{n}$ allows to reduce the problem to a finite dimensional
one and find fixed points of the finite dimensional reduced operators by the
Brouwer fixed point theorem. In concrete examples $\pi _{n}$ could be a
finite dimensional discretization, as the Ulam discretization or similar.
The assumptions $(Exi1),~(Exi1.b)$ tells that the invariant measures of the
original and discretized operators are unifornmly regular, and can be
verified in concrete examples by showing that these operators satisfy a
common Lasota-Yorke inequality. The assumption $(Exi2)$ \ in some sense says
that the family of operators $L_{\delta ,\mu }$ depends on $\mu $ in a
Lipschitz way, considering a (weak) mixed norm topology. The assumptions
made are then quite natural for a family of transfer operators depending on
a parameter. An interesting corollary of Theorem \ref{existence1} is
Proposition \ref{KB}, establishing a general statement for the existence of
an invariant probability meaure for general continuous maen field coupled
maps, even outside the weak cupling regime.

\begin{proof}[Proof of Theorem \protect\ref{existence1}]
Without loss of generality we can suppose that each operator $L_{\delta ,\mu
}$ is such that $||L_{\delta ,\mu }||_{B_{s}\rightarrow B_{s}}\leq M_{0},$ $%
||L_{\delta ,\mu }||_{B_{w}\rightarrow B_{w}}\leq M_{0}$. \ First we prove
that under the assumptions of the theorem, given a sequence $\mu _{i}$ of
probability measures which is bounded in $B_{s}$ there is $\mu \in P_{w}$
and converging subsequence $\mu _{i_{k}}\rightarrow \mu $, converging in the
weak topology. Indeed let $\mu _{i}$ be such a sequence, with $||\mu
_{i}||_{s}\leq M_{2}$, let us consider $\nu _{n,i}:=\pi _{n}\mu _{i}$. \
Since $\pi _{n}$ is Markov this is a sequence of probability measures. By $(%
\ref{chio1})$\ this sequence is bounded in $B_{w}$ and belongs to the finite
dimensional space $\pi _{n}B_{w}$, then it has a converging subsequence \ $%
\nu _{n,i_{n,k}}\rightarrow \nu _{n}$ where we denote by $i_{n,k}$ a
sequence of indices for which we have this convergence. We remark that this
indices can depend on $n$. \ Since $P_{w}$ is complete for the weak norm
(see Standing assumptions 1) we also have that $\nu _{n}\in P_{w}.$ By $(\ref%
{chio})$, for all $n$ and $k$ we have \ $||\nu _{n,i_{n,k}}-\mu
_{i_{n,k}}||_{w}\leq a(n)M_{2}.$ Without loss of generality we can suppose
that $i_{n,k}$ is a subsequence of $i_{n-1,k}$ \ (suppose we selected the
sequence of indices $i_{n-1,k}$, then we can select the subsequence $\nu
_{n,i_{n,k}}\rightarrow \nu _{n}$ only from the indices belonging to $%
i_{n-1,k}$ since $\nu _{n-1,i_{n-1,k}}$ is also a bounded sequence, and so
on for all $n$ by induction). \ In this case, for $m\geq n$ we have 
\begin{eqnarray*}
||\nu _{m,i_{m,k}}-\nu _{n,i_{m,k}}||_{w} &\leq &||\nu _{m,i_{m,k}}-\mu
_{i_{m,k}}||_{w}+||\nu _{n,i_{m,k}}-\mu _{i_{m,k}}||_{w} \\
&\leq &2a(n)M_{2}.
\end{eqnarray*}%
Since this is true for all $k,$ by taking the limits it holds that $||\nu
_{n}-\nu _{m}||_{w}\leq 2a(n)M_{2}$ and hence $\nu _{n}$ is a Cauchy
sequence of probability measures in $P_{w}$. By the completeness of $P_{w}$,
the sequence $\nu _{n}$ will then converge to some $\nu \in P_{w}$. \ We
also have that $\nu _{k,i_{k,k}}\rightarrow \nu $ in $B_{w}$ and since \ $%
||\nu _{k,i_{k,k}}-\mu _{i_{k,k}}||_{w}\leq a(k)M_{2}$ we also have $\mu
_{i_{k,k}}\rightarrow \nu $ in $B_{w}$, finding a converging subsequence as
claimed.

Let us consider a finite rank approximation of $\mathcal{L}_{\delta },$
defined by%
\begin{equation*}
\mathcal{L}_{\delta ,n}(\mu ):=\pi _{n}L_{\delta ,\pi _{n}\mu }\pi _{n}(\mu
).
\end{equation*}

We now prove that $\mathcal{L}_{\delta ,n}$ is a continuous function $%
P_{w}\rightarrow P_{w}$, indeed let $\mu \in P_{w},\nu \in B_{w}$ such that $%
\mu +\nu \in P_{w},$ we have%
\begin{eqnarray*}
||\mathcal{L}_{\delta ,n}(\mu +\nu )-\mathcal{L}_{\delta ,n}(\mu )||_{w}
&\leq &||\pi _{n}L_{\delta ,\pi _{n}(\mu +\nu )}\pi _{n}(\mu +\nu )-\pi
_{n}L_{\delta ,\pi _{n}\mu }\pi _{n}(\mu )||_{w} \\
&\leq &||\pi _{n}||_{B_{w}\rightarrow B_{w}}[||L_{\delta ,\pi _{n}(\mu +\nu
)}\pi _{n}(\mu +\nu )-L_{\delta ,\pi _{n}\mu }\pi _{n}(\mu +\nu )||_{w} \\
&&+||L_{\delta ,\pi _{n}\mu }\pi _{n}(\mu +\nu )-L_{\delta ,\pi _{n}\mu }\pi
_{n}(\mu )||_{w}] \\
&\leq &||\pi _{n}||_{B_{w}\rightarrow B_{w}}[||L_{\delta ,\pi _{n}(\mu +\nu
)}\pi _{n}(\mu )+L_{\delta ,\pi _{n}(\mu +\nu )}\pi _{n}(\nu ) \\
&&-L_{\delta ,\pi _{n}\mu }\pi _{n}(\mu )-L_{\delta ,\pi _{n}\mu }\pi
_{n}(\nu )||_{w} \\
&&+||L_{\delta ,\pi _{n}\mu }\pi _{n}(\mu +\nu )-L_{\delta ,\pi _{n}\mu }\pi
_{n}(\mu )||_{w}]
\end{eqnarray*}

and%
\begin{equation*}
||L_{\delta ,\pi _{n}\mu }\pi _{n}(\mu +\nu )-L_{\delta ,\pi _{n}\mu }\pi
_{n}(\mu )||_{w}\leq M^{2}||\nu ||_{w}
\end{equation*}%
while using $(Exi2)$%
\begin{eqnarray*}
||L_{\delta ,\pi _{n}(\mu +\nu )}\pi _{n}(\mu )-L_{\delta ,\pi _{n}\mu }\pi
_{n}(\mu )||_{w} &\leq &\delta K_{1}||\pi _{n}\nu ||_{w}||\pi _{n}\mu ||_{s},
\\
||L_{\delta ,\pi _{n}(\mu +\nu )}\pi _{n}(\nu )-L_{\delta ,\pi _{n}\mu }\pi
_{n}(\nu )||_{w} &\leq &2M^{2}||\nu ||_{w}
\end{eqnarray*}%
hence%
\begin{equation*}
||\mathcal{L}_{\delta ,n}(\mu +\nu )-\mathcal{L}_{\delta ,n}(\mu )||_{w}\leq
M[M^{2}+\delta MK_{1}||\pi _{n}\mu ||_{s}+2M^{2}]||\nu ||_{w}.
\end{equation*}

By assumption $\pi _{n}P_{w}$ is bounded in $B_{s}$, this shows that $||\pi
_{n}\mu ||_{s}$ is uniformly bounded as $\mu $ ranges in $P_{w}$ and then $%
\mathcal{L}_{\delta ,n}(\mu )$ is Lipschitz continuous $P_{w}\rightarrow
P_{w}.$ Note that since $P_{w}$ is a convex set, $\pi _{n}P_{w}$ is a finite
dimensional convex and bounded set. Now let us see that this is also a
closed set in $B_{w}$. We will deduce that it is compact. Suppose $p_{i}\in
\pi _{n}P_{w}\subseteq P_{w}$ is a Cauchy sequence for the $B_{w}$ norm.
Since $P_{w}$ is complete this will converge to a point $w$ of $P_{w}$. But $%
\pi _{n}B_{w},$ being a finite dimensional vector space, is a closed space,
then $w\in \pi _{n}B_{w}$. Suppose $w=\pi _{n}(u)$ with $u\in B_{w},$\ since 
$w=\pi _{n}(\pi _{n}u)=\pi _{n}(w)$ and $w\in P_{w}$ then $w\in \pi
_{n}P_{w}.$ By this $\pi _{n}P_{w}$ \ is a closed subspace of $B_{w}.$ Since 
$\pi _{n}P_{w}$ is a bounded, convex and closed \ subset of a finite
dimensional space, then it is homeomorphic to a closed disc (see e.g. \cite%
{Fl}, Corollary 1.1.1). We have that $\mathcal{L}_{\delta ,n}$ is continuous
on $\pi _{n}P_{w}$ and $\mathcal{L}_{\delta ,n}(\pi _{n}P_{w})\subseteq \pi
_{n}P_{w}$. Then by the Brouwer fixed point theorem there is $\mu _{n}\in
\pi _{n}P_{w}$ such that%
\begin{equation*}
\mathcal{L}_{\delta ,n}(\mu _{n})=\mu _{n}.
\end{equation*}

This means that $\pi _{n}L_{\delta ,\pi _{n}\mu _{n}}\pi _{n}(\mu _{n})=\mu
_{n}$ and then by $Exi1.b$ we have that for all $n\in \mathbb{N}$, $||\mu
_{n}||_{s}\leq M_{1}.$ As we proved above $\mu _{n}$ has then a converging
subsequence $\mu _{n_{k}}\rightarrow \hat{\mu}$ \ in the weak norm to some
element $\hat{\mu}\in P_{w}.$

Now let us prove that 
\begin{equation*}
\mathcal{L}_{\delta }(\hat{\mu})=\hat{\mu}.
\end{equation*}

In fact we have for all $k\geq 0$%
\begin{equation}
\mathcal{L}_{\delta }\hat{\mu}=L_{\delta ,\hat{\mu}}(\hat{\mu}-\mu
_{n_{k}})+L_{\delta ,\hat{\mu}}(\mu _{n_{k}}).
\end{equation}

Since $||\hat{\mu}-\mu _{n_{k}}||_{w}\rightarrow 0,$ and the operator $%
L_{\delta ,\hat{\mu}}$ is bounded then%
\begin{equation*}
L_{\delta ,\hat{\mu}}(\mu _{n_{k}})\rightarrow \mathcal{L}_{\delta }\hat{\mu}
\end{equation*}%
in the weak norm. By $Exi2$ 
\begin{equation*}
||L_{\delta ,\hat{\mu}}(\mu _{n_{k}})-L_{\delta ,\mu _{n_{k}}}(\mu
_{n_{k}})||_{w}\leq \delta K_{1}||\hat{\mu}-\mu _{n_{k}}||_{w}||\mu
_{n_{k}}||_{s}
\end{equation*}%
which by $Exi1.b$ becomes%
\begin{equation*}
||L_{\delta ,\hat{\mu}}(\mu _{n_{k}})-L_{\delta ,\mu _{n_{k}}}(\mu
_{n_{k}})||_{w}\leq \delta K_{1}M_{1}||\hat{\mu}-\mu _{n_{k}}||_{w}
\end{equation*}%
and then 
\begin{equation*}
L_{\delta ,\mu _{n_{k}}}(\mu _{n_{k}})\rightarrow \mathcal{L}_{\delta }\hat{%
\mu}
\end{equation*}%
in the weak norm. \ Since $\mu _{n_{k}}=\pi _{n_{k}}\mu _{n_{k}}$ we also
have that 
\begin{equation*}
L_{\delta ,\mu _{n_{k}}}(\mu _{n_{k}})-\mathcal{L}_{\delta ,n_{k}}\mu
_{n_{k}}=L_{\delta ,\mu _{n_{k}}}(\mu _{n_{k}})-\pi _{n_{k}}L_{\delta ,\mu
_{n_{k}}}(\mu _{n_{k}})
\end{equation*}%
and then by $(\ref{chio})$%
\begin{eqnarray*}
||L_{\delta ,\mu _{n_{k}}}(\mu _{n_{k}})-\mathcal{L}_{\delta ,n_{k}}\mu
_{n_{k}}||_{w} &\leq &a(n_{k})||L_{\delta ,\mu _{n_{k}}}\mu _{n_{k}}||_{s} \\
&\leq &a(n_{k})MM_{1}\rightarrow 0.
\end{eqnarray*}

We then proved that%
\begin{equation*}
\mathcal{L}_{\delta ,n_{k}}\mu _{n_{k}}\rightarrow \mathcal{L}_{\delta }\hat{%
\mu}.
\end{equation*}

Since%
\begin{equation*}
\mathcal{L}_{\delta ,n_{k}}\mu _{n_{k}}=\mu _{n_{k}}
\end{equation*}%
then we get%
\begin{equation*}
\hat{\mu}=\lim_{k\rightarrow \infty }\mu _{n_{k}}=\mathcal{L}_{\delta }(\hat{%
\mu})
\end{equation*}%
proving the invariance of $\hat{\mu}.$

Now we are only left to prove that \ $||\hat{\mu}||_{s}\leq M_{1}$. Since $%
L_{\delta ,\hat{\mu}}(\hat{\mu})=\hat{\mu}$ this directly follows from $%
(Exi1)$.
\end{proof}

Theorem \ref{existence1} gives general, sufficient conditions for the
existence of the invariant probability measure of a self-consistent
operator, but it is hard to apply it constructively to approximate the
invariant measure. Furthermore it does not give information about the
uniqueness.

Now we prove a kind of constructive existence and uniqueness result in the
case of weak coupling. \ Before stating the result, as a general remark on
the uniqueness of the invariant probability measure we show that when $%
\delta $ is small and the operators $L_{\delta ,\mu }$ are statistically
stable in some sense, different invariant probability measures in $B_{w}$ of
the associated $\mathcal{L}_{\delta }$ must be near each other. Indeed,
suppose that each operator of the family $L_{\delta ,\mu }$ with $\delta
\geq 0$ and $\mu \in P_{w}$ has a unique fixed probability measure in $B_{w}$
which we denote by $f_{\mu }\in P_{w}$ and suppose there is $F:\mathbb{%
R\rightarrow R}$ such that $\forall \mu _{1},\mu _{2}\in P_{w}$%
\begin{equation*}
||f_{\mu _{1}}-f_{\mu _{2}}||_{w}\leq F(\delta ).
\end{equation*}%
\ If $\mu $, $\nu \in P_{w}$ are invariant measures for $\mathcal{L}_{\delta
},$ this implies that $\mu =f_{\mu }$ and $\nu =f_{\nu }$. Then we have%
\begin{equation*}
||\mu -\nu ||_{w}=||f_{\nu }-f_{\mu }||_{w}\leq F(\delta ).
\end{equation*}%
In the case $\lim_{\delta \rightarrow 0}F(\delta )=0$ we see that when $%
\delta $ is small different invariant measures of $\mathcal{L}_{\delta }$
must be near each other. \ In the following statement we then suppose a
strong stability property (see $(Exi3)$) for the invariant measures of the
operators $L_{\delta ,\mu }$ as $\mu $ vary.

\begin{theorem}
\label{existence} Suppose there is $\overline{\delta }\geq 0$ such that for
all \ $0\leq \delta <\overline{\delta }$ the family $L_{\delta ,\mu }$
satisfies $(Exi1)$ and $(Exi2)$ uniformly (with the same constants for each
such $\delta $). Suppose that $P_{w}$ contains some probability measure $\mu 
$ with $||\mu ||_{w}\leq M_{1}$ (where $M_{1}$ is the constant coming from $%
(Exi1)$). Suppose that for all \ $0\leq \delta <\overline{\delta }$ \ and $%
\mu \in P_{w}$ with $||\mu ||_{w}\leq M_{1}$, $L_{\delta ,\mu }$ has a
unique fixed probability measure in $P_{w}$ which we denote by $f_{\mu }.$
Suppose furthermore that the family $L_{\delta ,\mu }$ satisfies the
following:

\begin{description}
\item[Exi3] there is $K_{2}\geq 1$ such that $\forall \mu _{1},\mu _{2}\in
P_{w}$ with $\max (||\mu _{1}||_{w},||\mu _{2}||_{w})\leq M_{1}$ 
\begin{equation*}
||f_{\mu _{1}}-f_{\mu _{2}}||_{w}\leq \delta K_{2}||\mu _{1}-\mu _{2}||_{w}.
\end{equation*}
\end{description}

Then for all $0\leq \delta \leq \min (\overline{\delta },\frac{1}{K_{2}})$,
there is a unique $\mu \in P_{w}$ \ such that%
\begin{equation*}
\mathcal{L}_{\delta }(\mu )=\mu .
\end{equation*}%
Furthermore $\mu =\lim_{k\rightarrow \infty }\mu _{k}$ where $\mu _{k}$ is
any sequence defined inductively in the following way: let\ $\mu _{0}$ be
some probability measure in $P_{w}$ with $||\mu _{0}||_{w}\leq M_{1}$, then $%
\mu _{1}$ is the fixed probability measure of $L_{\delta ,\mu _{0}},$ $\mu
_{i}$ in $\ P_{w}$; $\mu _{i}$ is the fixed probability measure of $%
L_{\delta ,\mu _{i-1}}$ in $\ P_{w}$ and so on.
\end{theorem}

While the assumptions $(Exi1)$, $(Exi2)$ and the uniqueness of the fixed
probability measure in $P_{w}$ for the family $L_{\delta ,\mu }$ can be
easily verified for a large class of examples, including families of
transfer operators coming from piecewise expanding maps, the assumption $%
(Exi3)$ imposes some stronger requirements on the kind of systems we can
consider when applying this statement.

The assumption $(Exi3)$ correspond to a Lipschitz quantitative stability for
the fixed points of the operators in the family $L_{\delta ,\mu }$ when the
operators are perturbed by changing $\mu $. This is a strong assumption
which is however satisfied for many interesting systems, as expanding and
uniformly hyperbolic or many random ones, but it is not satisfied for other
systems like piecewise expanding maps for perturbations changing their
turning points. We remark that indeed self-consistent transfer operators
arising from piecewise expanding maps show a complicated behavior from the
point of view of the uniqueness of the invariant measure (\cite{Se2}).

\begin{proof}[Proof of Theorem \protect\ref{existence}]
Let us consider $\overline{\delta }$ such that%
\begin{equation}
0<\overline{\delta }<K_{2}^{-1}.  \label{deltachose1}
\end{equation}%
\ Let us consider some $0<\delta \leq \overline{\delta }.$ Let $f_{0}\in
P_{w}$ with $||f_{0}||_{w}\leq M_{1}$. Let $f_{1}\in P_{s}$ be the fixed
probability measure of $L_{\delta ,f_{0}}$ in $B_{w}$, again $%
||f_{1}||_{w}\leq M_{1}$. Now, $L_{\delta ,f_{1}}$ has a fixed probability
measure which we will denote by $f_{2}.$ We also have $||f_{2}||_{w}\leq
M_{1}$. By $(Exi3)$%
\begin{equation*}
||f_{1}-f_{2}||_{w}\leq \delta K_{2}||f_{0}-f_{1}||_{w}.
\end{equation*}

Now let us consider the linear operator $L_{\delta ,f_{2}}$, this operator
has a fixed probability measure $f_{3}\in B_{s}$ with \ $||f_{3}||_{w}\leq
M_{1}$. We get 
\begin{equation*}
||f_{3}-f_{2}||_{w}\leq K_{2}\delta ||f_{2}-f_{1}||_{w}\leq (K_{2}\delta
)^{2}||f_{0}-f_{1}||_{w}.
\end{equation*}

Continuing as before, this will lead to a new fixed probability measure $%
f_{4}$ with $||f_{4}-f_{3}||_{w}\leq (K_{2}\delta )^{3}||f_{0}-f_{1}||_{w}$
and so on, defining a sequence $f_{k}$ with $||f_{k}||_{w}\leq M_{1}$ and $%
||f_{k}-f_{k-1}||_{w}\leq (K_{2}\delta )^{k-1}||f_{0}-f_{1}||_{w}$. Since $%
(K_{2}\delta )^{k}$ is summable, $f_{k}$ is a Cauchy sequence in $P_{w}.$

Since $P_{w}$ is complete this sequence has a limit. Let $%
f:=\lim_{k\rightarrow \infty }f_{k}\in P_{w}.$ By $(Exi1),$ $f_{k}$ is also
uniformly bounded in $B_{s}.$ Now we can prove that $\mathcal{L}_{\delta
}(f)=L_{\delta ,f}(f)=f$. Indeed%
\begin{eqnarray*}
L_{\delta ,f}(f) &=&L_{\delta ,f}(\lim_{k\rightarrow \infty }f_{k}) \\
&=&\lim_{k\rightarrow \infty }L_{\delta ,f}(f_{k})
\end{eqnarray*}

because of the continuity of $L_{\delta ,f}$ in the weak norm. Furthermore%
\begin{eqnarray*}
\lim_{k\rightarrow \infty }L_{\delta ,f}(f_{k}) &=&\lim_{k\rightarrow \infty
}L_{\delta ,f}(f_{k})-L_{\delta ,f_{k-1}}(f_{k})+L_{\delta ,f_{k-1}}(f_{k})
\\
&=&\lim_{k\rightarrow \infty }L_{\delta ,f}(f_{k})-L_{\delta
,f_{k-1}}(f_{k})+f_{k}
\end{eqnarray*}

because $L_{\delta ,f_{k-1}}(f_{k})=f_{k}.$ However, by $(Exi2)$ there is $%
K_{1}\geq 0$ such that $||L_{\delta ,\mu _{1}}-L_{\delta ,\mu
_{2}}||_{B_{s}\rightarrow B_{w}}\leq K_{1}\delta ||\mu _{1}-\mu _{2}||_{w}$
and using this together with $(Exi1)$ we get%
\begin{eqnarray*}
||L_{\delta ,f}(f_{k})-L_{\delta ,f_{k-1}}(f_{k})||_{w} &\leq &K_{1}\delta
||f-f_{k-1}||_{w}||f_{k}||_{s} \\
&\leq &K_{1}\delta M_{1}||f-f_{k-1}||_{w}\underset{k\rightarrow \infty }{%
\rightarrow }0.
\end{eqnarray*}

Then in the $B_{w}$ topology 
\begin{equation*}
L_{\delta ,f}(f)=\lim_{k\rightarrow \infty }L_{\delta
,f}(f_{k})=\lim_{k\rightarrow \infty }f_{k}=f.
\end{equation*}

Regarding the uniqueness, suppose \ $\mu _{1},\mu _{2}\in P_{w}$ are
invariant \ for $\mathcal{L}_{\delta }$. Then $L_{\delta ,\mu _{1}}(\mu
_{1})=\mu _{1}$ and $L_{\delta ,\mu _{2}}(\mu _{2})=\mu _{2}.$ By $(Exi1)$
we have $\max (||\mu _{1}||_{w},||\mu _{2}||_{w})\leq M_{1}$ and then\ \ by $%
(Exi3)$ \ we have $||\mu _{1}-\mu _{2}||_{w}\leq \delta K_{2}||\mu _{1}-\mu
_{2}||_{w}$, implying $||\mu _{1}-\mu _{2}||_{w}=0$ \ because $K_{2}\delta
<1 $.
\end{proof}

\begin{remark}
The way the fixed point $f$ is found in the previous proof is constructive,
provided we have a mean of finding the invariant measures of the various
operators $L_{\delta ,f_{k}}$ (which is possible in many interesting cases
by some suitable finite element reduction). In this case $f$ can be
approximated by the sequence $f_{k}\rightarrow f$ and the proof also
provides an explicit way to estimate the convergence rate of this sequence,
which is exponential.
\end{remark}

\section{Self-consistent operators, exponential convergence to equilibrium 
\label{sec1}}

Theorems \ref{existence1} and \ref{existence} give information about the
existence of fixed probability measures for the self-consistent operators
but gives no information on whether they are attractive fixed points. In
this section we address this question, giving general sufficient conditions
for this to hold. In the case where the invariant probability measure is
attractive we have that the associated system has convergence to equilibrium
in some sense, since iterates of some initial probability measure will
converge to the invariant one. It is important to estimate the speed of this
convergence. In the case of weak coupling we will show a set of general
conditions implying exponential speed of convergence to equilibrium for
self-consistent transfer operators.

\noindent \textbf{Standing assumptions 2. }In this section we will consider
a setup similar to the one in the previous section, with strong and weak
spaces $B_{s}$ and $B_{w}$ and a family of Markov bounded operators $%
L_{\delta ,\mu }$ satistyfing the General Standing assumptions and the
Standing assumptions 1 stated at beginning of Section \ref{exsec}. We will
also consider a stronger space $(B_{ss},||~||_{ss})$ with norm satisfying $%
||~||_{ss}\geq ||~||_{s}$. We denote by $P_{ss}$ the set of probability
measures in $B_{ss}$. We will suppose that for all $\mu \in P_{w}$ and $%
\delta \geq 0$ the operators $L_{\delta ,\mu }:B_{ss}\rightarrow B_{ss}$ are
bounded and that $P_{w}$ is a bounded set for the $B_{w}$ norm. We will
consider furthermore the following assumptions:

\begin{itemize}
\item[$Con1$] The operators $L_{\delta ,\mu }$ satisfy a common "one step"
Lasota Yorke inequality. There are constants $\hat{\delta},B,\lambda
_{1}\geq 0$ with $\lambda _{1}<1$ such that for all $f\in B_{s},$ $\mu \in
P_{w},$ $0\leq \delta \leq \hat{\delta}$%
\begin{equation}
\left\{ 
\begin{array}{c}
||L_{\delta ,\mu }f||_{w}\leq ||f||_{w} \\ 
||L_{\delta ,\mu }f||_{s}\leq \lambda _{1}||f||_{s}+B||f||_{w}.%
\end{array}%
\right.  \label{1}
\end{equation}

\item[$Con2$] The family of operators satisfy an extended $(Exi2)$ \
property: \ there is $K\geq 1$ such that for all $f\in B_{s},$ $\mu ,\nu \in
P_{w}$, $0\leq \delta \leq \hat{\delta}$%
\begin{equation}
||(L_{\delta ,\mu }-L_{\delta ,\nu })(f)||_{B_{s}\rightarrow B_{w}}\leq
\delta K||\mu -\nu ||_{w}  \label{nn}
\end{equation}%
and $\forall f\in B_{ss},$ $\mu ,\nu \in P_{w}$ 
\begin{equation*}
||(L_{\delta ,\mu }-L_{\delta ,\nu })(f)||_{B_{ss}\rightarrow B_{s}}\leq
\delta K||\mu -\nu ||_{w}.
\end{equation*}%
We remark that by $(\ref{nn})$, when $\delta =0$ $\ \ L_{\delta ,\mu },$ $%
L_{\delta ,\nu }:B_{s}\rightarrow B_{s}$ are identical operators for all $%
\mu ,\nu \in B_{w}$. We hence denote this operator as $L_{0}$. \ We also
suppose that for all $f\in B_{s},$ $\nu \in P_{w}$, $0\leq \delta \leq \hat{%
\delta}$ 
\begin{equation}
||(L_{0}-L_{\delta ,\nu })(f)||_{B_{s}\rightarrow B_{w}}\leq \delta K||\nu
||_{w}.  \label{mm}
\end{equation}

\item[$Con3$] The operator $L_{0}:B_{s}\rightarrow B_{s}$ has convergence to
equilibrium: there exists $a_{n}\geq 0$ with $a_{n}\rightarrow 0$ such that
for all $n\in \mathbb{N}$ and $v\in V_{s}$%
\begin{equation}
||L_{0}^{n}(v)||_{w}\leq a_{n}||v||_{s}  \label{3}
\end{equation}%
where%
\begin{equation*}
V_{s}=\{\mu \in B_{s}|\mu (X)=0\}\footnote{%
We recall that since $\mu \rightarrow \mu (X)$ is continuous, $V_{s}$ is
closed. Furthermore $\forall \mu \in P_{w},$ $L_{\delta ,\mu
}(V_{s})\subseteq V_{s}.$ }.
\end{equation*}
\end{itemize}

We remark that the assumption $(Con1)$ implies that the family of operators $%
L_{\delta ,\mu }$ is uniformly bounded when acting on $B_{s}$ and on $B_{w} $
\ as $\mu $ varies in $P_{w}.$

We also remark that the convergence to equilibrium assumption is sometimes
not trivial to be proved in a given system, but it is somehow expected in
systems having some sort of indecomposability and chaotic behavior (for
instance some kind of topological mixing, expansion, hyperbolicity or
presence of noise, see also Remark \ref{rmk37}).

The following statement estimates the speed of convergence to equilibrium
for self-consistent transfer operators \ $\mathcal{L}_{\delta }$ when $%
\delta $ is small.

\begin{theorem}
\label{expco} Let $L_{\delta ,\mu }$ be a family of \ Markov operators
satisfying the Standing assumptions 2 (including $(Con1),...,(Con3)$) for
some $\hat{\delta}>0$ and that 
\begin{equation}
\sup_{\mu \in P_{w},\delta \leq \hat{\delta}}||L_{\delta ,\mu
}||_{B_{ss}\rightarrow B_{ss}}<+\infty .  \label{supp1}
\end{equation}%
Let us consider for all\ $\delta \leq \hat{\delta}$ the self-consistent
operator $\mathcal{L}_{\delta }$ defined as in $(\ref{cupled})$, suppose
that for each such $\delta $ there is an invariant probability measure $\mu
_{\delta }\in P_{ss}$ for $\mathcal{L}_{\delta }$ and suppose that%
\begin{equation}
\sup_{\delta \leq \hat{\delta}}||\mu _{\delta }||_{ss}<+\infty .
\label{supp2}
\end{equation}%
Then there exists $\overline{\delta }$ such that $0<\overline{\delta }<\hat{%
\delta}$ and there are $C,\gamma \geq 0$ such that for all $n\in \mathbb{N}$%
, $0<\delta <\overline{\delta },$\ $\nu \in P_{ss}$ we have%
\begin{equation}
||\mathcal{L}_{\delta }^{n}(\nu )-\mu _{\delta }||_{s}\leq Ce^{-\gamma
n}||\nu -\mu _{\delta }||_{s}.  \label{spe}
\end{equation}
\end{theorem}

We remark that the convergence speed estimates provided in $(\ref{spe})$ are
in the strong norm. These estimates are uniform for $\delta $ small enough
and uniform in $\nu $. We also remark that since there is the strong norm on
both sides of the inequality, $(\ref{spe})$ is similar to a spectral gap
estimate, rather than a \ convergence to equilibrium estimate (where the
regularity of the measure is estimated in the strong norm and the
convergence is in the weak one, resulting in a weaker estimate).

Before the proof of Theorem \ref{expco} we prove several results on the
convergence to equilibrium of a sequential composition of operators in the
family $L_{\delta ,\mu }$. In particular it will be useful to prove a Lasota
Yorke inequality for such a composition.

\begin{lemma}
\label{lasotaY copy(1)}Let $L_{\delta ,\mu }$ be a family of Markov
operators \ satisfying $(Con1)$. Let $\mu _{1},...,\mu _{n}\in P_{w}$ and 
\begin{equation}
L(n):=L_{\delta ,\mu _{n}}\circ L_{\delta ,\mu _{n-1}}\circ ...\circ
~L_{\delta ,\mu _{1}}  \label{Ln}
\end{equation}%
be a sequential composition of operators in such family, then%
\begin{equation}
||L(n)f\Vert _{w}\leq ||f\Vert _{w}
\end{equation}%
and%
\begin{equation}
||L(n)f\Vert _{s}\leq \lambda _{1}^{n}\Vert f\Vert _{s}+\frac{B}{1-\lambda
_{1}}\Vert f\Vert _{w}.  \label{lyw}
\end{equation}
\end{lemma}

\begin{proof}
The first inequality is straightforward from $(Con1)$. Let us now prove $(%
\ref{lyw})$. We have 
\begin{equation*}
||L_{\delta ,\mu _{1}}f\Vert _{s}\leq \lambda _{1}\Vert f\Vert _{s}+B\Vert
f\Vert _{w}
\end{equation*}%
thus%
\begin{eqnarray*}
||L_{\delta ,\mu _{2}}\circ L_{\delta ,\mu _{1}}(f)\Vert _{s} &\leq &\lambda
_{1}\Vert L_{\delta ,\mu _{2}}f\Vert _{s}+B\Vert L_{\delta ,\mu _{2}}f\Vert
_{w} \\
&\leq &\lambda _{1}^{2}\Vert f\Vert _{s}+\lambda _{1}B||f||_{w}+B\Vert
f\Vert _{w} \\
&\leq &\lambda _{1}^{2}\Vert f\Vert _{s}+(1+\lambda _{1})B\Vert f\Vert _{w}
\end{eqnarray*}

Continuing the composition we get $(\ref{lyw})$.
\end{proof}

\begin{lemma}
\label{XXX}Let $\delta \geq 0$ and let $\ L(n)$ be a sequential composition
of operators $L_{\delta ,\mu _{i}}$ as in $(\ref{Ln})$ with \ $i\in
\{1,...,n\}$ and $\mu _{i}\in P_{w}$ satisfying the above Standing
assumptions 2 (including $(Con1)$,...,$(Con3)$). Let $L_{0}$ be the operator
in the family for $\delta =0$\ as defined in $(Con2)$. Since $P_{w}$ is
bounded, let us denote by $Q:=\sup_{\mu \in P_{w}}||\mu ||_{w}$. \ Then
there is $C\geq 0$ such that $\forall g\in B_{s},\forall n\geq 0$%
\begin{equation}
||L(n)g-L_{0}^{n}g||_{w}\leq \delta QK(C||g||_{s}+n\frac{B}{1-\lambda }%
||g||_{w}).  \label{2}
\end{equation}%
where $B$ is the second coefficient of the Lasota Yorke inequality $($\ref{1}%
$)$.
\end{lemma}

\begin{proof}
To shorten notation let us denote for $i\in \{1,...,n\},$ $L_{i}:=L_{\delta
,\mu _{i}}$. \ By $(Con2),$ equation $(\ref{mm})$ we get%
\begin{equation*}
||L_{0}g-L_{j}g||_{w}\leq \delta K||\mu _{j}||_{w}||g||_{s}\leq \delta
QK||g||_{s}.
\end{equation*}

The case $n=1$ of $(\ref{2})$ directly follows from $(\ref{mm})$. \ Let us
now suppose inductively%
\begin{equation*}
||L(n-1)g-L_{0}^{n-1}g||_{w}\leq \delta QK(C_{n-1}||g||_{s}+(n-1)\frac{B}{%
1-\lambda _{1}}||g||_{w})
\end{equation*}%
then%
\begin{eqnarray*}
||L_{n}L(n-1)g-L_{0}^{n}g||_{w} &\leq
&||L_{n}L(n-1)g-L_{n}L_{0}^{n-1}g+L_{n}L_{0}^{n-1}g-L_{0}^{n}g||_{w} \\
&\leq
&||L_{n}L(n-1)g-L_{n}L_{0}^{n-1}g||_{w}+||L_{n}L_{0}^{n-1}g-L_{0}^{n}g||_{w}
\\
&\leq &\delta QK(C_{n-1}||g||_{s}+(n-1)\frac{B}{1-\lambda _{1}}%
||g||_{w})+||[L_{n}-L_{0}](L_{0}^{n-1}g)||_{w} \\
&\leq &\delta QK(C_{n-1}||g||_{s}+(n-1)\frac{B}{1-\lambda _{1}}%
||g||_{w})+\delta QK||L_{0}^{n-1}g||_{s} \\
&\leq &\delta QK(C_{n-1}||g||_{s}+(n-1)\frac{B}{1-\lambda _{1}}||g||_{w}) \\
&&+\delta QK(\lambda _{1}^{n-1}||g||_{s}+\frac{B}{1-\lambda _{1}}||g||_{w})
\\
&\leq &\delta QK[(C_{n-1}+\lambda _{1}^{n-1})||g||_{s})+n\frac{B}{1-\lambda
_{1}}K||g||_{w}].
\end{eqnarray*}

The statement follows from the observation that continuing the composition, $%
C_{n}$ remains being bounded by the sum of a geometric series.
\end{proof}

Next statement is inspired by the methods developed in \cite{GNS} and allows
to estimate the speed of convergence to equilibrium of a sequential
composition of linear operators satisfying the Standing assumptions 2
(including $(Con1)$,...,$(Con3)$). The statement is in some sense homologous
to Proposition 2.7 in \cite{CR}.

\begin{proposition}
\label{prop1} Let us consider $\delta \geq 0$ and a family of operators $%
L_{\delta ,\mu }$ satisfying the Standing assumptions 2 (including $(Con1)$%
,...,$(Con3)$). Let us consider a sequential composition $L(n)$ as above.
Let us fix $n_{1}>0$ \ and consider the $2\times 2$ matrix $M$ defined by%
\begin{equation*}
M:=\left( 
\begin{array}{cc}
\lambda _{1}^{n_{1}} & \frac{B}{1-\lambda _{1}} \\ 
\delta QKC+a_{n_{1}} & \delta QKn_{1}\frac{B}{1-\lambda _{1}}%
\end{array}%
\right) .
\end{equation*}%
Under the previous assumptions for any $g\in V_{s}$ the following holds:

(i) for all integer $i\geq 0$ the norms of the iterates $L(in_{1})g$ are
bounded by 
\begin{equation*}
\left( 
\begin{array}{c}
||L(in_{1})g||_{s} \\ 
||L(in_{1})g||_{w}%
\end{array}%
\right) \preceq M^{i}\left( 
\begin{array}{c}
||g||_{s} \\ 
||g||_{w}%
\end{array}%
\right) .
\end{equation*}%
Here $\preceq $ indicates the componentwise $\leq $ relation (both
coordinates are less or equal).

(ii) Let $\rho $ be the maximum eigenvalue of $M^{T}$, \ with eigenvector $%
\left( 
\begin{array}{c}
a \\ 
b%
\end{array}%
\right) .$ Suppose $a,b\geq 0$ and $a+b=1,$ let us define the $(a,b)$
balanced-norm as 
\begin{equation*}
||g||_{(a,b)}:=a||g||_{s}+b||g||_{w}.
\end{equation*}%
\ In this case we have 
\begin{equation}
||L(in_{1})g||_{(a,b)}\leq \rho ^{i}||g||_{(a,b)}.  \label{contra}
\end{equation}%
Furthermore, the situation in which $\rho <1$, $a,b\geq 0$ can be achieved
if $n_{1}$ is big enough and $\delta $ small enough. More precisely, fixing $%
n_{1}$ large enough we have that $\rho =\rho (\delta )$ can be seen as a
function of $\delta $. There is some $\delta _{1}<1$ such that 
\begin{equation}
\rho _{1}=\sup_{\delta \leq \delta _{1}}\rho (\delta )<1  \label{rho1}
\end{equation}%
and there is a positive eigenvector of $\rho (\delta )$ for $\delta \leq
\delta _{1}$.

As a consequence we also have 
\begin{equation*}
||L(in_{1})g||_{s}\leq (1/a)\rho ^{i}||g||_{s},
\end{equation*}%
and 
\begin{equation*}
||L(in_{1})g||_{w}\leq (1/b)\rho ^{i}||g||_{s}.
\end{equation*}
\end{proposition}

For the proof of Proposition \ref{prop1} the following lemma will be useful

\begin{lemma}
\label{matrix}Let us consider real sequences $a_{n},b_{n}$ such that $%
a_{n}\geq 0,~b_{n}\geq 0$ for all $n\in \mathbb{N}$ \ and $%
a_{n},b_{n}\rightarrow 0$, real numbers $\delta ,A,B,C\geq 0$ and a real
matrix of the form 
\begin{equation*}
\left( 
\begin{array}{cc}
b_{n} & \delta B+a_{n} \\ 
A & \delta nC%
\end{array}%
\right) .
\end{equation*}%
Then there is $n_{1}\geq 0$, $\overline{\delta }\geq 0$ and $0\leq \overline{%
\rho }<1$ such that for all $0\leq \delta \leq \overline{\delta }$ the matrix%
\begin{equation*}
\left( 
\begin{array}{cc}
b_{n_{1}} & \delta B+a_{n_{1}} \\ 
A & \delta n_{1}C%
\end{array}%
\right)
\end{equation*}%
has largest eigenvalue $\rho $ such that \ $0\leq \rho \leq \overline{\rho }$
and an associated eigenvector $(a,b)$, such that $a,b\geq 0.$
\end{lemma}

\begin{proof}
Fixing $n$ and letting $\delta \rightarrow 0$, the matrix $\left( 
\begin{array}{cc}
b_{n} & a_{n} \\ 
A & 0%
\end{array}%
\right) $, has maximum right eigenvalue $\frac{1}{2}b_{n}+\frac{1}{2}\sqrt{%
b_{n}^{2}+4Aa_{n}}$ with eigenvector $\left( 
\begin{array}{c}
\frac{1}{2A}\left( b_{n}+\sqrt{b_{n}^{2}+4Aa_{n}}\right) \\ 
1%
\end{array}%
\right) .$ Now if we take $n_{1}$ big enough we can let $0\leq \frac{1}{2}%
b_{n_{1}}+\frac{1}{2}\sqrt{b_{n_{1}}^{2}+4Aa_{n_{1}}}<1$ and then for
sufficiently small $\delta $ the statement holds.
\end{proof}

Now we are ready to prove Proposition \ref{prop1}.

\begin{proof}[Proof of Proposition \protect\ref{prop1}]
For the proof of (i): let us consider $n_{1}\geq 0$ and $g_{0}\in V_{s}$ and
let us denote $g_{i}=L(in_{1})g_{0}.$ By Lemma \ref{lasotaY copy(1)} we have 
\begin{equation}
||g_{i+1}||_{s}\leq \lambda _{1}^{n_{1}}||g_{i}||_{s}+\frac{B}{1-\lambda _{1}%
}||g_{i}||_{w}.  \label{hhh}
\end{equation}%
By Lemma \ref{XXX}, assumption $(Con3)$ and $($\ref{3}$)$ we get 
\begin{equation}
\begin{split}
||g_{i+1}||_{w}& \leq ||L_{0}^{n_{1}}g_{i}||_{w}+\delta
QK(C||g_{i}||_{s}+n_{1}\frac{B}{1-\lambda _{1}}||g_{i}||_{w}) \\
& \leq a_{n_{1}}||g_{i}||_{s}+\delta QK(C||g_{i}||_{s}+n_{1}\frac{B}{%
1-\lambda _{1}}||g_{i}||_{w}).
\end{split}
\label{kkk}
\end{equation}

Compacting $(\ref{hhh})$ \ and $(\ref{kkk})$ into a vector notation, setting 
$v_{i}=\left( 
\begin{array}{c}
||g_{i}||_{s} \\ 
||g_{i}||_{w}%
\end{array}%
\right) $ we get 
\begin{equation}
v_{i+1}\preceq \left( 
\begin{array}{cc}
\lambda _{1}^{n_{1}} & \frac{B}{1-\lambda _{1}} \\ 
\delta QKC+a_{n_{1}} & \delta QKn_{1}\frac{B}{1-\lambda _{1}}%
\end{array}%
\right) v_{i}=Mv_{i}.  \label{4}
\end{equation}

We remark that the matrix $M$ does not depend on $g_{0}$ and depend on the
operators in the family $L_{\delta ,\mu }$, composing the sequential
composition $L(n)$ only by their common coefficients $\lambda
_{1},a_{n_{1}},K,B$ coming from the assumptions $(Con1),...,(Con3).$
Furthermore, since $M$ is positive, $v_{1}\preceq v_{2}$ implies $%
Mv_{1}\preceq Mv_{2}$. Hence the inequality can be iterated and we have%
\begin{equation*}
\nu _{1}\preceq Mv_{0},v_{2}\preceq Mv_{1}\preceq M^{2}v_{0}...
\end{equation*}%
proving (i). To prove (ii) let us consider the $(a,b)$ balanced-norm: $%
||g||_{(a,b)}=a||g||_{s}+b||g||_{w}$. The statement (i) \ implies 
\begin{eqnarray*}
||L(in_{1})g_{0}||_{(a,b)} &=&(a,b)\cdot \left( 
\begin{array}{c}
||g_{i}||_{s} \\ 
||g_{i}||_{w}%
\end{array}%
\right) \\
&\leq &(a,b)\cdot M^{i}\cdot \left( 
\begin{array}{c}
||g_{0}||_{s} \\ 
||g_{0}||_{w}%
\end{array}%
\right) , \\
&\leq &[((a,b)\cdot M^{i})^{T}]^{T}\cdot \left( 
\begin{array}{c}
||g_{0}||_{s} \\ 
||g_{0}||_{w}%
\end{array}%
\right) , \\
&\leq &[(M^{iT}\cdot (a,b)^{T}]^{T}\cdot \left( 
\begin{array}{c}
||g_{0}||_{s} \\ 
||g_{0}||_{w}%
\end{array}%
\right) , \\
&\leq &[\rho ^{i}\cdot (a,b)^{T}]^{T}\cdot \left( 
\begin{array}{c}
||g_{0}||_{s} \\ 
||g_{0}||_{w}%
\end{array}%
\right) ,
\end{eqnarray*}%
hence 
\begin{equation*}
||L(in_{1})g_{0}||_{(a,b)}\leq \rho ^{i}||g_{0}||_{(a,b)}
\end{equation*}%
proving (ii). The remaining part of the statement is a direct consequence of
Lemma \ref{matrix}.
\end{proof}

We are ready to prove the main statement of this section.

\begin{proof}[Proof of Theorem \protect\ref{expco}]
We need to estimate $||\mathcal{L}_{\delta }^{n}(\nu )-\mu _{\delta }||_{s}$%
. Let us denote by $\nu _{n}$ the sequence of probability measures where $%
\nu _{1}=\nu $ and $\nu _{n}=L_{\delta ,\nu _{n-1}}\nu _{n-1}$. The sequence 
$\mathcal{L}_{\delta }^{n}(\nu )$ can be seen as a sequential composition 
\begin{equation*}
\mathcal{L}_{\delta }^{n}(\nu )=L(n)(\nu )
\end{equation*}%
where using the same notations as in $(\ref{Ln})$%
\begin{equation*}
L(n)=L_{\delta ,\nu _{n}}\circ L_{\delta ,\nu _{n-1}}\circ ...\circ
L_{\delta ,\nu _{1}}.
\end{equation*}%
We remark that by the assumptions, $||L_{\delta ,\nu
_{i}}||_{B_{s}\rightarrow B_{s}}$ are uniformly bounded. \ Let us estimate
this by%
\begin{equation}
||L(n)(\nu )-\mathcal{L}_{\delta }^{n}(\mu _{\delta })||_{s}\leq ||L(n)(\nu
)-L(n)(\mu _{\delta })||_{s}+||L(n)(\mu _{\delta })-\mathcal{L}_{\delta
}^{n}(\mu _{\delta })||_{s}.  \label{bivio}
\end{equation}

Since our operators satisfy $(Con1)$,...,$(Con3)$ and $\nu -\mu _{\delta
}\in V$ we can estimate 
\begin{equation}
||L(n)(\nu )-L(n)(\mu _{\delta })||_{s}=||L(n)(\nu -\mu _{\delta })||_{s}
\label{eqlei}
\end{equation}%
using Proposition \ref{prop1}.\footnote{%
The proof is quite technical. We are going to explain its idea informally to
help the reader to understand the motivation of various estimates: \ by
Proposition \ref{prop1} \ we get that $||L(n)(\nu )-L(n)(\mu _{\delta
})||_{s}$ decreases exponentially in $n$. The remaining term $||L(n)(\mu
_{\delta })-L_{\delta ,\mu }^{n}(\mu _{\delta })||s$ \ is small when $\delta 
$ is small \ and $\nu _{1},...,\nu _{n}$ are close to $\mu _{\delta }$
because the operators involved in the composition $L(n)$ are all near to $%
L_{\delta ,\mu _{\delta }}.$\newline
The idea is to use the balanced norm $||~||_{(a,b)}$ to estimate $||~||_{s}$%
, and exploit the fact that after $n_{1}$ iterates $||L(n_{1})(\nu
)-L(n_{1})(\mu _{\delta })||_{(a,b)}$ is contracted by a certain factor $%
\rho _{1}<1$.\newline
If we prove that $\delta $ can be made small enough so that $||L(n_{1})(\mu
_{\delta })-L_{\delta ,\mu _{\delta }}^{n_{1}}(\mu _{\delta })||_{(a,b)}$ \
is not relevant, then we have that also $||L(n_{1})(\nu )-\mathcal{L}%
_{\delta }^{n_{1}}(\mu _{\delta })||_{(a,b)}$ is contracted. Hence
continuing the iteration we have an exponential decrease of this norm, which
implies exponential decrease of the $||~||_{s}$ norm.}

Let $n_{1},$ $\delta _{1},\rho _{1}$ and $||~||_{(a,b)}$ the parameters and
the norm found applying Proposition \ref{prop1} (see in particular $(\ref%
{rho1})$) to $(\ref{eqlei})$. Let us consider $\delta \leq \delta _{1}$. We
remark that the norm $||~||_{(a,b)}$ also depends on $\delta $.

To simplify notations let us define a general constant that will be used in
the estimates. Let%
\begin{equation*}
M_{\delta }:=\max (1+B,||\mu _{\delta }||_{ss},\sup_{\mu \in
P_{w}}||L_{\delta ,\mu }||_{B_{ss}\rightarrow B_{ss}},\sup_{\mu \in
P_{w}}||L_{\delta ,\mu }||_{B_{s}\rightarrow B_{s}})
\end{equation*}%
and%
\begin{equation*}
M_{1}=\sup_{\delta \leq \delta _{1}}(M_{\delta }).
\end{equation*}

By the assumptions $(\ref{supp1}),$ $(\ref{supp2})$ we have that $%
M_{1}<\infty $. \ To find $\overline{\delta }\leq \delta _{1}$ satisfying
our statement we are going to impose a further condition to the parameter $%
\delta $ which is again satisfied for $\delta $ small enough. Let us state
this condition: \ let us define for all $n\geq 0,$ by induction the
following sequence%
\begin{equation}
C_{0}=1,\ C_{n}=M_{1}^{n}C_{n-1}.  \label{at1}
\end{equation}%
Let%
\begin{equation}
M_{2}:=KC_{n_{1}}n_{1}(KM_{1}+1)^{n_{1}}.  \label{at2}
\end{equation}%
Now let us fix $\overline{\delta }\geq 0$ such that 
\begin{equation}
\rho _{2}=(\rho _{1}+\overline{\delta }M_{2})<1.  \label{at5}
\end{equation}

We now see why this condition is sufficient for our statement to hold. We
have indeed%
\begin{eqnarray*}
||L(n)(\mu _{\delta })-\mathcal{L}_{\delta }^{n}(\mu _{\delta })||_{(a,b)}
&=&||L_{\delta ,\nu _{n}}...L_{\delta ,\nu _{1}}\mu _{\delta }-L_{\delta
,\mu _{\delta }}^{n}(\mu _{\delta })||_{(a,b)} \\
&\leq &||L_{{\delta },\nu _{n}}...L_{\delta ,\nu _{1}}\mu _{\delta
}-L_{\delta ,\mu _{\delta }}L_{\delta ,\nu _{n-1}}...L_{\delta ,\nu _{1}}\mu
_{\delta }||_{(a,b)} \\
&&+||L_{\delta ,\mu _{\delta }}L_{\delta ,\nu _{n-1}}...L_{\delta ,\nu
_{1}}\mu _{\delta }-L_{\delta ,\mu _{\delta }}^{n}(\mu _{\delta })||_{(a,b)}.
\end{eqnarray*}%
\ We recall that by $(Con2)$ 
\begin{eqnarray*}
||(L_{\delta ,\nu _{i}}-L_{\delta ,\nu _{j}})(\omega )||_{(a,b)}
&=&a||(L_{\delta ,\nu _{i}}-L_{\delta ,\nu _{j}})(\omega
)||_{w}+b||(L_{\delta ,\nu _{i}}-L_{\delta ,\nu _{j}})(\omega )||_{s} \\
&\leq &a\delta K||\nu _{i}-\nu _{j}||_{w}||\omega ||_{s}+b\delta K||\nu
_{i}-\nu _{j}||_{w}||\omega ||_{ss} \\
&\leq &\delta K||\nu _{i}-\nu _{j}||_{w}||\omega ||_{ss}.
\end{eqnarray*}%
Suppose inductively that 
\begin{equation}
||L_{\delta ,\nu _{n-1}}...L_{\delta ,\nu _{1}}\mu _{\delta }-L_{\delta ,\mu
_{\delta }}^{n-1}(\mu _{\delta })||_{(a,b)}\leq \delta KC_{n-1}(||\nu
_{n-1}-\mu _{\delta }||_{w}+...+||\nu _{1}-\mu _{\delta }||_{w})
\label{indu}
\end{equation}%
(where $C_{n}\geq 1$ as defined in $(\ref{at1})$) then 
\begin{eqnarray*}
||L_{{\delta },\nu _{n}}L_{{\delta },\nu _{n-1}}...L_{\delta ,\nu _{1}}\mu
_{\delta }-L_{\delta ,\mu _{\delta }}L_{{\delta },\nu _{n-1}}...L_{\delta
,\nu _{1}}\mu _{\delta }||_{(a,b)} &\leq &\delta K||\nu _{n}-\mu _{\delta
}||_{w}||L_{\delta ,\nu _{n-1}}...L_{\delta ,\nu _{1}}\mu _{\delta }||_{ss}
\\
&\leq &\delta M_{1}^{n}K||\nu _{n}-\mu _{\delta }||_{w}
\end{eqnarray*}

and%
\begin{equation*}
||L_{\delta ,\mu _{\delta }}L_{{\delta },\nu _{n-1}}...L_{\delta ,\nu
_{1}}\mu _{\delta }-L_{\delta ,\mu _{\delta }}^{n}(\mu _{\delta
})||_{(a,b)}\leq ||L_{\delta ,\mu _{\delta }}||_{_{(a,b)}}||L_{\delta ,\nu
_{n-1}}...L_{\delta ,\nu _{1}}\mu _{\delta }-L_{\delta ,\mu _{\delta
}}^{n-1}(\mu _{\delta })||_{(a,b)}
\end{equation*}

and by $(\ref{indu})$%
\begin{equation*}
||L_{\delta ,\mu _{\delta }}...L_{\delta ,\nu _{1}}\mu _{\delta }-L_{\delta
,\mu _{\delta }}^{n}(\mu _{\delta })||_{(a,b)}\leq ||L_{\delta ,\mu _{\delta
}}||_{(a,b)}\delta KC_{n-1}(||\nu _{n-1}-\mu _{\delta }||_{w}+...+||\nu
_{1}-\mu _{\delta }||_{w})
\end{equation*}

putting the two estimates together%
\begin{eqnarray*}
||L_{\delta ,\nu _{n}}...L_{\delta ,\nu _{1}}\mu _{\delta }-L_{\delta ,\mu
_{\delta }}^{n}(\mu _{\delta })||_{(a,b)} &\leq &\delta M_{1}^{n}K||\nu
_{n}-\mu _{\delta }||_{w} \\
&&+\delta M_{1}KC_{n-1}(||\nu _{n-1}-\mu _{\delta }||_{w}+...+||\nu _{1}-\mu
_{\delta }||_{w}) \\
&\leq &\delta M_{1}^{n}KC_{n-1}(||\nu _{n}-\mu _{\delta }||_{w}+||\nu
_{n-1}-\mu _{\delta }||_{w}+... \\
&&...+||\nu _{1}-\mu _{\delta }||_{w}) \\
&\leq &\delta KC_{n}(||\nu _{n}-\mu _{\delta }||_{w}+||\nu _{n-1}-\mu
_{\delta }||_{w}+...+||\nu _{1}-\mu _{\delta }||_{w}).
\end{eqnarray*}

Now we find a coarse estimate for $||\nu _{n_{1}}-\mu _{\delta }||_{w},||\nu
_{n_{1}-1}-\mu _{\delta }||_{w},...,||\nu _{1}-\mu _{\delta }||_{w}$ which
will be sufficient for our purposes. Recalling that\ $\nu _{n}=L_{\delta
,\nu _{n-1}}\nu _{n-1}$ we have%
\begin{eqnarray*}
||\nu _{n}-\mu _{\delta }||_{w} &\leq &||L_{\delta ,\nu _{n-1}}\nu
_{n-1}-L_{\delta ,\mu _{\delta }}\mu _{\delta }||_{w} \\
&\leq &||L_{\delta ,\nu _{n-1}}\nu _{n-1}-L_{\delta ,\nu _{n-1}}\mu _{\delta
}||_{w}+||L_{\delta ,\nu _{n-1}}\mu _{\delta }-L_{\delta ,\mu _{\delta }}\mu
_{\delta }||_{w}
\end{eqnarray*}%
then%
\begin{eqnarray*}
||L_{\delta ,\nu _{n-1}}\mu _{\delta }-L_{\delta ,\mu _{\delta }}\mu
_{\delta }||_{w} &\leq &\delta K||\nu _{n-1}-\mu _{\delta }||_{w}~||\mu
_{\delta }||_{s} \\
||L_{\delta ,\nu _{n-1}}\nu _{n-1}-L_{\delta ,\nu _{n-1}}\mu _{\delta
}||_{w} &\leq &||\nu _{n-1}-\mu _{\delta }||_{w}
\end{eqnarray*}

and%
\begin{eqnarray*}
||\nu _{n}-\mu _{\delta }||_{w} &\leq &||\nu _{n-1}-\mu _{\delta
}||_{w}(\delta K~||\mu _{\delta }||_{s}+1) \\
&\leq &||\nu _{n-1}-\mu _{\delta }||_{w}(\delta KM_{1}+1)
\end{eqnarray*}

and then 
\begin{equation*}
\max (||\nu _{n}-\mu _{\delta }||_{w},||\nu _{n-1}-\mu _{\delta
}||_{w},...,||\nu _{1}-\mu _{\delta }||_{w})\leq ||\nu -\mu _{\delta
}||_{w}(\delta KM_{1}+1)^{n}.
\end{equation*}%
Finally we have an estimate for $||L_{\delta ,\nu _{n}}...L_{\delta ,\nu
_{1}}\mu _{\delta }-L_{\delta ,\mu }^{n}(\mu _{\delta })||_{(a,b)}:$%
\begin{equation*}
||L_{\delta ,\nu _{n}}...L_{\delta ,\nu _{1}}\mu _{\delta }-L_{\delta ,\mu
}^{n}(\mu _{\delta })||_{(a,b)}\leq \delta KC_{n}n||\nu -\mu _{\delta
}||_{w}(\delta KM_{1}+1)^{n}.
\end{equation*}

Now the main estimates are ready. Let us apply Proposition \ref{prop1} to $(%
\ref{eqlei}).$\ We get%
\begin{eqnarray*}
||L(n_{1})(\nu )-L(n_{1})(\mu _{\delta })||_{(a,b)} &=&||L(n_{1})(\mu
_{\delta }-\nu )||_{(a,b)} \\
&\leq &\rho _{1}||\mu _{\delta }-\nu ||_{(a,b)}
\end{eqnarray*}%
with $\rho _{1}<1$ and then%
\begin{eqnarray}
\qquad ||L(n_{1})(\nu )-L_{\delta ,\mu }^{n_{1}}(\mu _{\delta })||_{(a,b)}
&\leq &||L(n_{1})(\nu )-L(n_{1})(\mu _{\delta })||_{(a,b)}  \label{wwweee} \\
&&+||L(n_{1})(\mu _{\delta })-L_{\delta ,\mu _{\delta }}^{n_{1}}(\mu
_{\delta })||_{(a,b)} \\
&\leq &\rho _{1}||\mu _{\delta }-\nu ||_{(a,b)}  \notag \\
&&+\delta KC_{n_{1}}n_{1}||\nu -\mu _{\delta }||_{w}(\delta KM_{1}+1)^{n_{1}}
\\
&\leq &||(\mu _{\delta }-\nu )||_{(a,b)}(\rho _{1}+\delta
KC_{n_{1}}n_{1}(KM_{1}+1)^{n_{1}}) \\
&\leq &||(\mu _{\delta }-\nu )||_{(a,b)}(\rho _{1}+\delta M_{2})
\end{eqnarray}

where $M_{2}$ is defined as in $($\ref{at2}$)$. But by $(\ref{at5})$%
\begin{equation}
\rho _{2}=(\rho _{1}+\overline{\delta }M_{2})<1.  \label{deltabar}
\end{equation}

Taking $\delta \leq \overline{\delta }$ we hence get that \ for all $i\geq 1$%
\begin{equation*}
||L(in_{1})(\nu )-L_{\delta ,\mu }^{in_{1}}(\mu _{\delta })||_{(a,b)}\leq
\rho _{2}^{i}||(\mu _{\delta }-\nu )||_{(a,b)}
\end{equation*}%
proving the statement.
\end{proof}

\begin{remark}
We remark that if in the previous proof instead of considering $(\ref{bivio}%
) $ \ we considered the estimate 
\begin{eqnarray*}
||L(n)(\nu )-\mathcal{L}_{\delta }^{n}(\mu _{\delta })||_{s} &=&||L(n)(\nu
)-L_{\delta ,\mu _{\delta }}^{n}(\mu _{\delta })||_{s} \\
&\leq &||L(n)(\nu )-L_{\delta ,\mu _{\delta }}^{n}(\nu )||_{s}+||L_{\delta
,\mu _{\delta }}^{n}(\nu )-L_{\delta ,\mu _{\delta }}^{n}(\mu _{\delta
})||_{s}
\end{eqnarray*}%
we would have a much easier estimate for the summand 
\begin{equation*}
||L_{\delta ,\mu _{\delta }}^{n}(\nu )-L_{\delta ,\mu _{\delta }}^{n}(\mu
_{\delta })||_{s}=||L_{\delta ,\mu _{\delta }}^{n}(\nu -\mu _{\delta
})||_{s},
\end{equation*}%
but estimating $||L(n)(\nu )-L_{\delta ,\mu _{\delta }}^{n}(\nu )||_{s}$ by
our assumptions $(Con1),...,(Con3)$ would involve a term of the kind $||\nu
||_{ss}$, which would result in a weaker final statement.
\end{remark}

\section{Statistical stability and linear response for nonlinear
perturbations\label{sec:linresp}}

The concept of \emph{Linear Response} intends to quantify the response of
the statistical properties of the system when it is submitted to a certain
infinitesimal perturbation. This will be measured in some sense by the
derivative of the invariant measure of the system with respect to the
perturbation.\ Let $(\mathcal{L}_{\delta })_{\delta \geq 0}$ be a one
parameter family of transfer operators associated with a family of
perturbations of an initial operator $\mathcal{L}_{0}$, with strength $%
\delta ,$ and let us suppose that $\mu _{\delta }$ is\ the unique invariant
probability measure of the operator $\ \mathcal{L}_{\delta }$ in a certain
space $B_{ss}$. \ The linear response of the invariant measure of $\mathcal{L%
}_{0}$ under the given perturbation is defined by the limit 
\begin{equation}
\dot{\mu}:=\lim_{\delta \rightarrow 0}\frac{\mu _{\delta }-\mu _{0}}{\delta }%
.  \label{LRidea}
\end{equation}%
The topology where this convergence takes place may depend on the system and
on the kind of perturbation applied. The linear response to the perturbation
hence represents the first order term of the response of a system to a
perturbation and when it holds, a linear response formula can be written: $%
\mu _{\delta }=\mu _{0}+\dot{\mu}\delta +o(\delta )$, which is valid in some
weaker or stronger sense.

We remark that given an observable function $c:X\rightarrow \mathbb{R}$ if
the convergence in \eqref{LRidea} is strong enough with respect to the
regularity of $c$ we get%
\begin{equation}
\lim_{\delta \rightarrow 0}\frac{\int \ c\ d\mu _{\delta }-\int \ c\ d\mu
_{0}}{\delta }=\int \ c\ d\dot{\mu}  \label{LRidea2}
\end{equation}%
showing how the linear response of the invariant measure controls the
behavior of observable averages. For instance the convergence in %
\eqref{LRidea2} hold when $c\in L^{\infty }$ and the convergence of the
linear response is in $L^{1}$.

Linear response results in the context of deterministic dynamics have been
obtained first in the case of uniformly hyperbolic systems in \cite{R}.
Nowadays linear response results are known for many other kinds of systems
outside the uniformly hyperbolic case and also in the random case (see \cite%
{Ba2} for a survey mostly related to deterministic systems and the
introduction of \cite{GS} for an overview of the mathematical results in the
random case).

In the case of coupled hyperbolic map lattices with short range interaction,
results on the smooth dependence of the SRB measure were obtained in \cite%
{JD}, \cite{Jd2}. In the case of all-to-all coupled maps with mean field
interaction and hence in the context of the present paper, linear response
results were shown in \cite{ST}. Still in the context of all-to-all coupled
maps, the works \cite{WG} and \cite{WG2} \ show numerical evidence of the
fact that it is possible for a network of coupled maps to exhibit linear
response, even if its units do not.

The interest of the study of the self \ consistent transfer operators in a
weak coupling regime motivates the study of the response to \emph{nonlinear
perturbations} of linear operators. In this section we prove some stability
and linear response results for the invariant measures of a family $\mathcal{%
L}_{\delta }$ of such operators in the limit $\delta \rightarrow 0$ in the
case where the limit operator $\mathcal{L}_{0}$ is linear. We remark that in 
\cite{Sed} an abstract result is proved which can be also applied to the
linear response of fixed points of nonlinear operators under suitable
perturbations.

\noindent \textbf{Standing assumptions 3. }In this section we consider the
following general setting similar to the one considered in \cite{GS} (see
also \cite{Gdisp} and \cite{L2}) for families of linear operators and
independent of the standing assumptions from the previous sections. Let $X$
be a compact metric space. In the following we consider three normed vector
subspaces of $SM(X)$, the spaces $(B_{ss},\Vert ~\Vert _{ss})\subseteq
(B_{s},\Vert ~\Vert _{s})\subseteq (B_{w},\Vert ~\Vert _{w})\subseteq SM(X)$
with norms satisfying%
\begin{equation*}
\Vert ~\Vert _{w}\leq \Vert ~\Vert _{s}\leq \Vert ~\Vert _{ss}.
\end{equation*}%
We remark that, a priori, some of these spaces can be taken equal. Their
actual choice depends on the type of system and perturbation under study.
Again, we will assume that the linear form $\mu \rightarrow \mu (X)$ is
continuous on $B_{i}$, for $i\in \{ss,s,w\}$. Since we will mainly consider
positive, integral preserving operators acting on these spaces, the
following closed invariant spaces $V_{ss}\subseteq V_{s}\subseteq V_{w}$ of
\ zero average measures defined as:%
\begin{equation*}
V_{i}:=\{\mu \in B_{i}|\mu (X)=0\}
\end{equation*}%
where $i\in \{ss,s,w\}$, will play an important role (we recall that $V_{s}$
was already considered in $(Con3)$). \ 

Let us consider a family of \ \ functions $\mathcal{L}_{\delta
}:B_{i}\rightarrow B_{i}$, with $\delta \in \left[ 0,\overline{\delta }%
\right) $. \ $\mathcal{L}_{\delta }$ will be called a family of \emph{%
"nonlinear" Markov operators} if:

\begin{itemize}
\item each $\mathcal{L}_{\delta }$ preserves positive measures,

\item for all $\mu \in SM(X)$ it holds $[\mathcal{L}_{\delta }(\mu )](X)=\mu
(X).$
\end{itemize}

The following is a "statistical stability" statement for a suitable family
of such operators, showing sufficient conditions under which the invariant
probability measures of these operators are stable under small perturbations
of the operators.

\begin{theorem}
\label{ss}Let $\mathcal{L}_{\delta }:B_{i}\rightarrow B_{i}$ with $\delta
\in \left[ 0,\overline{\delta }\right) $ be a family of\ \ "nonlinear"
Markov operators. \ Suppose that $\mathcal{L}_{0}:B_{s}\rightarrow B_{s}$ is
linear and bounded. Suppose that for all $\delta \in \left[ 0,\overline{%
\delta }\right) $ there is a probability measure $h_{\delta }\in B_{ss}$
such that $\mathcal{L}_{\delta }h_{\delta }=h_{\delta }$. Suppose
furthermore that:

\begin{itemize}
\item[(SS1)] (regularity bounds) there is $M\geq 0$ such that for all $%
\delta \in \left[ 0,\overline{\delta }\right) $%
\begin{equation*}
\Vert h_{\delta }\Vert _{ss}\leq M.
\end{equation*}

\item[(SS2)] (convergence to equilibrium for the unperturbed operator) There
is a sequence \ $a_{n}\geq 0$ with $a_{n}\rightarrow 0$ such that for all $%
g\in V_{ss}$%
\begin{equation*}
\Vert \mathcal{L}_{0}^{n}g\Vert _{s}\leq a_{n}||g||_{ss}.
\end{equation*}

\item[(SS3)] (small perturbation) Let $B_{2M}=\{x\in B_{ss},||x||_{ss}\leq
2M\}$. There is $K\geq 0$ such that and $\mathcal{L}_{0}-\mathcal{L}_{\delta
}:B_{2M}\rightarrow B_{s}$ is $K\delta $-Lipschitz.

Then%
\begin{equation*}
\lim_{\delta \rightarrow 0}\Vert h_{\delta }-h_{0}\Vert _{s}=0.
\end{equation*}
\end{itemize}
\end{theorem}

\begin{remark}
\label{rmkLM}The convergence to equilibrium assumption in $(SS2)$ is
required only for the \emph{unperturbed operator} $\mathcal{L}_{0}$, which
is a linear operator. We also remark that under this assumption $h_{0}$ is
the unique invariant probability measure of $\mathcal{L}_{0}$ in $B_{ss}$.
\end{remark}

\begin{proof}
Let us estimate $\Vert h_{\delta }-h_{0}\Vert _{s}$ exploiting $\mathcal{L}%
_{\delta }h_{\delta }=h_{\delta }$\ in the following way: 
\begin{eqnarray*}
\Vert h_{\delta }-h_{0}\Vert _{s} &\leq &\Vert \mathcal{L}_{\delta
}^{n}h_{\delta }-\mathcal{L}_{0}^{n}h_{0}\Vert _{s} \\
&\leq &\Vert \mathcal{L}_{\delta }^{n}h_{\delta }-\mathcal{L}%
_{0}^{n}h_{\delta }\Vert _{s}+\Vert \mathcal{L}_{0}^{n}h_{\delta }-\mathcal{L%
}_{0}^{n}h_{0}\Vert _{s}.
\end{eqnarray*}%
Since $h_{\delta },h_{0}$ are probability measures, $h_{\delta }-h_{0}\in
V_{ss}$ and by $(SS1)$, $\Vert h_{\delta }-h_{0}\Vert _{ss}\leq 2M,$ then
because of the assumption $(SS2)$ we have%
\begin{equation*}
\Vert \mathcal{L}_{0}^{n}h_{\delta }-\mathcal{L}_{0}^{n}h_{0}\Vert _{s}\leq
Q(n)
\end{equation*}%
with $Q(n)=2a_{n}M\rightarrow 0$ (not depending on $\delta $). This implies%
\begin{equation*}
\Vert h_{\delta }-h_{0}\Vert _{s}\leq \Vert \mathcal{L}_{\delta
}^{n}h_{\delta }-\mathcal{L}_{0}^{n}h_{\delta }\Vert _{s}+Q(n).
\end{equation*}

To estimate $\Vert \mathcal{L}_{\delta }^{n}h_{\delta }-\mathcal{L}%
_{0}^{n}h_{\delta }\Vert _{s}$ we rewrite the sum $\mathcal{L}_{0}^{n}-%
\mathcal{L}_{\delta }^{n}$ telescopically so that%
\begin{eqnarray*}
(\mathcal{L}_{\delta }^{n}-\mathcal{L}_{0}^{n})h_{\delta } &=&\sum_{k=1}^{n}%
\mathcal{L}_{0}^{n-k}(\mathcal{L}_{\delta }-\mathcal{L}_{0})\mathcal{L}%
_{\delta }^{k-1}h_{\delta } \\
&=&\sum_{k=1}^{n}\mathcal{L}_{0}^{n-k}(\mathcal{L}_{\delta }-\mathcal{L}%
_{0})h_{\delta }
\end{eqnarray*}%
(note that only the linearity of $\mathcal{L}_{0}$ is used here). The
assumption that $\Vert h_{\delta }\Vert _{ss}\leq M,$ together with the
small perturbation assumption $(SS3)$ implies that $\Vert (\mathcal{L}%
_{\delta }-\mathcal{L}_{0})h_{\delta }\Vert _{s}\leq \delta KM$ as $\delta
\rightarrow 0.$ {Thus}%
\begin{equation}
\Vert h_{\delta }-h_{0}\Vert _{s}\leq Q(n)+nM_{2}(n)[\delta KM]
\end{equation}%
where $M_{2}(n)=\max_{i\leq N}(1,||\mathcal{L}_{0}||_{B_{s}\rightarrow
B_{s}}^{i}).$ Choosing first $n$ big enough to let $Q(n)$ be close to $0$
and then $\delta $ small enough we can make $nM_{2}(n)[\delta KM]$ as small
as wanted, proving the statement.
\end{proof}

We now show a general result about the linear response of fixed points of
Markov operators under suitable nonlinear perturbations, the result will be
applied to self-consistent transfer operators in the following sections.

\begin{theorem}[Linear Response]
\label{thm:linresp} Let $\mathcal{L}_{\delta }:B_{s}\rightarrow B_{s}$ $%
\mathcal{L}_{\delta }:B_{ss}\rightarrow B_{ss}$ with $\delta \in \left[ 0,%
\overline{\delta }\right) $ be a family of nonlinear Markov operators.
Suppose that $\mathcal{L}_{0}$ is linear and bounded $:B_{i}\rightarrow
B_{i} $ for $i\in \{w,s,ss\}$. Suppose that the family satisfy $%
(SS1),(SS2),(SS3).$ Suppose furthermore that the family $\mathcal{L}_{\delta
}$ satisfy

\begin{itemize}
\item[(LR1)] (resolvent of the unperturbed operator) $(Id-\mathcal{L}%
_{0})^{-1}:=\sum_{i=0}^{\infty }\mathcal{L}_{0}^{i}$ is a bounded operator $%
V_{w}\rightarrow V_{w}$.

\item[(LR2)] (small perturbation and derivative operator) Let $\overline{B}%
_{2M}=\{x\in B_{s},||x||_{s}\leq 2M\}$. There is $K\geq 0$ such that $%
\mathcal{L}_{0}-\mathcal{L}_{\delta }:\overline{B}_{2M}\rightarrow B_{w}$ is 
$K\delta $-Lipschitz. Furthermore, there is $\mathcal{\dot{L}}{h_{0}\in V_{w}%
}$ such that%
\begin{equation}
\underset{\delta \rightarrow 0}{\lim }\left\Vert \frac{(\mathcal{L}_{\delta
}-\mathcal{L}_{0})}{\delta }h_{0}-\mathcal{\dot{L}}h_{0}\right\Vert _{w}=0.
\label{derivativeoperator}
\end{equation}
\end{itemize}

Then we have the following Linear Response formula%
\begin{equation}
\lim_{\delta \rightarrow 0}\left\Vert \frac{h_{\delta }-h_{0}}{\delta }-(Id-%
\mathcal{L}_{0})^{-1}\mathcal{\dot{L}}h_{0}\right\Vert _{w}=0.
\label{linresp}
\end{equation}
\end{theorem}

\begin{remark}
The assumption $(LR1)$ on the existence of the resolvent is asked only for
the unperturbed transfer operator, which is linear. This allows a large
class of perturbations. In many systems this assumption will result from the
presence of a spectral gap (compactness or quasi-compactness of $\mathcal{L}%
_{0}$ acting on $B_{w}$).
\end{remark}

\begin{proof}[Proof of Theorem \protect\ref{thm:linresp}]
By Theorem \ref{ss} we have 
\begin{equation}
\lim_{\delta \rightarrow 0}\Vert h_{\delta }-h_{0}\Vert _{s}=0.
\end{equation}

Let us now consider $(Id-\mathcal{L}_{0})^{-1}$ as a continuous linear
operator $V_{w}\rightarrow V_{w}$. Remark that since $\mathcal{\dot{L}}%
h_{0}\in V_{w},$ the resolvent can be computed at $\mathcal{\dot{L}}h_{0}$.
By using that $h_{0}$ and $h_{\delta }$ are fixed points of their respective
operators we obtain that%
\begin{equation*}
(Id-\mathcal{L}_{0})\frac{h_{\delta }-h_{0}}{\delta }=\frac{1}{\delta }(%
\mathcal{L}_{\delta }-\mathcal{L}_{0})h_{\delta }.
\end{equation*}%
Since the operators preserve probability measures, $(\mathcal{L}_{\delta }-%
\mathcal{L}_{0})h_{\delta }\in V_{w}$. By applying the resolvent to both
sides 
\begin{eqnarray*}
(Id-\mathcal{L}_{0})^{-1}(Id-\mathcal{L}_{0})\frac{h_{\delta }-h_{0}}{\delta 
} &=&(Id-\mathcal{L}_{0})^{-1}\frac{\mathcal{L}_{\delta }-\mathcal{L}_{0}}{%
\delta }h_{\delta } \\
&=&(Id-\mathcal{L}_{0})^{-1}\frac{\mathcal{L}_{\delta }-\mathcal{L}_{0}}{%
\delta }h_{0} \\
&&+(Id-\mathcal{L}_{0})^{-1}[\frac{\mathcal{L}_{\delta }-\mathcal{L}_{0}}{%
\delta }h_{\delta }-\frac{\mathcal{L}_{\delta }-\mathcal{L}_{0}}{\delta }%
h_{0}]
\end{eqnarray*}%
we obtain that the left hand side is equal to $\frac{1}{\delta }(h_{\delta
}-h_{0})$. Moreover, with respect to the right hand side we observe that,
applying assumption $(LR2)$ eventually, as $\delta \rightarrow 0$%
\begin{equation*}
\left\Vert (Id-\mathcal{L}_{0})^{-1}[\frac{\mathcal{L}_{\delta }-\mathcal{L}%
_{0}}{\delta }h_{\delta }-\frac{\mathcal{L}_{\delta }-\mathcal{L}_{0}}{%
\delta }h_{0}]\right\Vert _{w}\leq \Vert (Id-\mathcal{L}_{0})^{-1}\Vert
_{V_{w}\rightarrow V_{w}}K\Vert h_{\delta }-h_{0}\Vert _{s}
\end{equation*}%
which goes to zero thanks to Theorem \ref{ss}. Thus considering the limit $%
\delta \rightarrow 0$ we are left with 
\begin{equation*}
\lim_{\delta \rightarrow 0}\frac{h_{\delta }-h_{0}}{\delta }=(Id-L_{0})^{-1}%
\mathcal{\dot{L}}h_{0}.
\end{equation*}%
converging in the $\Vert \cdot \Vert _{w}$ norm, which proves our claim.
\end{proof}

In Sections \ref{exsec} \ and \ref{sec1} we considered nonlinear
self-consistent transfer operators of the type%
\begin{equation*}
\mathcal{L}_{\delta }(\mu )=L_{\delta ,\mu }(\mu )
\end{equation*}%
for $\mu \in A\subseteq P_{w}$. These functions are positive and integral
preserving. In many cases these functions can be extended to nonlinear
Markov operators $B_{i}\rightarrow B_{i}$ for $i\in \{w,s,ss\}$ and the
above statistical stability theorems can be applied, as it will be shown in
the next sections.

\section{Mean field coupled continuous maps\label{contmap}}

We show the flexibility of Theorem \ref{existence1} proving the existence of
an invariant probability measure in the general case of continuous maps
interacting by a Lipschitz coupling function $h$. \ In the following we
denote by $||~||_{Lip}$ the Lipschitz norm, defined by 
\begin{equation*}
||g||_{Lip}=\max (||g||_{\infty },\underset{x,y\in \mathbb{S}^{1}}{\sup }%
\frac{g(y)-g(x)}{d(x,y)})
\end{equation*}%
for $g:\mathbb{S}^{1}\rightarrow \mathbb{R}$.

\begin{proposition}
\label{KB}Let us consider a system of mean field coupled maps as described
in Section \ref{expla} with a map $T_{0}\in C^{0}(\mathbb{S}^{1}\rightarrow 
\mathbb{S}^{1})$, $h\in Lip(\mathbb{S}^{1}\mathbb{\times S}^{1}\rightarrow 
\mathbb{R})$ and $\delta \geq 0$, \ then there is $\mu \in PM(\mathbb{S}%
^{1}) $ such that%
\begin{equation*}
\mathcal{L}_{\delta }(\mu )=\mu .
\end{equation*}
\end{proposition}

\begin{proof}
Let us consider the space of signed Borel measures $SM(\mathbb{S}^{1})$. \
We consider two different norms on this space. $||~||_{w},$ $||~||_{s}$
defined by%
\begin{equation*}
||\mu ||_{w}=\sup_{g\in Lip(\mathbb{S}^{1}\rightarrow \mathbb{R}%
),||g||_{Lip}\leq 1}\int g~d\mu
\end{equation*}
and $||\mu ||_{s}=\mu ^{+}(\mathbb{S}^{1})+\mu ^{-}(\mathbb{S}^{1})$ where $%
\mu ^{\pm }$ are the positive and negative parts of $\mu $ (the total
variation norm). \ We apply Theorem \ref{existence1} with $(SM(\mathbb{S}%
^{1}),||~||_{w})$, $(SM(\mathbb{S}^{1}),||~||_{s})$ as a weak and strong
space. \ We remark that by Prokhorov's theorem, $P_{w}$ is complete when
considered with the $||~||_{w}$ norm.

Now let us define a projection $\pi _{n}$ as requested by Theorem \ref%
{existence1}. \ Let us consider $n\in \mathbb{N}$ and divide $\mathbb{S}^{1}$
into $n$ equal intervals $I_{1},...,I_{n},$ with $I_{i}=[x_{i},x_{i+1})$. \
Let us consider a partition of unity $\{\phi _{1},...,\phi _{n}\}$ made of
continuous piecewise linear functions $\phi _{i}$ which are affine on each
interval of the partition, such that $\phi _{i}(x_{i+1})=1$ and they are
supported on $I_{i \ {mod}(n)}\cup I_{i+1\ {mod}(n)}$ (hat functions).
Let us consider the projection $\pi _{n}:SM(\mathbb{S}^{1})\rightarrow SM(%
\mathbb{S}^{1})$ defined by 
\begin{equation*}
\pi _{n}(\mu )=\sum_{i\leq n}\delta _{x_{i+1}}\int \phi _{i}~d\mu
\end{equation*}%
we have that this projection is linear, preserves probability measures, and $%
||\pi _{n}(\mu )||_{s}\leq ||\mu ||_{s},$ $||\pi _{n}(\mu )||_{w}\leq ||\mu
||_{w}$ (the first inequality is straightforward, for the second see \ \cite[%
Proposition 9.4]{GMN2} Proposition 9.4\footnote{%
The idea of the proof is the following. We consider $\mu $ with $||\mu
||_{W}\leq 1$ and prove $||\pi _{n}\mu ||_{W}\leq 1.$ For this we first
remark that by the way the discretization is constructed, for each function $%
\tilde{g}$ such that $||\tilde{g}||_{Lip}\leq 1$ \ and $\tilde{g}$ is affine
on each interval $I_{n}$ we have $\int \tilde{g}d\mu =\int \tilde{g}d\pi
_{n}\mu \leq 1.$ Then consider a generic Lipschitz function $g$ with $%
||g||_{Lip}\leq 1$ and note that there is a function $\tilde{g}$ affine on
each interval $I_{n}$ \ such that $||\tilde{g}||_{Lip}\leq 1$ and $\int 
\tilde{g}d\pi _{n}\mu =\int gd\pi _{n}\mu $ \ and then $\int gd\pi _{n}\mu
\leq 1$.}). Since, by the definition of $\pi _{n}$, for each interval $I_{n} 
$, the part of the measure $\mu $ which is contained in $I_{n}$ is
transported to the endpoints of the interval $\{x_{i},x_{i+1}\}$ and hence
at a distance $\leq \frac{1}{n}$ we get (see \ \cite[Proposition 9.4]{GMN2},
proof of Proposition 9.5 for the details) 
\begin{equation}
||\pi _{n}(\mu )-\mu ||_{w}\leq \frac{1}{n}||\mu ||_{s}.  \label{221}
\end{equation}

Each invariant probability measure $\mu $ for \ each $L_{\delta ,\mu }$ is
such that $||\mu ||_{s}\leq 1$. The same can be said for the finite
dimensional reduced operator $\pi _{n}L_{\delta ,\pi _{n}\mu }\pi _{n},$
hence $Exi1,~Exi1.b$ are satisfied.

To verify $Exi2$ we have to verify that%
\begin{equation}
||[L_{\delta ,\mu _{1}}-L_{\delta ,\mu _{2}}]\mu ||_{w}\leq \delta K||\mu
||_{s}||\mu _{1}-\mu _{2}||_{w}  \label{on}
\end{equation}%
we remark that since $h$ is $K$ Lipschitz, for all $x\in \mathbb{S}^{1}$ 
\begin{eqnarray*}
|\Phi _{\delta ,\mu _{1}}(x)-\Phi _{\delta ,\mu _{2}}(x)| &=&\delta \int
h(x,y)~d[\mu _{1}-\mu _{2}](y) \\
&\leq &\delta K||\mu _{1}-\mu _{2}||_{w}.
\end{eqnarray*}

Hence 
\begin{eqnarray*}
||[L_{\delta ,\mu _{1}}-L_{\delta ,\mu _{2}}]\mu ||_{w} &=&||[L_{\Phi
_{\delta ,\mu _{1}}}-L_{\Phi _{\delta ,\mu _{2}}}]L_{T}\mu ||_{w} \\
&\leq &\underset{x\in \mathbb{S}^{1}}{\sup }|\Phi _{\delta ,\mu
_{1}}(x)-\Phi _{\delta ,\mu _{2}}(x)|~||L_{T}\mu ||_{s} \\
&\leq &\delta K||\mu _{1}-\mu _{2}||_{w}||L_{T}\mu ||_{s} \\
&\leq &\delta K||\mu _{1}-\mu _{2}||_{w}||\mu ||_{s}
\end{eqnarray*}%
leading to \ref{on}.

We can then apply Theorem \ref{existence1} \ leading to \ the existence of
an invariant probability measure for $\mathcal{L}_{\delta }$.
\end{proof}

\begin{remark}
For simplicity the statement is proved for maps on $\mathbb{S}^{1}$. It
seems that the statement can be generalized with the same idea to maps on
compact metric spaces for which there is a sequence of Lipschitz partitions
of unity, which can be used to define suitable projections $\pi _{n}$ on
combinations of delta measures placed on some sequences of $\epsilon -nets$
covering the space.
\end{remark}

\begin{remark}
In the introduction we described this statement as a kind of
Krilov-Bogoliubov theorem for mean field coupled maps. This similarity is
restricted to the fact that we get a general statement about invariant
measures and continuous maps. Our statement allows to find a fixed
probability measure for the self-consistent transfer operator $\mathcal{L}%
_{\delta }$ associated to the coupled system, and not an invariant measure
for a continuous map on a compact metric space.

We remark that finding such a fixed probability measure for the
self-consistent transfer operator $\mathcal{L}_{\delta }$ associated to the
system (which is a measure on $\mathbb{S}^{1}$) is not equivalent to the
problem of finding invariant measures for the global system $(\mathcal{X},%
\mathcal{T})$ associated to a network of coupled maps as defined in Section %
\ref{expla}. These are measures on $(\mathbb{S}^{1})^{M}$ which could be
equipped with the product $\sigma -$algebra. In this case the system $(%
\mathcal{X},\mathcal{T})$ can also have invariant measures which are product
of different measures on $\mathbb{S}^{1}.$ For a trivial example let us
think about the uncoupled system $(\mathbb{S}^{1},T,\delta ,h)$ where $T$ is
the doubling map and $\delta =0$. In this case, an invariant measure is
given by the product of the Lebesgue measure on some set of coordinates and
the delta measure placed in $0$ (which is a fixed point for $T$) in all the
other coordinates.
\end{remark}

\section{Coupled expanding circle maps\label{secmap}}

In this section we consider self-consistent operators modeling a network of
all to all coupled expanding maps, we will prove the existence of an
absolutely continuous invariant measure and exponential convergence to
equilibrium for this kind of systems. Similar results appeared in \cite{K}
where the rigorous study\ of maps coupled by mean field interaction was
started and in \cite{bal}, \cite{ST} in a more general setting. We \ will \
also consider the zero-coupling limit and the related linear response. We
show that the transfer operators in this limit satisfy the assumptions of
our general theorems considering as a strong and weak spaces suitable
Sobolev spaces $W^{k,1}(\mathbb{S}^{1})$ of measures having a density whose $%
k$-th derivative is in $L^{1}(\mathbb{S}^{1}).$

Let \ $k>1$ and $T_{0}\in C^{k}(\mathbb{S}^{1},\mathbb{S}^{1})$ be \ a
nonsingular map\footnote{%
A nonsingular map $T$ is a map such that for any Lebesgue measurable set $A$
we have $m(A)=0\iff m(T^{-1}(A))=0$, where $m$ is the Lebesgue measure. If $%
T $ is nonsingular its associated pushforward map induces a function $L^{1}(%
\mathbb{S}^{1},\mathbb{R})\rightarrow L^{1}(\mathbb{S}^{1},\mathbb{R})$.} of
the circle. Let us denote the transfer operator associated with $T_{0}$ by $%
L_{T_{0}}.$ We recall that the transfer operator associated with a map can
be defined on signed measures by the pushforward of the map, however when
the map is nonsingular, this operator preserves measures having a density
with respect to the Lebesgue measure, $L^{1}(\mathbb{S}^{1},\mathbb{R})$ and
then with a small abuse of notation, identifying a measure $\mu $ with its
density $h_{\mu }=\frac{d\mu }{dm}$ with respect to the Lebesgue measure $m,$
the same operator can be also considered as $L_{T_{0}}:L^{1}(\mathbb{S}^{1},%
\mathbb{R})\rightarrow L^{1}(\mathbb{S}^{1},\mathbb{R})$. In this case,
given any density $\phi \in L^{1}(\mathbb{S}^{1},\mathbb{R})$ the action of
the operator on the density can then be described by the explicit formula%
\begin{equation*}
\lbrack L_{0}(\phi )](x):=\sum_{y\in T_{0}^{-1}(x)}\frac{\phi (y)}{%
|T_{0}^{\prime }(y)|}.
\end{equation*}

Given $h\in C^{k}(\mathbb{S}^{1}\times \mathbb{S}^{1},\mathbb{R}),$ $\delta
\geq 0$ and (a probability density) $\psi \in L^{1}(\mathbb{S}^{1},\mathbb{R}%
)$, coherently with Section \ref{expla}, we define \ $\Phi _{\delta ,\psi }:%
\mathbb{S}^{1}\rightarrow \mathbb{S}^{1}$ as 
\begin{equation*}
\Phi _{\delta ,\psi }(x)=x+\pi _{\mathbb{S}^{1}}(\delta \int_{\mathbb{S}%
^{1}}h(x,y)\psi (y)dy).
\end{equation*}%
We will always consider $\delta $ small enough such that $\Phi _{\delta
,\psi }$ is a diffeomorphism. Denote by $Q_{\delta ,\psi }$ the transfer
operator associated with $\Phi _{\delta ,\psi },$defined as 
\begin{equation}
\lbrack Q_{\delta ,\psi }(\phi )](x)=\frac{\phi (\Phi _{\delta ,\psi
}^{-1}(x))}{|\Phi _{\delta ,\psi }^{\prime }(\Phi _{\delta ,\psi }^{-1}(x))|}
\label{wex}
\end{equation}%
for any $\phi \in L^{1}(S^{1},\mathbb{R})$.

We will consider expanding maps $T_{0}:\mathbb{S}^{1}\rightarrow \mathbb{S}%
^{1}$ satisfying the following assumptions:

\begin{enumerate}
\item $T_{0}\in C^{6},$

\item there is $\alpha <1$ such that $\forall x\in \mathbb{S}^{1}$, $%
|T_{0}^{\prime }(x)|\geq \alpha ^{-1}>1$.
\end{enumerate}

\begin{definition}
\label{ufm}A set $A_{M,L}$ of expanding maps is called a \emph{uniform family%
} with parameters $M\geq 0$ and $L>1$ if it satisfies uniformly the
expansiveness and regularity condition: $\forall T\in A_{M,L}$%
\begin{equation*}
||T||_{C^{6}}\leq M,~\inf_{x\in S^{1}}|T^{\prime }(x)|\geq L.
\end{equation*}
\end{definition}

It is well known that the transfer operator associated with a smooth
expanding map has spectral gap and it is quasicompact when acting on
suitable Sobolev spaces (see e.g.\cite{L2}). \ In the following we recall
some particularly important related estimates we will use in this paper. We
start by recalling the fact that such transfer operators satisfy some one
step Lasota Yorke inequalities over these Sobolev spaces (see \cite{GS},
Lemma 29 \ and its proof). This will be useful when applying the results of
Section \ref{sec1}.

\begin{lemma}
\label{Lemsu} Let $A_{M,L}$ be a uniform family of expanding maps, the
transfer operators $L_{T}$ associated with a map $T\in A_{M,L}$ satisfy a
uniform Lasota-Yorke inequality on $W^{k,1}(\mathbb{S}^{1})$: let $\alpha
:=L^{-1}<1$. For all $1\leq k\leq 5$ there are, $A_{k},~B_{k}\geq 0$ such
that for all $n\geq 0,$ $T\in A_{M,L}$%
\begin{eqnarray}
||L_{T}^{n}f\Vert _{W^{k-1,1}} &\leq &A_{k}||f\Vert _{W^{k-1,1}}
\label{LY13} \\
||L_{T}^{n}f\Vert _{W^{k,1}} &\leq &\alpha ^{kn}\Vert f\Vert
_{W^{k,1}}+B_{k}\Vert f\Vert _{W^{k-1,1}}.
\end{eqnarray}
\end{lemma}

From this result, it is classically deduced that the transfer operator $%
L_{T} $ of a $C^{6}$ expanding map $T$ is \emph{quasi-compact} on each $%
W^{k,1}(\mathbb{S}^{1}),$ with $1\leq k\leq 5$. Furthermore, by the
topological transitivity of expanding maps, $1$ is the only eigenvalue on
the unit circle and this implies the following result. (see \cite{GS}
Proposition 30).

\begin{proposition}
\label{propora} \ For all $T\in A_{M,L}$, there are $C\geq 0$ and $\rho \in
(0,1)$ such that for all%
\begin{equation*}
g\in V_{k}:=\{g\in W^{k,1}(\mathbb{S}^{1})~s.t.~\int_{\mathbb{S}^{1}}g~dm=0\}
\end{equation*}%
with $1\leq k\leq 5$ and $n\geq 0$ it holds%
\begin{equation*}
\Vert L_{T}^{n}g\Vert _{W^{k,1}}\leq C\rho ^{n}\Vert g\Vert _{W^{k,1}}.
\end{equation*}%
In particular, the resolvent $R(1,L):=(Id-L_{T})^{-1}=\sum_{i=0}^{\infty
}L_{T}^{i}$ is a well-defined and bounded operator on $V_{k}$.
\end{proposition}

Now we recall some estimates relative to small perturbations of expanding
maps and their associated transfer operators. These will be useful to apply
our general framework to self-consistent transfer operators representing a
family of coupled expanding maps. The estimates will be useful to check that
the assumptions of our general theorems are satisfied. We will again
identify absolutely continuous measures with their densities and consider
the spaces $W^{3,1}(\mathbb{S}^{1})$,..., $L^{1}(\mathbb{S}^{1})$ as
strongest, strong and weak space.

\begin{proposition}
\label{launo}If $L_{0}$ and $L_{1}$ are transfer operators associated with
expanding maps $T_{0}$ and $T_{1},$ then there is a $C\in \mathbb{R}$\ such
that $\forall k\in \{1,2,3\},~\forall f\in W^{k,1}$:%
\begin{equation}
||(L_{1}-L_{0})f||_{W^{k-1,1}}\leq C||T_{1}-T_{0}||_{C^{k+2}}||f||_{W^{k,1}}.
\end{equation}
\end{proposition}

\begin{proof}
In the case $k=1$ the proof of this statement can be found for example in 
\cite{Gdisp}, Proposition 26. When $k=2$ we have to prove that 
\begin{equation}
||((L_{1}-L_{0})f)^{\prime }||_{L^{1}}\leq
C||T_{1}-T_{0}||_{C^{4}}||f||_{W^{2,1}}  \label{ladue}
\end{equation}%
we have the well known formula (see \cite{Gdisp} Equation 3) valid for $\
i\in \{0,1\}$%
\begin{equation}
\left( L_{i}f\right) ^{\prime }=L_{i}\left( \frac{1}{T^{\prime }}f^{\prime
}\right) -L_{i}\left( \frac{T^{\prime \prime }}{(T^{^{\prime }})^{2}}%
f\right) .  \label{latre}
\end{equation}

By $(\ref{latre})$ we have%
\begin{eqnarray*}
\left\Vert ((L_{1}-L_{0})f)^{\prime }\right\Vert _{1} &\leq &\left\Vert
L_{1}\left( \frac{1}{T_{1}^{\prime }}f^{\prime }\right) -L_{0}\left( \frac{1%
}{T_{0}^{\prime }}f^{\prime }\right) \right\Vert _{1} \\
&&+\left\Vert L_{1}\left( \frac{T_{1}^{\prime \prime }}{(T_{1}^{^{\prime
}})^{2}}f\right) -L_{0}\left( \frac{T_{0}^{\prime \prime }}{(T_{0}^{^{\prime
}})^{2}}f\right) \right\Vert _{1}.
\end{eqnarray*}

Considering each summand and applying the statement for the case $k=1$ we get%
\begin{eqnarray*}
\left\Vert L_{1}\left( \frac{1}{T_{1}^{\prime }}f^{\prime }\right)
-L_{0}\left( \frac{1}{T_{0}^{\prime }}f^{\prime }\right) \right\Vert _{1}
&\leq &\left\Vert L_{1}\left( \frac{1}{T_{1}^{\prime }}f^{\prime }\right)
-L_{1}\left( \frac{1}{T_{0}^{\prime }}f^{\prime }\right) \right\Vert _{1} \\
&&+\left\Vert L_{1}\left( \frac{1}{T_{0}^{\prime }}f^{\prime }\right)
-L_{0}\left( \frac{1}{T_{0}^{\prime }}f^{\prime }\right) \right\Vert _{1} \\
&\leq &||T_{1}-T_{0}||_{C^{4}}K_{1}||f^{\prime
}||_{1}+||T_{1}-T_{0}||_{C^{4}}C\left\Vert \frac{1}{T_{0}^{\prime }}%
f^{\prime }\right\Vert _{W^{1,1}} \\
&\leq &||T_{1}-T_{0}||_{C^{4}}[K_{1}||f^{\prime }||_{1}+C\left\Vert \frac{1}{%
T_{0}^{\prime }}\right\Vert _{C^{1}}||f^{\prime }||_{W^{1,1}}]
\end{eqnarray*}

for some constant $K_{1}\geq 0$. \ Similarly%
\begin{eqnarray*}
\left\Vert L_{1}\left( \frac{T_{1}^{\prime \prime }}{(T_{1}^{^{\prime }})^{2}%
}f\right) -L_{0}\left( \frac{T_{0}^{\prime \prime }}{(T_{0}^{^{\prime }})^{2}%
}f\right) \right\Vert _{1} &\leq &\left\Vert L_{1}\left( \frac{T_{1}^{\prime
\prime }}{(T_{1}^{^{\prime }})^{2}}f\right) -L_{1}\left( \frac{T_{0}^{\prime
\prime }}{(T_{0}^{^{\prime }})^{2}}f\right) \right\Vert _{1} \\
&&+\left\Vert L_{1}\left( \frac{T_{0}^{\prime \prime }}{(T_{0}^{^{\prime
}})^{2}}f\right) -L_{0}\left( \frac{T_{0}^{\prime \prime }}{(T_{0}^{^{\prime
}})^{2}}f\right) \right\Vert _{1} \\
&\leq &\left\Vert \frac{T_{1}^{\prime \prime }}{(T_{1}^{^{\prime }})^{2}}-%
\frac{T_{0}^{\prime \prime }}{(T_{0}^{^{\prime }})^{2}}\right\Vert _{\infty
}||f||_{1} \\
&&+||T_{1}-T_{0}||_{C^{4}}C\left\Vert \frac{T_{0}^{\prime \prime }}{%
(T_{0}^{^{\prime }})^{2}}f\right\Vert _{W^{1,1}} \\
&\leq &||T_{1}-T_{0}||_{C^{4}}[K_{2}||f||_{1}+\delta C\left\Vert \frac{%
T_{0}^{\prime \prime }}{(T_{0}^{^{\prime }})^{2}}\right\Vert
_{C^{1}}||f||_{W^{1,1}}]
\end{eqnarray*}%
for some constant $K_{2}\geq 0$. Proving the statement. We remark that $%
\left\Vert \frac{T_{0}^{\prime \prime }}{(T_{0}^{^{\prime }})^{2}}%
\right\Vert _{C^{1}}$ involves the third derivative of $T_{0}.$

When $k=3$ we have to prove that 
\begin{equation}
||((L_{1}-L_{0})f)^{\prime \prime }||_{1}\leq
C||T_{1}-T_{0}||_{C^{5}}||f||_{W^{3,1}}
\end{equation}%
taking a further derivative in $(\ref{latre})$ for a transfer operator $%
L_{1} $ we get 
\begin{equation*}
((L_{1}f)^{^{\prime }})^{\prime }=\left( L_{1}\left( \frac{1}{T^{\prime }}%
f^{\prime }\right) \right) ^{\prime }-\left( L_{1}\left( \frac{T^{\prime
\prime }}{(T^{^{\prime }})^{2}}f\right) \right) ^{\prime }
\end{equation*}%
where%
\begin{eqnarray*}
\left( L_{1}\left( \frac{1}{T^{\prime }}f^{\prime }\right) \right) ^{\prime
} &=&L_{1}\left( \frac{1}{T^{\prime }}(\frac{1}{T^{\prime }}f^{\prime
})^{\prime }\right) -L_{1}\left( \frac{T^{\prime \prime }}{(T^{^{\prime
}})^{2}}(\frac{1}{T^{\prime }}f^{\prime })\right) \\
&=&L_{1}\left( \frac{1}{T^{\prime }}(\frac{-T^{\prime \prime }}{(T^{\prime
})^{2}}f^{\prime }+\frac{1}{T^{\prime }}f^{^{\prime \prime }})\right)
-L_{1}\left( \frac{T^{\prime \prime }}{(T^{^{\prime }})^{2}}(\frac{1}{%
T^{\prime }}f^{\prime })\right)
\end{eqnarray*}%
and%
\begin{eqnarray*}
\left( L_{1}(\frac{T^{\prime \prime }}{(T^{^{\prime }})^{2}}f)\right)
^{\prime } &=&L_{1}\left( \frac{1}{T^{\prime }}(\frac{T^{\prime \prime }}{%
(T^{^{\prime }})^{2}}f)^{\prime }\right) -L_{1}\left( \frac{T^{\prime \prime
}}{(T^{^{\prime }})^{2}}(\frac{T^{\prime \prime }}{(T^{^{\prime }})^{2}}%
f)\right) \\
&=&L_{1}\left( \frac{1}{T^{\prime }}((\frac{T^{\prime \prime }}{(T^{^{\prime
}})^{2}})^{\prime }f+\frac{T^{\prime \prime }}{(T^{^{\prime }})^{2}}%
f^{\prime })\right) -L_{1}\left( \frac{T^{\prime \prime }}{(T^{^{\prime
}})^{2}}(\frac{T^{\prime \prime }}{(T^{^{\prime }})^{2}}f)\right)
\end{eqnarray*}%
and the proof can be concluded as before.
\end{proof}

Now we can prove that the self-consistent transfer operator associated with
a family of coupled expanding maps has a regular invariant measure. We
remark that since expanding maps are continuous, the mere existence for an
invariant measure for these systems can be obtained by Proposition \ref{KB}.
In the following result we prove the existence in of such measure in a space
of measures having a smooth density. Let us consider the expanding map $%
T_{0} $ and denote with $L_{T_{0}}$ its transfer operator, \ consider $%
\delta \geq 0$ and a coupling function $\ h\in C^{6}(\mathbb{S}^{1}\times 
\mathbb{S}^{1}\rightarrow \mathbb{R})$, consider the extended system $(%
\mathbb{S}^{1},T_{0},\delta ,h)$ in which these maps are coupled by $h$ as
explained in Section \ref{expla} \ and the associated self-consistent
transfer operator $\mathcal{L}_{\delta }:L^{1}(\mathbb{S}^{1},\mathbb{R}%
)\rightarrow L^{1}(\mathbb{S}^{1},\mathbb{R})$%
\begin{equation}
\mathcal{L}_{\delta }(\phi )=Q_{\delta ,\phi }(L_{T_{0}}(\phi ))
\end{equation}%
as defined at $($\ref{cupled}$)$. We show that this transfer operator has
under suitable assumptions a fixed probability density in $W^{5,1}$ which is
unique when \ $\delta $ is small enough.

\begin{proposition}[Existence and uniqueness of the invariant measure]
\label{existenceexp}Let $T_{0}$, $h$, $\delta $ and $\mathcal{L}_{\delta }$
as above. Suppose $\delta $ is such that the set $\underset{\phi \in P_{w}}{%
\cup }\{\Phi _{\delta ,\phi }\circ T_{0}\}\subseteq A_{M,L}$ is contained in
a uniform family of expanding maps with parameters $M,L$ (see Definition \ref%
{ufm}). Then there is at least one probability density $f_{\delta }\in
W^{5,1}$ such that%
\begin{equation*}
\mathcal{L}_{\delta }(f_{\delta })=f_{\delta }.
\end{equation*}%
For every such \ invariant measure, $||f_{\delta }||_{W^{3,1}}\leq C(M,L)$
(the $W^{5,1}$ norm is bonded by a constant only depending on $M$ and $L$).

Furthermore we have the uniqueness in the weak coupling regime. There is $%
\overline{\delta }$ such that for each\ $\delta \leq \overline{\delta \text{ 
}}$, $\mathcal{L}_{\delta }$ \ has unique invariant measure in $L^{1}$.
\end{proposition}

Before the proof of this proposition we need to collect some further
preliminary result.

First we prove a one-step Lasota-Yorke inequality for the bounded variation
norm for expanding maps. This result is surely known to the experts. We
prove it for completeness.

\begin{proposition}
\label{LYBV}Let $T$ be an expanding map of the circle. Let $\phi $ a bounded
variation density. Then 
\begin{equation}
Var(L_{T}(\phi ))\leq \frac{1}{\inf_{\mathbb{S}^{1}}(T^{\prime })}Var(\phi
)+\sup_{\mathbb{S}^{1}}(|\frac{T^{\prime \prime }}{T^{\prime 2}}|)\int_{%
\mathbb{S}^{1}}|\phi |~dm.  \label{ok}
\end{equation}
\end{proposition}

\begin{proof}
Let us suppose $T$ of degree $n$ and let us consider $y_{1},...,y_{k}\in 
\mathbb{S}^{1}$. \ Suppose $I_{i}=[y_{i},y_{i+1}]$ (where the indices are
considered modulo $k$). Suppose $T^{-1}(I_{i})=\cup _{1\leq j\leq n}J_{j,i}$
\ and denote $J_{j,i}=[l_{j,i},r_{j,i}]$ (the left and right endpoints).
Given a function $\phi :\mathbb{S}^{1}\rightarrow \mathbb{R}$, let us denote
by 
\begin{equation*}
v(\phi ,I_{i}):=|\phi (y_{i \ {mod}(k)})-\phi (y_{i+1 \  {mod}(k)})|.
\end{equation*}%
We estimate $Var(L_{T}[\phi ])$. We have $Var(L_{T}[\phi ])\leq
\sum_{i=1}^{k}v(L[\phi ],I_{i})$.

We have that%
\begin{eqnarray*}
v(L[\phi ],I_{i}) &=&|L[\phi ](y_{i~ {mod}(k)})-L[\phi ](y_{i+1~  \ {mod%
}(k)})| \\
&\leq &|\sum_{j=1}^{n}\frac{\phi (l_{j,i})}{T^{^{\prime }}(l_{j,i})}-\frac{%
\phi (r_{j,i})}{T^{^{\prime }}(r_{j,i})}| \\
&\leq &|\sum_{j=1}^{n}\frac{\phi (l_{j,i})}{T^{^{\prime }}(l_{j,i})}-\frac{%
\phi (r_{j,i})}{T^{^{\prime }}(l_{j,i})}|+|\sum_{j=1}^{n}\frac{\phi (r_{j,i})%
}{T^{^{\prime }}(l_{j,i})}-\frac{\phi (r_{j,i})}{T^{^{\prime }}(r_{j,i})}| \\
&\leq &\frac{1}{\inf_{\mathbb{S}^{1}}(T^{\prime })}|\sum_{j=1}^{n}\phi
(l_{j,i})-\phi (r_{j,i})|+|\sum_{j=1}^{n}\frac{\phi (r_{j,i})}{T^{^{\prime
}}(l_{j,i})}-\frac{\phi (r_{j,i})}{T^{^{\prime }}(r_{j,i})}|.
\end{eqnarray*}

The second summand can be bounded by remarking that by Lagrange theorem
there is $\xi _{j,i}\in J_{j,i}$ such that 
\begin{equation*}
|\frac{1}{T^{^{\prime }}(l_{j,i})}-\frac{1}{T^{^{\prime }}(r_{j,i})}|=|\frac{%
T^{\prime \prime }(\xi _{i})}{(T^{\prime }(\xi _{i}))^{2}}||r_{j,i}-l_{j,i}|.
\end{equation*}

Then%
\begin{eqnarray*}
|\sum_{j=1}^{n}\frac{\phi (r_{j,i})}{T^{^{\prime }}(l_{j,i})}-\frac{\phi
(r_{j,i})}{T^{^{\prime }}(r_{j,i})}| &\leq &\sum_{j=1}^{n}|\phi (r_{j,i})||%
\frac{1}{T^{^{\prime }}(l_{j,i})}-\frac{1}{T^{^{\prime }}(r_{j,i})}| \\
&\leq &\sum_{j=1}^{n}|\phi (r_{j,i})||\frac{T^{\prime \prime }(\xi _{i})}{%
(T^{\prime }(\xi _{i}))^{2}}||r_{j,i}-l_{j,i}| \\
&\leq &\sup_{x\in \mathbb{S}^{1}}|\frac{T^{\prime \prime }(\xi _{i})}{%
(T^{\prime }(\xi _{i}))^{2}}|\sum_{j=1}^{n}|\phi (r_{j,i})||r_{j,i}-l_{j,i}|
\end{eqnarray*}

Finally we have%
\begin{eqnarray*}
\sum_{i=1}^{k}v(L[\phi ],I_{i}) &\leq &\sum_{i=1}^{k}[\frac{1}{\inf_{\mathbb{%
S}^{1}}(T^{\prime })}|\sum_{j=1}^{n}\phi (l_{j,i})-\phi
(r_{j,i})|+\sup_{x\in \mathbb{S}^{1}}|\frac{T^{\prime \prime }(\xi _{i})}{%
(T^{\prime }(\xi _{i}))^{2}}|\sum_{j=1}^{n}|\phi (r_{j,i})||r_{j,i}-l_{j,i}|]
\\
&\leq &\frac{1}{\inf_{\mathbb{S}^{1}}(T^{\prime })}\sum_{i=1}^{k}%
\sum_{j=1}^{n}v(\phi ,J_{j,i})+\sup_{x\in \mathbb{S}^{1}}|\frac{T^{\prime
\prime }(\xi _{i})}{(T^{\prime }(\xi _{i}))^{2}}|\sum_{i=1}^{k}%
\sum_{j=1}^{n}|\phi (r_{j,i})||r_{j,i}-l_{j,i}|.
\end{eqnarray*}

We remark that when the subdivision $J_{j,i}$ is fine enough $%
\sum_{i=1}^{k}\sum_{j=1}^{n}|\phi (r_{j,i})||r_{j,i}-l_{j,i}|\leq 2\int_{%
\mathbb{S}^{1}}\phi ~dm$ and $\sum_{i=1}^{k}\sum_{j=1}^{n}v(\phi
,J_{j,i})\leq Var(\phi )$. This leads directly to the result.
\end{proof}

The following Lemma is about the nowadays well known statistical stability
of expanding maps (see e.g. \cite{Gdisp} Section 4 and Section 7.4. \ for
more details).

\begin{lemma}
\label{lipz} Given a uniform set of expanding maps $A_{M,L}$, there is $%
K\geq 1$ such that it holds%
\begin{equation*}
||f_{1}-f_{2}||_{L^{1}}\leq K||T_{1}-T_{2}||_{C^{6}}
\end{equation*}%
for all $T_{1},T_{2}\in A_{M,L}$ having \ $f_{1},f_{2}\in W^{5,1}$ as
absolutely continuous invariant densities.
\end{lemma}

Now we estimate how the transfer operator $Q_{\delta ,\psi }$ changes when
changing $\psi $.\ This will allow to apply Lemma \ref{lipz} \ in the proof
of Proposition \ref{existence}.

\begin{lemma}
\label{rimonta}If $T_{0},h\in C^{k}$ there are $K_{1}$ $\geq 1$ such that
for all $\psi ,\phi \in L^{1}$ 
\begin{eqnarray*}
||\Phi _{\delta ,\psi }-\Phi _{\delta ,\phi }||_{C^{k}} &\leq &\delta
K_{1}||\psi -\phi ||_{L^{1}} \\
||\Phi _{\delta ,\psi }\circ T_{0}-\Phi _{\delta ,\phi }\circ
T_{0}||_{C^{k}} &\leq &\delta K_{1}||\psi -\phi ||_{L^{1}}.
\end{eqnarray*}
\end{lemma}

\begin{proof}
We have that $\Phi _{\delta ,\psi }(x)=x+\pi (\delta \int_{\mathbb{S}%
^{1}}h(x,y)\psi (y)dy)$ and $\Phi _{\delta ,\phi }(x)=x+\pi (\delta \int_{%
\mathbb{S}^{1}}h(x,y)\phi (y)dy)$, hence when $\delta $ and $||\psi -\phi
||_{L^{1}}$ are small enough%
\begin{eqnarray*}
|\Phi _{\delta ,\psi }(x)-\Phi _{\delta ,\phi }(x)| &=&|x+\pi (\delta \int_{%
\mathbb{S}^{1}}h(x,y)\psi (y)dy)-x+\pi (\delta \int_{\mathbb{S}%
^{1}}h(x,y)\phi (y)dy)| \\
&=&|\delta \int_{\mathbb{S}^{1}}h(x,y)[\psi (y)-\phi (y)]dy|\leq \delta
||h||_{L^{\infty }}||\phi -\psi ||_{L^{1}}.
\end{eqnarray*}

Considering the derivative with respect to $x$ we get $\Phi _{\delta ,\psi
}^{\prime }(x)=1+\pi (\delta \int_{\mathbb{S}^{1}}\frac{\partial h(x,y)}{%
\partial x}\psi (y)dy)$ and similarly for $\Phi _{\delta ,\phi }(x)$. We
have then%
\begin{equation*}
|\Phi _{\delta ,\psi }^{^{\prime }}(x)-\Phi _{\delta ,\phi }^{\prime
}(x)|=|\delta \int_{\mathbb{S}^{1}}\frac{\partial h(x,y)}{\partial x}[\psi
(y)-\phi (y)]dy|
\end{equation*}

and 
\begin{equation*}
|\Phi _{\delta ,\psi }^{^{\prime }}(x)-\Phi _{\delta ,\phi }^{\prime
}(x)|\leq \delta ||\frac{\partial h}{\partial x}||_{L^{\infty }}||\phi -\psi
||_{L^{1}}.
\end{equation*}%
similarly, we get the same estimate for the further derivatives, proving the
statement.
\end{proof}

Now we are ready for the proof of Proposition \ref{existenceexp}.

\begin{proof}[Proof of Proposition \protect\ref{existenceexp}]
Now we consider the first part of the statement and the existence of an
invariant measure in the stronger coupling case. The existence in $L^{1}$\
of a fixed probability measure for $\mathcal{L}_{\delta }$ in this case
follows from Theorem \ref{existence1}, applying it with $B_{s}=BV[\mathbb{S}%
^{1}]$ and $B_{w}=L^{1}[\mathbb{S}^{1}]$ to the family of operators $%
L_{\delta ,\mu }=Q_{\delta ,\mu }\circ L_{T_{0}}$. We now verify that the
required assumptions hold.

The maps $\Phi _{\delta ,\mu }\circ T_{0}$ involved in the system are a
uniform family of expanding maps. By Proposition \ref{LYBV} the operators $%
L_{\delta ,\mu }$ satisfy a common Lasota Yorke inequality on $BV[\mathbb{S}%
^{1}]$ and $L^{1}[\mathbb{S}^{1}]$ and this gives a family of invariant
measures for the operators $L_{\delta ,\mu }$ which is uniformly bounded in
in $BV[\mathbb{S}^{1}]$ hence $(Exi1)$ \ is verified in this case.

We now verify $(Exi2)$ for the $BV$ norm. Let $f\in BV$, consider $%
f_{\epsilon }\in W^{1,1}$ with $||f_{\epsilon }||_{W^{1,1}}\leq
||f||_{BV}+\epsilon $ and $||f_{\epsilon }-f||_{L^{1}}\leq \epsilon .$ 
\begin{eqnarray*}
||(L_{\delta ,\mu _{1}}-L_{\delta ,\mu _{2}})f||_{L^{1}} &=&||(L_{\delta
,\mu _{1}}-L_{\delta ,\mu _{2}})[f-f_{\epsilon }+f_{\epsilon }]||_{L^{1}} \\
&\leq &||(L_{\delta ,\mu _{1}}-L_{\delta ,\mu _{2}})[f-f_{\epsilon
}]||_{L^{1}}+||(L_{\delta ,\mu _{1}}-L_{\delta ,\mu _{2}})f_{\epsilon
}||_{L^{1}} \\
&\leq &2M\epsilon +C||\mu _{1}-\mu _{2}||_{L^{1}}||f_{\epsilon }||_{W^{1,1}}
\\
&\leq &2M\epsilon +C||\mu _{1}-\mu _{2}||_{L^{1}}(||f_{\epsilon
}||_{BV}+\epsilon )
\end{eqnarray*}

and since $\epsilon $ is arbitrary, also $(Exi2)$ is verified in this case.

Let $\mathcal{P}_{n}$ be the partition subdividing the circle into $n$ equal
intervals. We can consider $\pi _{n}:L^{1}(\mathbb{S}^{1})\rightarrow L^{1}(%
\mathbb{S}^{1})$ to be the Ulam discretization defined as $\pi _{n}(f)=%
\mathbf{E}(f|\mathcal{P}_{n})$, where the conditional expectation is made
using the Lebesgue measure, projecting to piecewise constant functions.

For the Ulam projection it is known that $||\pi _{n}f-f||_{L^{1}}\leq \frac{K%
}{n}||f||_{BV}$, $||\pi _{n}||_{L^{1}\rightarrow L^{1}}\leq 1$, $||\pi
_{n}||_{BV\rightarrow BV}\leq 1$ (see \cite{gani} or the proof of \cite[%
Lemma 4.1]{L3} e.g.) and then the discretized operators $\pi _{n}L_{\delta
,\mu }\pi _{n}$ satisfy a uniform Lasota Yorke inequality on $BV$ and $L^{1}$
(see e.g. \cite{Gdisp}, Section 9.3). By this the assumption $Exi1.b$ is
satisfied. We can then apply Theorem \ref{existence1}, and get the existence
of an invariant probability density $f$ in $BV.$ Since $T_{0},h$ $\in C^{6}$
and $L_{\delta ,f}f=f$ we get that $f\in W^{5,1}$ and its norm can be
uniformly estimated by the uniform Lasota Yorke inequality \ on $W^{5,1}$
and $W^{4,1}$ and then on $W^{4,1}$ and $W^{3,1}$ and so on, satisfied
uniformly by all the transfer operators related to the family of maps $%
A_{M,L}.$

For the second part of the statement (the weak coupling case) we apply
Theorem \ref{existence} with $B_{s}=W^{1,1}[\mathbb{S}^{1}]$ and $%
B_{w}=L^{1}[\mathbb{S}^{1}].$

By Lemma \ref{rimonta} and Lemma \ref{Lemsu} when $\delta $ is small enough
all the operators in the family $L_{\delta ,\mu }$ with $\mu \in P_{w}$ are
the transfer operators associated with a uniform family of expanding maps
and satisfy \ a uniform Lasota Yorke inequality on $W^{1,1}$ and $L^{1},$ by
this each one of these operators has a unique invariant probability measure
in $W^{1,1}$ with uniformly bounded norm and $(Exi1)$ is verified.

By Lemma \ref{rimonta} and Proposition \ref{launo} we get 
\begin{equation}
||(L_{\delta ,\mu _{1}}-L_{\delta ,\mu _{2}})f||_{L^{1}}\leq Const||\mu
_{1}-\mu _{2}||_{L^{1}}||f||_{W^{1,1}}  \label{spr}
\end{equation}%
verifying $(Exi2)$ in this case.

By Lemma \ref{lipz}, Lemma \ref{rimonta} \ also $(Exi3)$ \ is verified. Then
we can apply Theorem \ref{existence} \ to get the existence and uniqueness
for small $\delta $.
\end{proof}

The following statement is an estimate for the speed of convergence to
equilibrium of mean field coupled expanding maps (see \cite{K}, Theorem 4 or 
\cite{ST} Theorem 1.1 for similar statements). \ 

\begin{proposition}[Exponential convergence to equilibrium]
\label{convmaps}Let $\mathcal{L}_{\delta }$ be the family of self-consistent
transfer operators arising from $T_{0}\in C^{6}$ and $h\in C^{6}$ as above.
Let $f_{\delta }\in W^{1,1}$ be an invariant probability density of $%
\mathcal{L}_{\delta }.$ Then there exists $\overline{\delta }>0$ \ and $%
C,\gamma \geq 0$ such that for all $0<\delta <\overline{\delta }$, and each
probability density\ $\nu \in W^{1,1}$ we have%
\begin{equation*}
||\mathcal{L}_{\delta }^{n}(\nu )-f_{\delta }||_{W^{1,1}}\leq Ce^{-\gamma
n}||\nu -f_{\delta }||_{W^{1,1}}.
\end{equation*}
\end{proposition}

\begin{proof}
The proof is an application of Theorem \ref{expco}, considering $%
B_{ss}=W^{2,1},$ $B_{s}=W^{1,1},$ $B_{w}=L^{1}$. Let $L_{\delta f}$ be the
family of transfer operators associated with this system.

By the Lasota Yorke inequalities (Lemma \ref{Lemsu}) we have that the
operators $L_{\delta ,\mu }:W^{2,1}\rightarrow W^{2,1},~L_{\delta ,\mu
}:W^{1,1}\rightarrow W^{1,1},$ $L_{\delta ,\mu }:L^{1}\rightarrow L^{1}$
with $\mu \in P_{w}$ are bounded uniformly for $\delta $ small enough. By
Lemma \ref{Lemsu} they satisfy $(Con1).$ Furthermore by Lemma \ref{rimonta}
\ and Proposition \ref{launo} they satisfy $(Con2)$.

By \ Proposition \ref{existenceexp} the invariant measures $f_{\delta }$
satisfy $\lim_{\delta \rightarrow 0}||f_{\delta }||_{W^{2,1}}<+\infty .$
Since the circle expanding map $T_{0}$ has convergence to equilibrium then $%
(Con3)$ is satisfied. \ We can then apply Theorem \ref{expco} directly
implying the statement.
\end{proof}

To get some useful formula for the linear response for expanding maps
coupled in a mean field regime, let us now consider small perturbations of
expanding maps $T:\mathbb{S}^{1}\rightarrow \mathbb{S}^{1}$ by left
composition with a family of diffeomorphisms $(D_{\delta })_{\delta \in
\lbrack -\epsilon ,\epsilon ]}$. \ More precisely, let $D_{\delta }:\mathbb{S%
}^{1}\rightarrow \mathbb{S}^{1}$ be a diffeomorphism, with 
\begin{equation}
D_{\delta }=\pi _{\mathbb{S}^{1}}\circ (Id+\delta S)
\label{def:diffeoperturb}
\end{equation}%
and $S\in C^{6}(\mathbb{S}^{1},\mathbb{R})$. In this setting let us define
the perturbed transfer operators as 
\begin{equation*}
L_{\delta }=L_{D_{\delta }}\circ L_{T}
\end{equation*}%
(remark that here $L_{0}=L_{T}$ ). \ These kinds of perturbations are of the
type induced by a mean field coupling, they satisfy the "small perturbation"
\ and "existence of a derivative operator" assumptions of our general
theorems like $(Con2)$ or $(LR2)$ of Theorem \ref{thm:linresp}. We have
indeed (see \cite[Proposition 35, 36]{GS}):

\begin{proposition}
\label{thm:regdettransferop} Let $(D_{\delta })_{\delta \in \lbrack 0,%
\overline{\delta }]}$ be as in \eqref{def:diffeoperturb}, and $T:\mathbb{S}%
^{1}\rightarrow \mathbb{S}^{1}$ be a $C^{6}$ uniformly expanding map. Let us
define $\dot{L}:W^{4,1}(\mathbb{S}^{1})\rightarrow W^{3,1}(\mathbb{S}^{1})$
by%
\begin{equation}
\dot{L}(f):=-(S\cdot L_{T}(f))^{\prime }.  \label{eq:realdetderivop}
\end{equation}%
Then one has that\ for all $1\leq k\leq 4$ and $f\in W^{k,1}$ 
\begin{equation}
\left\Vert \dfrac{L_{\delta }-L_{0}}{\delta }(f)-\dot{L}(f)\right\Vert
_{W^{k-1,1}}\underset{\delta \rightarrow 0}{\longrightarrow }0.
\end{equation}%
\ 
\end{proposition}

We have now all the ingredients to prove a result regarding the Linear
Response of the coupled system in the small coupling regime.

\begin{proposition}[Linear Response for coupled expanding maps (zero
coupling limit)]
\label{thm:linresp copy(1)} Consider the family of self-consistent transfer
operators $\mathcal{L}_{\delta }$ associated with a $C^{6}$ expanding map $T$
and a coupling driven by the function $\ h$, with $h\in C^{6}$. \ Let $h_{0}$
be the unique invariant probability measure in $L^{1}$ for $\mathcal{L}_{0}$
and $h_{\delta }$ some invariant probability measure for $\mathcal{L}%
_{\delta }$. Then for $\delta \rightarrow 0$ we have the following Linear
Response formula 
\begin{equation}
\lim_{\delta \rightarrow 0}\left\Vert \frac{h_{\delta }-h_{0}}{\delta }%
+(Id-L_{0})^{-1}(h_{0}\int_{S^{1}}h(x,y)h_{0}(y)dy)^{\prime }\right\Vert
_{W^{1,1}}=0.  \label{iii}
\end{equation}
\end{proposition}

\begin{proof}
The proof is a direct application of Theorem \ref{thm:linresp} to our case
with $B_{ss}=W^{3,1}(\mathbb{S}^{1})\subset B_{s}=W^{2,1}(\mathbb{S}%
^{1})\subset B_{w}=W^{1,1}(\mathbb{S}^{1}).$ Let us we see why the
assumptions needed to apply the theorem are satisfied. We recall that the
transfer operators $\mathcal{L}_{\delta }:W^{3,1}(\mathbb{S}^{1})\rightarrow
W^{3,1}(\mathbb{S}^{1})$ involved are defined by%
\begin{equation*}
\mathcal{L}_{\delta }(\phi )=Q_{\delta ,\phi }(L_{T_{0}}(\phi )).
\end{equation*}

The assumption $(SS1)$ (regularity bounds), is implied by Proposition \ref%
{existenceexp}. The assumption $(SS2)$ (convergence to equilibrium for the
unperturbed operator), is well known to be verified, as it stands for the
unperturbed transfer operator $\mathcal{L}_{0}$ \ which is the transfer
operator associated with a smooth expanding map. The assumption $(LR1)$
regarding the existence of the resolvent of the unperturbed operator on the
weak space $W^{1,1}$ follows from Proposition \ref{propora}.

As $\mathcal{L}_{\delta }$ is a small perturbation of $\mathcal{L}_{0}$ \
given by the composition of the transfer operator $Q_{\delta ,\phi }$ \
associated with a diffeomorphism near to the identity, the assumption $(SS3)$
and the first part of $(LR2)$ (small perturbation) follows from Proposition %
\ref{thm:regdettransferop}, Proposition \ref{launo} and Lemma \ref{rimonta}
as before. Let us prove indeed that there is $K\geq 0$ such that and $%
\mathcal{L}_{0}-\mathcal{L}_{\delta }$ is $K\delta $-Lipschitz when
considered as a function $\overline{B}_{2M}\rightarrow B_{w}$ and $%
B_{2M}\rightarrow B_{s}$. In the first case we have to prove that for $\phi
_{1},\phi _{2}\in W^{2,1}$%
\begin{equation}
||[Q_{\delta ,\phi _{1}}(L_{T_{0}}(\phi _{1}))-L_{T_{0}}(\phi
_{1})]-[Q_{\delta ,\phi _{2}}(L_{T_{0}}(\phi _{2}))-L_{T_{0}}(\phi
_{2})]||_{W^{1,1}}\leq K\delta ||\phi _{1}-\phi _{2}||_{W^{2,1}}.
\label{similarr}
\end{equation}

Developing the formula we get%
\begin{eqnarray*}
&&||Q_{\delta ,\phi _{1}}(L_{T_{0}}(\phi _{1}))-L_{T_{0}}(\phi
_{1})-Q_{\delta ,\phi _{2}}(L_{T_{0}}(\phi _{2}))+L_{T_{0}}(\phi
_{2})||_{W^{1,1}} \\
&\leq &||Q_{\delta ,\phi _{1}}(L_{T_{0}}(\phi _{1}))-L_{T_{0}}(\phi
_{1})-Q_{\delta ,\phi _{1}}(L_{T_{0}}(\phi _{2})) \\
&&+Q_{\delta ,\phi _{1}}(L_{T_{0}}(\phi _{2}))-Q_{\delta ,\phi
_{2}}(L_{T_{0}}(\phi _{2}))+L_{T_{0}}(\phi _{2})||_{W^{1,1}}
\end{eqnarray*}

and%
\begin{eqnarray*}
&&||Q_{\delta ,\phi _{1}}(L_{T_{0}}(\phi _{1}))-L_{T_{0}}(\phi
_{1})-Q_{\delta ,\phi _{1}}(L_{T_{0}}(\phi _{2}))+L_{T_{0}}(\phi
_{2})||_{W^{1,1}} \\
&\leq &||Q_{\delta ,\phi _{1}}(L_{T_{0}}(\phi _{1}-\phi
_{2}))-L_{T_{0}}(\phi _{1}-\phi _{2})||_{W^{1,1}} \\
&\leq &CK_{2}2M\delta ||\phi _{1}-\phi _{2}||_{W^{2,1}}
\end{eqnarray*}

by applying Lemma\ \ref{rimonta} with $\psi =0$ and $\phi =\phi _{1}$\ and
Proposition \ref{launo}. The other term can be estimated as%
\begin{eqnarray*}
||Q_{\delta ,\phi _{1}}(L_{T_{0}}(\phi _{2}))-Q_{\delta ,\phi
_{2}}(L_{T_{0}}(\phi _{2}))||_{W^{11}} &\leq &||Q_{\delta ,\phi _{1}}\circ
L_{T_{0}}-Q_{\delta ,\phi _{2}}\circ L_{T_{0}}||_{W^{2,1}\rightarrow
W^{1,1}}||\phi _{2}||_{W^{2,1}} \\
&\leq &\delta CK_{2}2M||\phi _{1}-\phi _{2}||_{W^{2,1}}
\end{eqnarray*}%
using Proposition \ref{launo} and Lemma \ref{rimonta} \ and proving the
Lipschitz assumption in the $\ \overline{B}_{2M}\rightarrow B_{w}$ case. The
case $B_{2M}\rightarrow B_{s}$ is similar.

The assumption $(LR2)$ on the derivative operator follows from Proposition %
\ref{thm:regdettransferop}. Let us apply it and find an expression for $%
\mathcal{\dot{L}}{h_{0}}$ in our case. In this case the perturbing operator
to be considered is $Q_{\delta ,h_{0}}$ associated with the diffeomorphism $%
\Phi _{\delta ,h_{0}}.$ With the notation $(\ref{def:diffeoperturb})$ \ we
have $D_{\delta }=\Phi _{\delta ,h_{0}}=Id+\delta S$ with 
\begin{equation*}
S(x)=\int_{S^{1}}h(x,y)h_{0}(y)dy
\end{equation*}%
and then%
\begin{equation}
\mathcal{\dot{L}}(h_{0})=-(S~L_{T_{0}}h_{0})^{\prime
}=-(h_{0}\int_{S^{1}}h(x,y)h_{0}(y)dy)^{\prime }.  \label{oooo}
\end{equation}

Applying Theorem \ref{thm:linresp}, we then get 
\begin{equation}
\lim_{\delta \rightarrow 0}\left\Vert \frac{h_{\delta }-h_{0}}{\delta }%
-(Id-L_{T_{0}})^{-1}\mathcal{\dot{L}}(h_{0})\right\Vert _{W^{1,1}}=0
\label{ee}
\end{equation}%
as in our case $W^{1,1}$ is the weak space \ Substituting $(\ref{oooo})$ in $%
(\ref{ee})$ we get $(\ref{iii})$.
\end{proof}

\begin{remark}
\label{dalla}From $(\ref{LRidea2})$ we see that this response result with
convergence in the quite strong topology $W^{1,1}$ gives information on the
behavior of a large class of observables, for example we can consider
observables in $L^{\infty }$ or $L^{1}$ or even distributions in the dual of 
$W^{1,1}$.
\end{remark}

\section{Maps with additive noise on $\mathbb{S}^{1}\label{noise}$}

We illustrate the flexibility of our approach with an application to systems
of coupled random maps. For simplicity we will consider a class of random
dynamical systems on the unit circle $\mathbb{S}^{1}$. Informally speaking,
a random dynamics on $\mathbb{S}^{1}$ is defined by the iteration of maps
chosen randomly in the family $T_{\omega }:\mathbb{S}^{1}\rightarrow \mathbb{%
S}^{1}$, $\omega \in \Omega $ according to a certain probability
distribution $p$ defined on $\Omega $. \ In our case we will model this
random choice as independent and identically distributed at each time.

Let $T_{0}$:$~\mathbb{S}^{1}\rightarrow \mathbb{S}^{1}$ a continuous and
piecewise $C^{1},$ nonsingular map\footnote{%
We mean that $\mathbb{S}^{1}$ can be partitioned in a finite set of
intevrals where $T$ is $C^{1}$ \ and that the associated pushforward map $%
T_{\ast }$ sends a measure which is absolutely continuous with respect to
the Lebesgue measure to another measure which is absolutely continuous with
respect to the Lebesgue measure.}. We consider a random dynamical system,
corresponding to the stochastic process $(X_{n})_{n\in \mathbb{N}}$ defined
by%
\begin{equation}
X_{n+1}=T_{0}(X_{n})+\Omega _{n}\mod 1  \label{eq:syswaddnoise}
\end{equation}%
where $(\Omega _{n})_{n\in \mathbb{N}}$ are i.i.d random variables
distributed according to some smooth kernel $\rho .$ We will call $T_{0}$
the deterministic part of the system and $\rho $ the noise kernel of the
system.

\begin{remark}
We remark that the maps considered here are quite general. We do not require
expansiveness or hyperbolicity, allowing many examples of random maps coming
as models of natural phenomena (see e.g. \cite{GMN}, \cite{tak}, \cite{jst}).
\end{remark}

We will consider the annealed transfer operators associated with these
systems (see \cite{Viana}, Section 5 for more details about transfer
operators associated with this kind of systems or \cite{GS}). Let $SM(%
\mathbb{S}^{1})$ be the space of signed Borel measures in $\mathbb{S}^{1}$.
The annealed transfer operator $L:SM(\mathbb{S}^{1})\rightarrow SM(\mathbb{S}%
^{1})$ associated with the random system is defined by%
\begin{equation}
L(\mu )=\int_{\Omega }L_{T_{\omega }}(\mu )dp
\end{equation}%
where $L_{T_{\omega }}(\mu ):SM(\mathbb{S}^{1})\rightarrow SM(\mathbb{S}%
^{1}) $ is the transfer operator associated with $T_{\omega }$, hence taking
the average of the pushforward maps with respect to $p$. For some class of
random dynamical systems $L$ is defined as an operator $:L^{1}(\mathbb{S}%
^{1})\rightarrow L^{1}(\mathbb{S}^{1})$ and sometime this operator is a 
\emph{kernel} operator: let $k\in L^{\infty }(\mathbb{S}^{1}\times \mathbb{S}%
^{1})$ (the kernel), consider the operator $L$ defined in the following way:
for $f\in L^{1}(\mathbb{S}^{1})$%
\begin{equation}
Lf(x)=\int_{\mathbb{S}^{1}}k(x,y)f(y)dy.  \label{kernelL1}
\end{equation}

This kind of operators naturally appear when the random dynamics is defined
by the action of a deterministic map and some additive noise \ Since the
effect of the additive noise is to perturb the deterministic map by a
translation, the annealed transfer operator will be an average of
translations, i.e. a convolution. \ The well known regularization properties
of convolutions then imply that the annealed transfer operator associated
with a system with additive noise is a regularizing one.

Let us consider a probability density $\rho :$ $\mathbb{R}\rightarrow 
\mathbb{R}$ representing how the noise is distributed. For simplicity we
will suppose $\rho $ being such that $\rho (x)=\rho (-x)$ for all $x\in 
\mathbb{R}$ and being a Schwartz function, hence $\rho \in C^{k}$ for all $%
k\geq 1$. The periodization $\tilde{\rho}:\mathbb{S}^{1}\rightarrow \mathbb{R%
}$ of such a function is defined as 
\begin{equation*}
\tilde{\rho}(x)=\sum_{k\in \mathbb{Z}}\rho (x+k)
\end{equation*}%
which clearly converge for a rapidly decreasing function as $\rho $.

\begin{definition}
Let $f\in L^{1}(\mathbb{S}^{1})$ and $\rho $ as before. We define the
convolution $\rho \ast f$ by 
\begin{equation}
\rho \ast f(x):=\int_{\mathbb{S}^{1}}\tilde{\rho}(x-y)f(y)dy.
\label{def:convprod}
\end{equation}
\end{definition}

To a system with additive noise as defined in $(\ref{eq:syswaddnoise})$ we
then associate the annealed transfer operator $L:L^{1}\rightarrow L^{1}$
defined by 
\begin{equation}
L(\phi ):=\rho \ast \lbrack L_{T_{0}}(\phi )]  \label{def:annealedtransferop}
\end{equation}%
for all $\phi \in L^{1},$ where $L_{T_{0}}$ is the transfer operator
associated with the deterministic map $T_{0}$.

We now define the self-consistent transfer operator associated with an
infinite collection of interacting random maps in a mean field coupling.
Like in the deterministic case let us consider $h\in C^{k}(\mathbb{S}%
^{1}\times \mathbb{S}^{1},\mathbb{R}),$ $\delta \geq 0$, for some $k\geq 1$
and a probability density $\psi \in L^{1}$. Define \ $\Phi _{\delta ,\psi }:%
\mathbb{S}^{1}\rightarrow \mathbb{S}^{1}$ again as 
\begin{equation*}
\Phi _{\delta ,\psi }(x)=x+\pi _{\mathbb{S}^{1}}(\delta \int_{\mathbb{S}%
^{1}}h(x,y)\psi (y)dy).
\end{equation*}

Denote by $Q_{\delta ,\psi }$ the transfer operator associated with $\Phi
_{\delta ,\psi },$ as in $(\ref{wex})$. \ We consider the following family
of operators depending on a probability density $\phi \in L^{1}$ and $\delta
\geq 0$ defined as%
\begin{equation}
L_{\delta ,\phi }(\psi )=\rho \ast \lbrack Q_{\delta ,\phi }(L_{T_{0}}(\psi
))].  \label{noisefam}
\end{equation}%
\ Finally we define the nonlinear self-consistent transfer operator
associated with this system of coupled random maps by 
\begin{equation}
\mathcal{L}_{\delta }(\phi )=\rho \ast \lbrack Q_{\delta ,\phi
}(L_{T_{0}}(\phi ))].  \label{opnoise}
\end{equation}

This represents the idea that a certain initial condition is first moved by
the deterministic part of the dynamics represented by the map $T_{0}$ and by
the mean field perturbation $\Phi _{\delta ,\phi }$, then a \ further
(external and independent of the dynamics) random perturbation is applied by
the noise. In the remaining part of the section we will apply our general
theory to this kind of operators.

The following proposition contains some of the regularization properties for
the convolution we need (see \cite{GS} Proposition 15 for the proof and
details).

\begin{proposition}
Let $f\in L^{1}$ and $\rho $ be as before. The convolution $\rho \ast f$ has
the following properties:

\begin{enumerate}
\item For all $k\geq 1$, $\rho \ast f:\mathbb{S}^{1}\rightarrow \mathbb{R}$
is $C^{k}$, and $(\rho \ast f)^{(i)}=\rho ^{(i)}\ast f$ \ for any $i\leq k$.

\item One has the following regularization inequality:%
\begin{equation}
\Vert \rho \ast f\Vert _{C^{k}}\leq \Vert \rho \Vert _{C^{k}}\Vert f\Vert
_{L^{1}}.  \label{eq:regineqI}
\end{equation}
\end{enumerate}
\end{proposition}

These regularization properties together with the Ascoli Arzela theory imply
that a linear operator $L_{0}:L^{1}\rightarrow L^{1}$ of the kind%
\begin{equation}
L_{0}(\phi )=\rho \ast L_{T_{0}}(\phi )  \label{supra}
\end{equation}%
is a compact operator. If we suppose that the system satisfy a convergence
to equilibrium assumption, this will allow to obtain the spectral gap and
the existence of the resolvent operators, required to apply Theorem \ref%
{thm:linresp}.

\begin{proposition}
\label{uuuu}Let $L_{0}$ be the annealed transfer operator associated with a
map $T_{0}$ with additive noise distributed with a kernel $\rho $ as in $(%
\ref{supra}).$ Consider $m\geq 1$ and suppose that for all $g\in C^{m}$ \
such that $\int g~dm=0$%
\begin{equation}
\lim_{n\rightarrow \infty }\Vert L_{0}^{n}g\Vert _{L^{1}}=0.  \label{coneq}
\end{equation}
Then for all $k\geq 1,$ $(Id-L_{0})^{-1}:=\sum_{i=0}^{\infty }L_{0}^{i}$ is
a bounded operator $C^{k}\rightarrow C^{k}$.
\end{proposition}

\begin{proof}
$(\ref{coneq})$ and $(\ref{eq:regineqI})$ impliy that $\lim_{n\rightarrow
\infty }\Vert L_{0}^{n}g\Vert _{C^{k}}=0$ \ for all $g\in L^{1}$ \ such that 
$\int g~dm=0$ and then some iterate of the transfer operator is a uniform
contraction on the space of $C^{k}$ densities with zero average. By this the
operator has a spectral gap, implying the existence of the resolvent
operator (see for the details \cite{GS}, section IV).
\end{proof}

By \cite[proposition 18 and 19]{GS} and their simple proof the lemma below
directly follows

\begin{lemma}
\label{laquattro}Let us consider transfer operators $L_{0},$ $L_{1}$
associated with dynamical systems with additive noise having noise kernel $%
\rho $, deterministic part given by continuous maps $T_{0}$ and $T_{1}$ and $%
k\geq 0.$ Then there is $C\geq 0$ such that for all such $T_{0},$ $T_{1}$
and $f\in L^{1}$ 
\begin{equation*}
||L_{1}f-L_{0}f||_{C^{k-1}}\leq C||\rho
||_{C^{k}}||T_{0}-T_{1}||_{C^{0}}||f||_{1}.
\end{equation*}
\end{lemma}

We state a result analogous to Proposition \ref{existenceexp} in the case of
systems with additive noise. The application to this case is simpler due to
the regularizing effect of the noise.

\begin{proposition}
\label{existence copy(1)}Suppose $T_{0}$ is $~\mathbb{S}^{1}\rightarrow 
\mathbb{S}^{1}$ continuous, nonsingular and piecewise $C^{1}.$ Suppose $h\in
C^{1}$, let $\mathcal{L}_{\delta }$ be the self-consistent family of
operators associated with this coupled system as defined in $(\ref{opnoise}%
), $ then for all $\delta \geq 0$ there is \ $f_{\delta }\in C^{\infty }$
such that for all $k\geq 1$ 
\begin{equation}
||f_{\delta }||_{C^{k}}\leq ||\rho ||_{C^{k}}  \label{1234}
\end{equation}%
and%
\begin{equation*}
\mathcal{L}_{\delta }(f_{\delta })=f_{\delta }.
\end{equation*}%
Let us suppose that the (linear) operator $\mathcal{L}_{0}$ has convergence
to equilibrium when considered as acting on the spaces $C^{1}$ and $L^{1}$
(see $(\ref{coneq})$) then there is $\overline{\delta }>0$ such that for all 
$\delta \in \lbrack 0,\overline{\delta }],$ \ $f_{\delta }$ is unique.
\end{proposition}

\begin{proof}
We sketch the proof, whose arguments are similar to the ones of Proposition %
\ref{existenceexp}. \ We will obtain the statement applying Theorems \ref%
{existence1} and \ref{existence} to the family of operators%
\begin{equation*}
L_{\delta ,\phi }=\rho \ast \lbrack Q_{\delta ,\phi }\circ L_{T_{0}}]
\end{equation*}%
as defined in $(\ref{noisefam})$. \ First we will apply Theorem \ref%
{existence1} with $B_{s}=BV$ and $B_{w}=L^{1}$. We remark that given $k\geq
1,$ $(\ref{eq:regineqI})$ implies a Lasota Yorke inequality which is
uniformly satisfied by these operators, indeed%
\begin{eqnarray}
||L_{\delta ,\phi }(\psi )||_{C^{k}} &=&||\rho \ast \lbrack Q_{\delta ,\phi
}\circ L_{T_{0}}(\psi )]||_{C^{k}}  \label{LYnoise} \\
&\leq &0||\psi ||_{C^{k}}+||\rho ||_{C^{k}}||[Q_{\delta ,\phi }\circ
L_{T_{0}}(\psi )]||_{1}  \notag \\
&\leq &0||\psi ||_{C^{k}}+||\rho ||_{C^{k}}||\psi ||_{1}.  \notag
\end{eqnarray}

This implies that the transfer operators in the family are uniformly bounded
as operators $L^{1}\rightarrow C^{k}$ and hence each invariant probability
measure $f_{\delta ,\phi }$ of $L_{\delta ,\phi }$ is such that $||f_{\delta
,\phi }||_{BV}\leq ||f_{\delta ,\phi }||_{C^{k}}\leq ||\rho ||_{C^{k}}$ and
then the family of operators satisfy $(Exi1).$

The assumption $(Exi2)$ (or $(\ref{nn})$) is provided similarly as a
consequence of the Lemma \ref{rimonta} and Lemma \ref{laquattro}.

In order to apply Theorem \ref{existence1} we consider a suitable projection 
$\pi _{n}$. As in the proof of Proposition \ref{existenceexp}, let $\mathcal{%
P}_{n}$ be the partition subdividing the circle into $n$ equal intervals.
Consider $\pi _{n}$ defined as in the proof of Proposition \ref{existenceexp}
by $\pi _{n}(f)=\mathbf{E}(f|\mathcal{P}_{n})$, where the conditional
expectation is made using the Lebesgue measure, projecting to piecewise
constant functions. Again, the discretized operators satisfy a uniform
Lasota\ Yorke inequality on $BV$ and $L^{1}$, indeed%
\begin{eqnarray*}
||\pi _{n}L_{\delta ,\phi }\pi _{n}(\psi )||_{BV} &\leq &||L_{\delta ,\phi
}\pi _{n}(\psi )||_{BV} \\
&\leq &||L_{\delta ,\phi }\pi _{n}(\psi )||_{C^{k}} \\
&\leq &||\rho \ast \lbrack Q_{\delta ,\phi }\circ L_{T_{0}}(\pi _{n}(\psi
))]||_{C^{k}} \\
&\leq &||\rho ||_{C^{k}}||\pi _{n}(\psi )||_{1} \\
&\leq &||\rho ||_{C^{k}}||\psi ||_{1}
\end{eqnarray*}

and $Exi1.b$ is satisfied. We can then apply Theorem \ref{existence1}, and
get the existence of an invariant probability density $f$ in $BV.$ Since $%
\rho $ $\in C^{k}$ and $L_{\delta ,f}f=f$ we get that $f\in C^{k}$ for all $%
k\geq 1$, also proving $($\ref{1234}$).$

Now we apply Theorem \ref{existence} to get the uniqueness. In this case we
consider $B_{s}=C^{1}$ and $B_{w}=L^{1}$. \ We first have to prove that for $%
\delta $ small enough each operator $L_{\delta ,\phi }$ with $\phi \in P_{w}$
ha a unique invariant probability measure in $P_{w}.$ Since $L_{0}$ has
convergence to equilibrium, is regularizing and $C^{1}$ is compactly
immersed in $L^{1}$ it is standard to find that this operator has a unique
invariant probability measure. \ From the convergence to equilibrium, the
small perturbation assumption $(Exi2)$ we verified above and the
regularization property $(\ref{eq:regineqI})$ we get that there is $\gamma
\geq 0$ such that $L_{\delta ,\phi }$ has convergence to equilibrium for all 
$\delta \leq \gamma $ \ and $\phi \in P_{w}$. \ Indeed let us consider $f\in
V_{s}$ and suppose that by convergence to equilibrium $n$ is such that $%
||L_{0}^{n}f||_{L^{1}}\leq \frac{1}{2||\rho ||_{C^{1}}}||f||_{C^{1}}$ then 
\begin{equation*}
||L_{0}^{n+1}f||_{C^{1}}\leq ||\rho
||_{C^{1}}||L_{T_{0}}(L_{0}^{n}f)||_{L^{1}}\leq \frac{1}{2}||f||_{C^{1}}.
\end{equation*}

This implies that $L_{0}^{n+1}$ is a contraction of $V_{s}$. \ Now let us
consider $\phi \in P_{w},$ $\gamma \leq \frac{1}{4||\rho ||_{C^{1}}K(C+nB)}$
\ and $\delta \leq \gamma .$\ By a computation similar to the proof of Lemma %
\ref{XXX} \ (remark that $Q=1$ in the case $B_{w}=L^{1}$ and by $(\ref%
{LYnoise})$, $\lambda =0$) we can get%
\begin{eqnarray}
||L_{\delta ,\phi }^{n}g-L_{0}^{n}g||_{L^{1}} &\leq &\delta
K(C||g||_{C^{1}}+nB||g||_{L^{1}}) \\
&\leq &\delta K(C+nB)||g||_{C^{1}} \\
&\leq &\frac{1}{4||\rho ||_{C^{1}}}||g||_{C^{1}}
\end{eqnarray}%
and then $||L_{\delta ,\phi }^{n}g||_{L^{1}}\leq \frac{3}{4||\rho ||_{C^{1}}}%
||g||_{C^{1}}$, thus repeating the same reasoning as before $L_{\delta ,\phi
}^{n+1}$ also is a contraction of $V_{s}$. \ Hence we have that when $\delta 
$ is small enough each $L_{\delta ,\phi }$ has convergence to equilibrium.
It follow that $L_{\delta ,\phi }$ also has spectral gap on $C^{1}$ and on $%
L^{1}$ Indeed for all $f\in V_{L^{1}}:=\{f\in L^{1},\int f=0\}$ we get $%
||L_{\delta ,\phi }f||_{C^{1}}\leq ||\rho ||_{C^{1}}||f||_{L^{1}}$ and if we
have if $n$ is such that $||L_{\delta ,\phi }^{n}f||_{L^{1}}\leq \frac{3}{%
4||\rho ||_{C^{1}}}||f||_{C^{1}}$ then again 
\begin{equation*}
||L_{\delta ,\phi }^{n+1}f||_{L^{1}}\leq \frac{3}{4||\rho ||_{C^{1}}}%
||L_{\delta ,\phi }f||_{C^{1}}\leq \frac{3}{4}||f||_{L^{1}}.
\end{equation*}%
Thus each $L_{\delta ,\phi }$ has a unique invariant probability measure $%
f_{\phi }\in P_{w}$. Furthermore, the resolvent of $L_{\delta ,\phi }$ is
defined on $V_{L^{1}}$and its norm uniformly bounded for every $\phi \in
P_{w}$.

Now we can prove that\ $(Exi3)$ holds. Let us consider probability measures $%
\phi _{1}$ and $\phi _{2}\in P_{w}$ and the operators $L_{\delta ,\phi _{1}}$
$L_{\delta ,\phi _{2}}$ we have seen that when $\delta \leq \gamma $ \ these
operators have unique fixed probability densities $f_{\phi _{1}},f_{\phi
_{2}}$. We want to prove that 
\begin{equation*}
||f_{\phi _{1}}-f_{\phi _{2}}||_{L^{1}}\leq K_{2}||\phi _{1}-\phi
_{2}||_{L^{1}}.
\end{equation*}

We hence apply a construction similar to the proof of Theorem \ref%
{thm:linresp} to the family of operators $\hat{L}_{\epsilon
}:L^{1}\rightarrow L^{1}$ defined by%
\begin{equation*}
\hat{L}_{\epsilon }=L_{\delta ,\phi _{1}}+\epsilon \lbrack L_{\delta ,\phi
_{2}}-L_{\delta ,\phi _{1}}].
\end{equation*}

Consider%
\begin{eqnarray*}
(Id-L_{\delta ,\phi _{2}})(f_{\phi _{2}}-f_{\phi _{1}}) &=&f_{\phi
_{2}}-L_{\delta ,\phi _{2}}f_{\phi _{2}}-f_{\phi _{1}}+L_{\delta ,\phi
_{2}}f_{\phi _{1}} \\
&=&(L_{\delta ,\phi _{2}}-L_{\delta ,\phi _{1}})f_{\phi _{1}}.
\end{eqnarray*}%
We have that%
\begin{equation*}
(f_{\phi _{2}}-f_{\phi _{1}})=(Id-L_{\delta ,\phi _{2}})^{-1}(L_{\delta
,\phi _{2}}-L_{\delta ,\phi _{1}})f_{\phi _{1}}.
\end{equation*}%
By the fact that $(Id-L_{\delta ,\phi _{2}})^{-1}$ is well defined and
continuous on $V_{L^{1}}$ remarked before and by the fact that \ $||f_{\phi
_{1}}||_{C^{1}}\leq ||\rho ||_{C^{1}}$ and $(Exi2)$ we get 
\begin{equation*}
||f_{\phi _{2}}-f_{\phi _{1}}||_{L^{1}}\leq \delta K||\rho ||_{C^{1}}||\phi
_{2}-\phi _{1}||_{L^{1}}
\end{equation*}%
and then $(Exi3)$ is verified. Applying Theorem \ref{existence} we then get
the uniqueness for $\delta $ small enough.
\end{proof}

\begin{proposition}[Exponential convergence to equilibrium]
\label{connoise}Let $\mathcal{L}_{\delta }$ be the family of self-consistent
transfer operators arising from a map $T_{0}$, a kernel $\rho $ as above,
and $h\in C^{k}$ with $k\geq 1$. Suppose the uncoupled system $\mathcal{L}%
_{0}$ has convergence to equilibrium. Let $f_{\delta }$ be an invariant
probability measure of $\mathcal{L}_{\delta }.$ Then there exists $\overline{%
\delta }>0$ \ and $C,\gamma \geq 0$ such that for all $0<\delta <\overline{%
\delta }$ and each probability density\ $\nu \in C^{k}$ we have%
\begin{equation*}
||\mathcal{L}_{\delta }^{n}(\nu )-f_{\delta }||_{C^{k}}\leq Ce^{-\gamma
n}||\nu -f_{\delta }||_{C^{k}}.
\end{equation*}
\end{proposition}

\begin{proof}
Again the proof is a direct application of Theorem \ref{expco} to $\mathcal{L%
}_{\delta }$ \ considering the spaces $B_{ss}=C^{k+1},B_{s}=C^{k}$ and $%
B_{w}=L^{1}$. The assumption $(Con1)$ for this kind of systems is already \
verified in $(\ref{LYnoise})$. \ The assumption $(Con2)$ is as a \ direct
consequence of Lemma \ref{rimonta} and Lemma \ref{laquattro}. The assumption 
$(Con3)$ is required as an assumption in this statement.
\end{proof}

Let us now consider a linear response result for the invariant measure of $%
\mathcal{L}_{\delta }$ when $\delta \rightarrow 0$ in the case of coupled
maps with additive noise.

\begin{proposition}[Linear Response for coupled maps with additive noise]
\label{thm:linresp copy(2)} Let $\mathcal{L}_{\delta }$ be the family of
self-consistent transfer operators arising from a map $T_{0}$, a kernel $%
\rho $ and $h$ as above. Suppose \ the uncoupled initial transfer operator $%
\mathcal{L}_{0}$ has convergence to equilibrium in the sense of $\ (\ref%
{coneq})$. Let $f_{0}\in C^{k}$ be the unique invariant probability density
of $L_{0}$ and $f_{\delta }\in C^{k}$ be the invariant probability density
of $L_{\delta }$ (unique when $\delta $ is small enough). Then we have the
following Linear Response formula%
\begin{equation}
\lim_{\delta \rightarrow 0}\left\Vert \frac{f_{\delta }-f_{0}}{\delta }%
+(Id-L_{0})^{-1}\rho \ast
(L_{T_{0}}(h_{0})\int_{S^{1}}h(x,y)h_{0}(y)dy)^{\prime }\right\Vert
_{C^{k-1}}=0.  \label{11121}
\end{equation}
\end{proposition}

Before the proof, we recall some preliminary result on the response of
systems with additive noise. We now consider small perturbations of our
random maps with additive noise by composition with a map $D_{\delta }$,
which is when $\delta $ is small, a diffeomorphism near to the identity.
Consider a map $S\in C^{2}(\mathbb{S}^{1},\mathbb{R})$ and \ the map $%
D_{\delta }:\mathbb{S}^{1}\rightarrow \mathbb{S}^{1}$ defined by $D_{\delta
}=Id+\delta S${. \ \ }Let us consider then the perturbation of $T_{0}$ by
composition with $D_{\delta }$ defined by $T_{\delta }=T_{0}\circ D_{\delta
}.$ Starting from this family of maps and a kernel $\rho $ we can consider a
family of dynamical systems with additive noise as in $(\ref{eq:syswaddnoise}%
).$ Since $L_{T_{\delta }}=L_{D_{\delta }}\circ L_{T_{0}}$, to this system
we associate the annealed transfer operator defined by%
\begin{equation}
L_{\delta }:=\rho \ast L_{D_{\delta }\circ T_{0}}=\rho \ast (L_{D_{\delta
}}\circ L_{T_{0}}).  \label{lui}
\end{equation}

Now, in order to apply our general theorems it will be useful to consider
the derivative operator $\dot{L}$ for the family of operators $(L_{D_{\delta
}\circ T_{0}})_{\delta \in \lbrack 0,\overline{\delta }]}$. \ In this
direction, the following result (\cite[Theorem 24]{GS})\ gives some useful
estimates.

\begin{proposition}
\label{thm:Taylorexpfortransferop} Let $(L_{\delta })_{\delta \in \lbrack 0,%
\overline{\delta }]}$ be the family of transfer operators as described in $(%
\ref{lui})$.\ For all $k\geq 2$%
\begin{equation}
\lim_{\delta \rightarrow 0}\left\Vert \dfrac{L_{\delta }-L_{0}}{\delta }-%
\dot{L}\right\Vert _{C^{k}\rightarrow C^{k-1}}=0
\end{equation}%
where $\dot{L}:C^{k}(\mathbb{S}^{1})\rightarrow C^{k-1}(\mathbb{S}^{1})$ is
defined by: 
\begin{equation*}
\dot{L}(f)=-\rho \ast (S\cdot L_{T_{0}}f)^{\prime }.
\end{equation*}
\end{proposition}

\begin{proof}[Proof of Proposition \protect\ref{thm:linresp copy(2)}]
Similar to the proof of Proposition \ref{thm:linresp copy(1)}, the proof is
a direct application of Theorem \ref{thm:linresp}. We apply Theorem \ref%
{thm:linresp} to the family of operators $\mathcal{L}_{\delta }(\phi )=\rho
\ast \lbrack Q_{\delta ,\phi }(L_{T_{0}}(\phi ))]$ considering the spaces $%
B_{ss}=C^{k+2}(\mathbb{S}^{1})\subseteq B_{s}=C^{k+1}(\mathbb{S}^{1})\subset
B_{w}=C^{k}(\mathbb{S}^{1}).$ \ Let us now verify that the assumptions
needed to apply the theorem are satisfied. The assumption $(SS1)$
(regularity bounds), is implied by Proposition \ref{existence copy(1)}.

The assumption $(SS2)$(convergence to equilibrium for the unperturbed
operator), is supposed to hold. The assumption $(SS3)$ and the first part of
the assumption $(LR2)$ (small perturbation) follows from Lemma \ref{rimonta}
and \ \ref{laquattro}. \ Indeed we have to prove that there is $K\geq 0$
such that $\mathcal{L}_{0}-\mathcal{L}_{\delta }$ is $K\delta $-Lipschitz
when considered as a function $\overline{B}_{2M}\rightarrow B_{w}$ and $%
B_{2M}\rightarrow B_{s}$. In the first case we have to prove that for $\phi
_{1},\phi _{2}\in \{\phi \in C^{k+1},||\phi ||_{C^{k+1}}\leq 2M\}$%
\begin{equation*}
||\rho \ast \lbrack Q_{\delta ,\phi _{1}}(L_{T_{0}}(\phi
_{1}))-L_{T_{0}}(\phi _{1})]-\rho \ast \lbrack Q_{\delta ,\phi
_{2}}(L_{T_{0}}(\phi _{2}))-L_{T_{0}}(\phi _{2})]||_{C^{k}}\leq K\delta
||\phi _{1}-\phi _{2}||_{C^{k+1}}.
\end{equation*}

We have 
\begin{eqnarray*}
&&||\rho \ast \lbrack Q_{\delta ,\phi _{1}}(L_{T_{0}}(\phi
_{1}))-L_{T_{0}}(\phi _{1})]-\rho \ast \lbrack Q_{\delta ,\phi
_{2}}(L_{T_{0}}(\phi _{2}))-L_{T_{0}}(\phi _{2})]||_{C^{k}} \\
&\leq &||\rho \ast Q_{\delta ,\phi _{1}}(L_{T_{0}}(\phi _{1}))-\rho \ast
L_{T_{0}}(\phi _{1})-\rho \ast Q_{\delta ,\phi _{1}}(L_{T_{0}}(\phi _{2})) \\
&&+\rho \ast Q_{\delta ,\phi _{1}}(L_{T_{0}}(\phi _{2}))-\rho \ast Q_{\delta
,\phi _{2}}(L_{T_{0}}(\phi _{2}))+\rho \ast L_{T_{0}}(\phi _{2})||_{C^{k}}
\end{eqnarray*}

and 
\begin{eqnarray*}
&&||\rho \ast Q_{\delta ,\phi _{1}}(L_{T_{0}}(\phi _{1}))-\rho \ast
L_{T_{0}}(\phi _{1})-\rho \ast Q_{\delta ,\phi _{1}}(L_{T_{0}}(\phi
_{2}))+\rho \ast L_{T_{0}}(\phi _{2})||_{C^{k}} \\
&\leq &||\rho \ast Q_{\delta ,\phi _{1}}(L_{T_{0}}(\phi _{1}-\phi
_{2}))-\rho \ast L_{T_{0}}(\phi _{1}-\phi _{2})||_{C^{k}} \\
&\leq &\delta CK_{1}2M||\phi _{1}-\phi _{2}||_{C^{k+1}}
\end{eqnarray*}%
applying Lemma \ref{rimonta} with $\psi =0$ and $\phi =\phi _{1}$ and Lemma
\ \ref{laquattro}.

Furthermore%
\begin{equation*}
||\rho \ast Q_{\delta ,\phi _{1}}(L_{T_{0}}(\phi _{2}))-\rho \ast Q_{\delta
,\phi _{2}}(L_{T_{0}}(\phi _{2}))||_{C^{k}}\leq \delta CK_{1}2M||\phi
_{1}-\phi _{2}||_{C^{k+1}}
\end{equation*}

again by Lemma \ref{rimonta} and Lemma \ \ref{laquattro}.

In the second case we have to prove that for $\phi _{1},\phi _{2}\in \{\phi
\in C^{k+2},||\phi ||_{C^{k+2}}\leq 2M\}$ 
\begin{equation*}
||\rho \ast \lbrack Q_{\delta ,\phi _{1}}(L_{0}(\phi _{1}))-L_{0}(\phi
_{1})]-\rho \ast \lbrack Q_{\delta ,\phi _{2}}(L_{0}(\phi _{2}))-L_{0}(\phi
_{2})]||_{C^{k+1}}\leq k\delta ||\phi _{1}-\phi _{2}||_{C^{k+2}}
\end{equation*}%
which can be proved similarly as before.

The assumption $(LR1)$ (resolvent of the unperturbed operator) follows from
Proposition \ref{uuuu}. The second part of assumption $(LR2)$ (derivative
operator) follows from Proposition \ref{thm:Taylorexpfortransferop} in a way
similar to what is done in the proof of Proposition \ref{thm:linresp copy(1)}%
. Applying Theorem \ref{thm:linresp}, we then get%
\begin{equation*}
\lim_{\delta \rightarrow 0}\left\Vert \frac{h_{\delta }-h_{0}}{\delta }-(Id-%
\mathcal{L}_{0})^{-1}\mathcal{\dot{L}}h_{0}\right\Vert _{C^{k}}=0.
\end{equation*}%
We can now let the formula be more explicit by finding an expression for $%
\mathcal{\dot{L}}.$ In our case 
\begin{equation*}
S(x)=\int_{S^{1}}h(x,y)h_{0}(y)dy
\end{equation*}%
and then%
\begin{equation*}
\mathcal{\dot{L}}(h_{0})=\rho \ast (-h_{0}S)^{\prime }=\rho \ast
(-L_{T_{0}}(h_{0})\int_{S^{1}}h(x,y)h_{0}(y)dy)^{\prime }.
\end{equation*}
\end{proof}

\begin{remark}
\label{rmk37}The convergence to equilibrium assumption $(\ref{coneq})$
required in\ Proposition \ref{connoise} and Proposition \ref{thm:linresp
copy(2)}\ for the uncoupled transfer operator $\mathcal{L}_{0}$\ is easy to
be verified in many examples of systems whose deterministic part has some
kind of topological mixing and the noise is distributed by a smooth kernel
or it has large support in some sense, see \cite[ Corollary 5.7.1]{LM}, \cite%
[Lemma 41]{GG} or \cite[Remarks 2.3 and 6.4]{AFG}. In more complicated
situations it can be also verified by computer aided estimates (\cite{GMN}).
\end{remark}

\begin{remark}
\label{rmk2}We remark that another meaningful definition for the transfer
operator associated with a family of random maps coupled in by their mean
field could be the following%
\begin{equation}
\mathcal{L}_{\delta }(\phi )=[Q_{\delta ,\phi }(\rho \ast L_{T_{0}}(\phi ))].
\label{lali}
\end{equation}%
In this case one applies the coupling directly to the annealed transfer
operator of the random maps. \ Here for small $\delta $ the application of
our theory seems to be possible by estimates similar to the ones shown in
this section and in Section \ref{secmap}. Indeed the transfer operator
realizing the coupling $Q_{\delta ,\phi }$ is applied after the convolution.
Considering $\phi \in L^{1}$, by $(\ref{eq:regineqI})$ we get that $(\rho
\ast L_{T_{0}}(\phi ))$ is regularized to the regularity of the kernel $\rho 
$. If $h$ is smooth enough and $\delta $ small enough, this regularity is
preserved by the application of $Q_{\delta ,\phi }$ leading to the
verification of regularity properties like $Exi1$, $Exi1.b,~Con1$ and $SS1$.

The verification of small perturbation properties like $Exi2$, $Con2$, $SS3$
and $LR2$ \ for the family of transfer operators associated with $(\ref{lali}%
)$ $L_{\delta ,\phi }=Q_{\delta ,\phi }(\rho \ast L_{T_{0}})$ relies on \
the estimation of the distance of $Q_{\delta ,\phi _{1}}$ from $Q_{\delta
,\phi _{2}}$ on a mixed norm which can be done in a way similar to the use
of Lemma \ref{launo} and Lemma \ref{rimonta} as done in Section \ref{secmap}%
. The form of the derivative operator $\dot{L}$ is probably similar to the
one given at Proposition \ref{thm:regdettransferop}.
\end{remark}

\section{Self-consistent operators not coming from coupled map networks\label%
{strange}}

In this section we consider a class of self-consistent transfer operators
not coming from networks of coupled maps, giving other examples of
application of our general theory. The systems considered are inspired to
some examples studied in \cite{Se2} and \cite{Bla}, where we have a map
whose slope depends on the average of a certain observable\ during the
iterates. We add noise to simplify the functional analytic properties of the
system. Let us consider again a family of random maps on $[0,1]$ depending
on a probability measure $\mu $ and on a parameter $\delta \geq 0.$

Let us consider the classical tent map $T:[0,1]\rightarrow \lbrack 0,1]$,
defined by $T(x)=\min (2x,2-2x)$, the family of maps \ $T_{\delta ,\mu
}:[0,1]\rightarrow \lbrack 0,1]$ we consider as a self-consistent
perturbation of the tent map are defined by%
\begin{equation*}
T_{\delta ,\mu }(x)=\frac{T(x)}{1+\delta \int xd\mu }.
\end{equation*}%
Then adding a noise-like perturbation to the map $T_{\delta ,\mu }$ we
consider the process $(X_{n})_{n\in \mathbb{N}}$ defined on $[0,1]$ by%
\begin{equation}
X_{n+1}=T_{\delta ,\mu }(X_{n})\hat{+}\Omega _{n}\mod 1  \label{lei}
\end{equation}%
where $(\Omega _{n})_{n\in \mathbb{N}}$ are i.i.d random variables
distributed according to a kernel $\rho \in Lip(\mathbb{R}),$ supported on $%
[-1,1]$ with Lipschitz constant $L$ and where $\hat{+}$ is the
\textquotedblleft boundary reflecting" sum, defined by $a\hat{+}b:=\pi (a+b)$%
, and $\pi :\mathbb{R}\rightarrow \lbrack 0,1]$ is the piecewise linear map $%
\pi (x)=\min_{i\in \mathbb{Z}}|x-2i|$. \ This is a model of a system on $%
[0,1]$ with reflecting boundary conditions. When the noise sends a point
outside the space $[0,1]$ the projection $\pi $ is applied to let the image
of the map again in $[0,1].$ Let us denote as $L_{\pi }$ the transfer
operator $L_{\pi }:L^{1}(\mathbb{R})\rightarrow L^{1}([0,1])$ associated
with the map $\pi .$ Let $b\in \mathbb{R}$ and $g\in Lip(\mathbb{R})$
consider the translation operator $\tau _{b}$ defined by $(\tau
_{b}~g)(y):=g(y+b)$.

The annealed transfer operator associated with the random dynamical system $(%
\ref{lei})$ is a Markov operator and is given by the following kernel
operator (for details see \cite{AFG}, Section 6):%
\begin{equation}
L_{\delta ,\mu }f(x)=\int k_{\delta ,\mu }(x,y)f(y)dy,  \label{kernelL2}
\end{equation}%
where 
\begin{equation}
k_{\delta ,\mu }(x,y)=(L_{\pi }\tau _{-T_{\delta ,\mu }(y)}\rho )(x)
\label{eq:ker-add-noise-explicit}
\end{equation}%
and $x,y\in \lbrack 0,1]$. Since the perturbation induced on the system with
additive noise by increasing the parameter $\delta $ is not coming from the
composition with a diffeomorphism we cannot use the estimates from the
previous sections directly. We hence take a slightly different point of view
on systems with additive noise, and related basic estimates which were
developed in \cite{AFG}.

In this case we will consider $B_{w}=L^{2}[0,1].$ Let $P_{w}$ be the set of
measures having a probability density in $L^{2}$. The nonlinear
self-consistent operator we consider in this case hence is given by $%
\mathcal{L}_{\delta }:P_{w}\rightarrow P_{w}$ defined as%
\begin{equation}
\mathcal{L}_{\delta }\mu =L_{\delta ,\mu }\mu  \label{strangeop}
\end{equation}%
for all $\mu \in P_{w}.$ We remark that since $\rho \in Lip(\mathbb{R})$ and
it is supported on $[-1,1]$ the kernel of this operator is bounded: $%
k_{\delta ,\mu }\in L^{\infty }([0,1]^{2})$. Let us recall some classical
and useful facts about kernel operators.

\begin{itemize}
\item If $k_{\delta ,\mu }\in L^{\infty }([0,1]^{2})$, then 
\begin{equation}
||L_{\delta ,\mu }f||_{\infty }\leq ||k_{\delta ,\mu }||_{L^{\infty
}([0,1]^{2})}||f||_{1}  \label{KF2}
\end{equation}%
and the operator $L_{\delta ,\mu }:L^{1}\rightarrow L^{\infty }$ is bounded.
Furthermore, $\Vert L_{\delta ,\mu }\Vert _{L^{p}\rightarrow L^{\infty
}}\leq \Vert k_{\delta ,\mu }\Vert _{L^{\infty }([0,1]^{2})}$ for $1\leq
p\leq \infty $.

\item The operator $L_{\delta ,\mu }:L^{2}\rightarrow L^{2}$ is compact and 
\begin{equation}
||L_{\delta ,\mu }f||_{2}\leq ||k_{\delta ,\mu
}||_{L^{2}([0,1]^{2})}||f||_{2}  \label{KF}
\end{equation}%
(see \cite[Proposition 4.7]{C} \ or \cite{KF}).
\end{itemize}

It is also well known that these Markov operators have invariant probability
densities in $L^{2}$ (see e.g. \cite[Theorem 2.2]{AFG}). \ Since $k_{\delta
,\mu }\in L^{\infty }([0,1]^{2}),$ by $($\ref{KF2}$)$ we also have that any
invariant probability density $f_{\delta ,\mu }$ for this operator satisfies%
\begin{equation}
||f_{\delta ,\mu }||_{\infty }\leq ||k_{\delta ,\mu }||_{L^{\infty
}([0,1]^{2})}\leq ||\rho ||_{L^{\infty }[0,1]}.  \label{sl2}
\end{equation}%
\ In \cite[Section 6]{AFG} \ the following estimates are proved for such
kernel operators coming from maps with additive noise and reflecting
boundaries conditions (see Proposition 6.2):

\begin{proposition}
\label{kdottt}Assume that $k_{\delta ,\mu }$ is the kernel associated to the
transfer operator of a system with additive noise and reflecting boundaries
composed by a map $T_{\delta ,\mu }$ and a noise kernel $\rho $ (see $($\ref%
{eq:ker-add-noise-explicit}$)$ ). Let us fix $\delta ,$ suppose that the
family of interval maps $\{T_{\epsilon }\}_{\epsilon \in \lbrack 0,\overline{%
\epsilon })}$ satisfies%
\begin{eqnarray}
T_{0} &=&T_{\delta ,\mu }  \label{mub} \\
T_{\epsilon } &=&T_{0}+\epsilon \cdot \dot{T}+T_{\epsilon },  \notag
\end{eqnarray}%
where $\dot{T},~T_{\epsilon }\in L^{2}$ and $\Vert T_{\epsilon }\Vert
_{2}=o(\epsilon )$. Consider the transfer operator $L_{0}$ associated with
the unperturbed system with map $T_{0}$ \ and kernel $\rho $. Let $%
L_{\epsilon }$ be the transfer operator associated with the system driven by 
$T_{\epsilon }$ and kernel $\rho $. Then there are $\overline{\epsilon }%
,K\geq 0$ such that for $\epsilon \in \lbrack 0,\overline{\epsilon })$%
\begin{equation}
||L_{0}-L_{\epsilon }||_{L^{2}\rightarrow L^{2}}\leq \epsilon K  \label{kd2}
\end{equation}

and $\forall f_{0}\in L^{2}$%
\begin{equation}
\lim_{\epsilon \rightarrow 0}\frac{L_{\epsilon }-L_{0}}{\epsilon }%
f_{0}=-\int_{0}^{1}\left( L_{\pi }\left( \tau _{-T_{0}(y)}\frac{d\rho }{dx}%
\right) \right) (x)\dot{T}(y)f_{0}(y)dy,  \label{kde}
\end{equation}%
with convergence in $L^{2}.$
\end{proposition}

The inequality $(\ref{kd2})$ shows that the perturbations we are interested
in applying to the transfer operators associated with this kind of systems
are small perturbations in the $L^{2}\rightarrow L^{2}$ topology. We will
then consider the transfer operators associated with this kind of systems as
operators acting on $L^{2}[0,1]$ and in this subsection we will apply our
general statements with the choice $B_{ss}=B_{s}=B_{w}=L^{2}[0,1].$ \ We now
can apply Theorem \ref{existence} and prove

\begin{proposition}
\label{exist3}Let $\mathcal{L}_{\delta }$ be the self-consistent transfer
operator associated to $T_{\delta ,\mu }$ and $\rho $ as defined in $(\ref%
{strangeop})$. There are $M,\overline{\delta }\geq 0$ such that for all $%
\delta \in \lbrack 0,\overline{\delta }]$ there is a unique $f_{\delta }\in
P_{w}$ with $||f_{\delta }||_{L^{2}}\leq M$ such that%
\begin{equation*}
\mathcal{L}_{\delta }(f_{\delta })=f_{\delta }.
\end{equation*}
\end{proposition}

Before the proof of Proposition \ref{exist3} we need a couple of preliminary
results

\begin{proposition}
\label{delta}There is $C\geq 0$ such that for all $\mu _{1},$ $\mu _{2}\in
P_{w}\subseteq L^{2}$, 
\begin{eqnarray*}
||L_{\delta ,\mu _{2}}-L_{\delta ,\mu _{1}}||_{L^{2}\rightarrow L^{2}} &\leq
&\delta C||\mu _{1}-\mu _{2}||_{L^{2}} \\
||L_{0,\mu _{1}}-L_{\delta ,\mu _{1}}||_{L^{2}\rightarrow L^{2}} &\leq
&\delta C||\mu _{1}||_{L^{2}}
\end{eqnarray*}
\end{proposition}

\begin{proof}
The proof follows by $($\ref{KF}$)$, estimating the difference of the
associated kernels. Let us first consider $||L_{\delta ,\mu _{2}}-L_{\delta
,\mu _{1}}||_{L^{2}\rightarrow L^{2}}$. We have 
\begin{eqnarray*}
||L_{\delta \mu _{2}}-L_{\delta \mu _{1}}||_{L^{2}\rightarrow L^{2}} &\leq
&||k_{\delta ,\mu _{1}}-k_{\delta ,\mu _{2}}||_{L^{2}([0,1]^{2})} \\
&=&(\int_{[0,1]^{2}}(k_{\delta ,\mu _{1}}(x,y)-k_{\delta ,\mu
_{2}}(x,y))^{2}dxdy)^{\frac{1}{2}}.
\end{eqnarray*}

We first estimate the distance between the two deterministic parts of the
dynamics. For all $y\in \lbrack 0,1]$ we get%
\begin{eqnarray*}
|T_{\delta ,\mu _{1}}(y)-T_{\delta ,\mu _{2}}(y)| &\leq &|\frac{T(y)}{%
1+\delta \int xd\mu _{1}(x)}-\frac{T(y)}{1+\delta \int xd\mu _{2}(x)}| \\
&\leq &|\frac{T(y)(1+\delta \int xd\mu _{2}(x))-T(y)(1+\delta \int xd\mu
_{1}(x))}{(1+\delta \int xd\mu _{1}(x))(1+\delta \int xd\mu _{2}(x))}| \\
&\leq &\delta |\int xd\mu _{2}(x)-\int xd\mu _{1}(x)| \\
&\leq &\delta ||\mu _{1}-\mu _{2}||_{L^{2}}.
\end{eqnarray*}

Now let us suppose that $\rho $ is $L-$Lipschitz. Since $\rho $ is supported
in $[-1,1]$ we get that $\tau _{-T_{\delta ,\mu }(y)}\rho (x)$ is supported
in $[-2,1]$ for each $\mu $. By this \ $L_{\pi }(\tau _{-T_{\delta ,\mu
}(y)}\rho (x))$ is the sum of at most three non zero contributions for each $%
x\in \lbrack 0,1],$ hence%
\begin{eqnarray*}
|[k_{\delta ,\mu _{2}}-k_{\delta ,\mu _{1}}](x,y)| &\leq &|L_{\pi }[\tau
_{-T_{\delta ,\mu _{2}}(y)}\rho (x)-\tau _{-T_{\delta ,\mu _{1}}(y)}\rho
(x)]| \\
&\leq &3\sup_{x\in \lbrack -1,2],y\in \lbrack 0,1]}|[\tau _{-T_{\delta ,\mu
_{2}}(y)}\rho (x)-\tau _{-T_{\delta ,\mu _{1}}(y)}\rho (x)]| \\
&\leq &3\delta L~||\mu _{1}-\mu _{2}||_{L^{2}}
\end{eqnarray*}%
proving the statement. \ The estimate for $||L_{0,\mu _{1}}-L_{\delta ,\mu
_{1}}||_{L^{2}\rightarrow L^{2}}$ is similar. We have%
\begin{eqnarray*}
|T_{\delta ,\mu _{1}}(y)-T_{0,\mu _{1}}(y)| &\leq &|\frac{T(y)}{1+\delta
\int xd\mu _{1}(x)}-T(y)| \\
&\leq &|\frac{T(y)(1+\delta \int xd\mu _{2}(x))-T(y)}{(1+\delta \int xd\mu
_{1}(x))}| \\
&\leq &\delta |\int xd\mu _{1}(x)| \\
&\leq &\delta ||\mu _{1}||_{L^{2}}
\end{eqnarray*}%
and the statement is obtained as before.
\end{proof}

\begin{proposition}
\label{corrr}Let us consider a self-consistent operator $\mathcal{L}_{\delta
}$ as defined at beginning of Section \ref{strange}. Consider $%
V_{L^{2}}:=\{v\in L^{2},$~$\int vdm=0\}.$ Suppose that there is $n$ such
that for each $v\in V_{L^{2}}$, $||\mathcal{L}_{0}^{n}(v)||_{L^{2}}\leq 
\frac{1}{2}||v||_{L^{2}}$. Then there are $K,\overline{\delta }\geq 0$ such
that for every $\delta \in \lbrack 0,\overline{\delta }),$ and probability
measures $\mu _{1},\mu _{2}\in L^{2}$ satisfying $||\mu _{i}||_{L^{2}}\leq
||\rho ||_{\infty }$ for all $i\in \{1,2\}$ it holds that \ $L_{\delta ,\mu
_{i}}$ has a\ unique invariant probability measure with density in $L^{2}$
which we denote by $f_{\mu _{i}}.$ Furthermore, with these notations%
\begin{equation*}
||f_{\mu _{1}}-f_{\mu _{2}}||_{L^{2}}\leq \delta K||\mu _{1}-\mu
_{2}||_{L^{2}}.
\end{equation*}
\end{proposition}

\begin{proof}
We sketch the proof, which is similar to part of the proof of Proposition %
\ref{existence copy(1)}. Since $||L_{0}^{n}(v)||_{L^{2}}\leq \frac{1}{2}%
||v||_{L^{2}},$ by Proposition \ref{delta} for $\delta $ small enough we
have $||L_{\delta ,\mu _{1}}^{n}(v)||_{L^{2}}\leq \frac{3}{4}||v||_{L^{2}}$
for all $\mu _{1}$ with $||\mu _{1}||_{L^{2}}\leq ||\rho ||_{\infty }$ and $%
v\in V_{L^{2}}$, impliying the uniqueness of the invariant probability
density in $L^{2}$. By this we can also define the resolvent for each such
operator $L_{\delta ,\mu _{1}}$ on $V_{L^{2}}$ with a uniform bound on its $%
L^{2}$ norm. Since%
\begin{eqnarray*}
(Id-L_{\delta ,\mu _{2}})(f_{\mu _{2}}-f_{\mu _{1}}) &=&f_{\mu
_{2}}-L_{\delta ,\mu _{2}}f_{\mu _{2}}-f_{\mu _{1}}+L_{\delta ,\mu
_{2}}f_{\mu _{1}} \\
&=&(L_{\delta ,\mu _{2}}-L_{\delta ,\mu _{1}})f_{\mu _{1}}.
\end{eqnarray*}%
We have that%
\begin{equation*}
(f_{\mu _{2}}-f_{\mu _{1}})=(Id-L_{\delta ,\mu _{2}})^{-1}(L_{\delta ,\mu
_{2}}-L_{\delta ,\mu _{1}})f_{\mu _{1}}.
\end{equation*}

Since $||(Id-L_{\delta ,\mu _{2}})^{-1}||_{L^{2}\rightarrow L^{2}}$ is
uniformly bounded and $||f_{\mu _{1}}||_{L^{2}}\leq ||f_{\mu
_{1}}||_{L^{\infty }}\leq ||\rho ||_{\infty }$, applying Proposition \ref%
{delta} we get the statement.
\end{proof}

\begin{proof}[Proof of Proposition \protect\ref{exist3}]
We apply Theorem \ref{existence} \ with $B_{s}=B_{w}=L^{2}[0,1].$ The
assumption $(Exi1)$ is provided by $(\ref{sl2}),$ $(Exi2)$ is provided by
Proposition \ref{delta}, $(\ref{kd2})$ and $(Exi3)$ is provided by
Proposition \ref{corrr}. The unique fixed point we find is in $L^{2}$ and
since the kernels we consider in the construction are uniformly bounded in $%
L^{\infty }$ the $\ L^{2}$ norm of the fixed point is uniformly bounded as $%
\delta $ varies.
\end{proof}

Now we can prove the linear response formula for these self-consistent
operators in the small nonlinear perturbation regime.

\begin{proposition}[Linear response]
\label{thm:linrespx} Consider the family of self-consistent transfer
operators $\mathcal{L}_{\delta }:L^{2}\rightarrow L^{2}$ as described before
with $\delta \in (0,\overline{\delta })$ as found in Proposition \ref{exist3}
and with invariant probability measures $f_{\delta }$. We have the following
Linear Response formula%
\begin{equation}
\lim_{\delta \rightarrow 0}\frac{f_{\delta }-f_{0}}{\delta }%
=(Id-L_{0})^{-1}\int_{0}^{1}\left( L_{\pi }\left( \tau _{-T_{0}(y)}\frac{%
d\rho }{dx}\right) \right) (x)aT(y)f_{0}(y)dy,
\end{equation}%
where $a=\int tdf_{0}(t)$ and the limit is converging in $L^{2}.$
\end{proposition}

\begin{proof}
The proof is an application of Theorem \ref{thm:linresp} where the spaces
considered in this case are $B_{ss}=B_{s}=B_{w}=L^{2}$. The assumption $%
(SS1) $ (regularity bounds), is implied by Proposition \ref{exist3}. Let us
remark that the unperturbed system is a noisy tent map, hence it has
convergence to equilibrium (by \cite[Remarks 6.4]{AFG}) and the assumption $%
(SS2)$ is satisfied. To verify assumption $(SS3)$ (small perturbation) we
need to verify that, considering $B_{2M}=\{x\in L^{2},||x||\leq 2M\}$. There
is $K\geq 0$ such that and $\mathcal{L}_{0}-\mathcal{L}_{\delta
}:B_{2M}\rightarrow L^{2}$ is $K\delta $-Lipschitz. We have to verify that \
for all $\mu _{1},\mu _{2}\in B_{2M}$%
\begin{equation}
||(\mathcal{L}_{\delta }-\mathcal{L}_{0})\mu _{1}-(\mathcal{L}_{\delta }-%
\mathcal{L}_{0})\mu _{2}||_{L^{2}}\leq K\delta ||\mu _{1}-\mu _{2}||_{L^{2}}.
\label{123}
\end{equation}

Recalling that by Proposition \ref{delta} \ we have $L_{0,\mu _{1}}=L_{0,\mu
_{2}}:=L_{0},$ we have 
\begin{eqnarray*}
(\mathcal{L}_{\delta }-\mathcal{L}_{0})\mu _{1}-(\mathcal{L}_{\delta }-%
\mathcal{L}_{0})\mu _{2} &=&L_{\delta ,\mu _{1}}\mu _{1}-L_{0,\mu _{1}}\mu
_{1}-L_{\delta ,\mu _{2}}\mu _{2}+L_{0,\mu _{2}}\mu _{2} \\
&=&L_{\delta ,\mu _{1}}\mu _{1}-L_{\delta ,\mu _{1}}\mu _{2}+L_{\delta ,\mu
_{1}}\mu _{2}-L_{0,\mu _{1}}\mu _{1}-L_{\delta ,\mu _{2}}\mu _{2}+L_{0,\mu
_{2}}\mu _{2} \\
&=&[L_{\delta ,\mu _{1}}-L_{0}](\mu _{1}-\mu _{2})+[L_{\delta ,\mu
_{1}}-L_{\delta ,\mu _{2}}]\mu _{2}
\end{eqnarray*}

Now by Proposition \ref{delta} 
\begin{equation*}
||[L_{\delta ,\mu _{1}}-L_{0}](\mu _{1}-\mu _{2})||_{L^{2}}\leq \delta
C||\mu _{1}||_{L^{2}}||\mu _{1}-\mu _{2}||_{L^{2}}
\end{equation*}%
and%
\begin{equation*}
||[L_{\delta ,\mu _{1}}-L_{\delta ,\mu _{2}}]\mu _{2}||_{L^{2}}\leq \delta
C||\mu _{2}||_{L^{2}}||\mu _{1}-\mu _{2}||_{L^{2}}
\end{equation*}

proving the statement. Now we can apply Theorem \ref{ss} and deduce that $%
f_{\delta }\rightarrow f_{0}$ in $L^{2}.$ The assumption $(LR1)$ on the
existence of the resolvent is equivalent to $(SS2)$ since we consider only
one space $L^{2}$ and for the same reason the first part of $(LR2)$is
equivalent to $(SS3).$ We now only need to compute the derivative operator.
When the self-consistent operator is considered, as $\delta $ increases, the
effect of the perturbation on the system is only on the map defining the
deterministic part of the dynamics. We\ then use $(\ref{kde})$, from
Proposition \ref{kdottt}. We remark that this perturbation on the
deterministic part of the dynamics depends on the invariant measure \ $%
f_{\delta }$ of the system as $\delta $ changes, however we will see that
since $f_{\delta }\rightarrow f_{0}$ in $L^{2}$ this also give rise to a
family of maps with additive noise of the type 
\begin{equation*}
T_{\delta }=T_{0}+\delta \cdot \dot{T}+t_{\delta }
\end{equation*}%
as in $(\ref{mub}).$ Indeed let us compute $\dot{T}$ in this case.
Considering that $T_{\delta ,f_{\delta }}(x)=\frac{T(x)}{1+\delta \int
xdf_{\delta }}$ we get 
\begin{eqnarray*}
\frac{T_{\delta }-T_{0}}{\delta } &=&\frac{1}{\delta }[\frac{T(x)}{1+\delta
\int xdf_{\delta }}-T(x)] \\
&=&T(x)\frac{-\delta \int xdf_{\delta }}{\delta +\delta ^{2}\int xdf_{\delta
}} \\
&=&-T(x)\frac{\int xdf_{\delta }}{1+\delta \int xdf_{\delta }}.
\end{eqnarray*}

Since $f_{\delta }\rightarrow f_{0}$ in $L^{2}$ hence%
\begin{equation*}
\dot{T}=\lim_{\delta \rightarrow 0}\frac{T_{\delta }-T_{0}}{\delta }%
=-T(x)\int xdf_{0}
\end{equation*}%
and we have the expression for $\dot{L}f_{0}$ from $(\ref{kde}).$ Applying
Theorem \ref{thm:linresp} then, we then get%
\begin{equation*}
\lim_{\delta \rightarrow 0}\frac{f_{\delta }-f_{0}}{\delta }%
=(Id-L_{0})^{-1}\int_{0}^{1}\left( L_{\pi }\left( \tau _{-T_{0}(y)}\frac{%
d\rho }{dx}\right) \right) (x)aT(y)f_{0}(y)dy
\end{equation*}%
where $a=\int tdf_{0}(t).$
\end{proof}

\section{Coupling different maps\label{diff}}

In this section we show how one can use a self-consistent transfer operator
approach as a model for the behavior of networks of coupled maps of
different types. We will see that our general theoretical framework
naturally includes this case. For simplicity we will consider only two types
of maps, also for simplicity we will consider coupled expanding maps on the
circle. Let us consider two different $C^{6}$ expanding maps of the circle $%
(T_{1},S^{1}),$ $(T_{2},S^{1}).$ Given two probability densities $\psi
_{1},\psi _{2}\in L^{1}(\mathbb{S}^{1},\mathbb{R})$ representing the
distribution of probability of the states in the two systems, two coupling
functions $h_{1},h_{2}\in C^{6}(\mathbb{S}^{1}\times \mathbb{S}^{1},\mathbb{R%
})$ and $\delta \in \lbrack 0,\epsilon _{0}]$ \ representing the way in
which these distributions perturb the dynamics (which can be different for
the two different systems). Let us define \ $\Phi _{\delta ,\psi _{1},\psi
_{2}}:\mathbb{S}^{1}\rightarrow \mathbb{S}^{1}$ with $i\in \{1,2\}$ as%
\begin{equation*}
\Phi _{\delta ,\psi _{1},\psi _{2}}(x)=x+\pi _{\mathbb{S}^{1}}(\delta \int_{%
\mathbb{S}^{1}}h_{1}(x,y)\psi _{1}(y)dy+\delta \int_{\mathbb{S}%
^{1}}h_{2}(x,y)\psi _{2}(y)dy)
\end{equation*}%
(here for simplicity we suppose that the diffeomorphism perturbing the two
different maps is the same though with different contributions for the two
different maps, but one can consider different ways to define $\Phi _{\delta
,\psi _{i}}$ for each map) the maps will hence be perturbed by the combined
action of the two densities $\psi _{1}$ and $\psi _{2}$.

Again we assume $\epsilon _{0}$ is so small that $\Phi _{\delta ,\psi
_{1},\psi _{2}}$ is a diffeomorphism \ for all $\delta \in \lbrack
0,\epsilon _{0}]$\ and $\Phi _{\delta ,\psi _{1},\psi _{2}}^{\prime }>0.$
Denote by $Q_{\delta ,\psi _{1},\psi _{2}}$ the transfer operator associated
with $\Phi _{\delta ,\psi _{1},\psi _{2}},$ defined as 
\begin{equation*}
\lbrack Q_{\delta ,\psi _{1},\psi _{2}}(\phi )](x)=\frac{\phi (\Phi _{\delta
,\psi _{1},\psi _{2}}^{-1}(x))}{|\Phi _{\delta ,\delta ,\psi _{1},\psi
_{2}}^{\prime }(\Phi _{\delta ,\delta ,\psi _{1},\psi _{2}}^{-1}(x))|}
\end{equation*}%
for any $\phi \in L^{1}(S^{1},\mathbb{R}).$ Now we consider the action of
the two maps by considering a global system $(S^{1}\times S^{1},F_{\delta
,\psi _{1},\psi _{2}})$ with 
\begin{equation*}
F_{\delta ,\psi _{1},\psi _{2}}(x_{1},x_{2})=(\Phi _{\delta ,\psi _{1},\psi
_{2}}\circ T_{1}(x_{1}),\Phi _{\delta ,\psi _{1},\psi _{2}}\circ
T_{2}(x_{2})).
\end{equation*}

Finally let us consider the space of functions $B_{1}:=\{(f_{1},f_{2})\in
L^{1}(\mathbb{S}^{1})\times L^{1}(\mathbb{S}^{1})\}$ with the norm $%
||(f_{1},f_{2})||_{B_{1}}=||f_{1}||_{L^{1}}+||f_{2}||_{L^{1}}$ (this is also
called the direct sum $L^{1}(\mathbb{S}^{1})\oplus L^{1}(\mathbb{S}^{1})$),
the space $P_{1}$ of probability elements in $B_{1},$ $P_{1}=$ $%
\{(f_{1},f_{2})\in B_{1}~s.t.\forall i\in \{1,2\},~f_{i}\geq 0,\int
f_{i}dm=1\}$ and the stronger spaces $B_{2}:=\{(f_{1},f_{2})\in W^{1,1}(%
\mathbb{S}^{1})\times W^{1,1}(\mathbb{S}^{1})\}$ with the norm $%
||(f_{1},f_{2})||_{B_{1}}=||f_{1}||_{W^{1,1}}+||f_{2}||_{W^{1,1}}$ and $%
B_{3}:=\{(f_{1},f_{2})\in W^{2,1}(\mathbb{S}^{1})\times W^{2,1}(\mathbb{S}%
^{1})\}$ with the norm $%
||(f_{1},f_{2})||_{B_{1}}=||f_{1}||_{W^{2,1}}+||f_{2}||_{W^{2,1}}$ (again
direct sums of Sobolev spaces). These sets can be trivially endowed with a
structure of normed vector space. Coherently with the previous sections we
define a family of transfer operators $L_{\delta ,\phi _{1},\phi
_{2}}:B_{w}\rightarrow B_{w}$ depending on elements of the weaker space $%
(\phi _{1},\phi _{2})\in P_{1}$ as%
\begin{equation}
L_{\delta ,\phi _{1},\phi _{2}}((f_{1},f_{2}))=(Q_{\delta ,\phi _{1},\phi
_{2}}(L_{T_{1}}(f_{1})),Q_{\delta ,\phi _{1},\phi _{2}}(L_{T_{2}}(f_{2})).
\end{equation}%
By this we can define the self-consistent transfer operator $\mathcal{L}%
_{\delta }:B_{w}\rightarrow B_{w}$ associated with this system as%
\begin{equation}
\mathcal{L}_{\delta }((f_{1},f_{2}))=L_{\delta ,f_{1},f_{2}}((f_{1},f_{2})).
\label{cupled2}
\end{equation}

We remark that $B_{1}$ can be identified with a closed subset of $L^{1}(%
\mathbb{S}^{1}\times \mathbb{S}^{1})$ by $(f_{1},f_{2})\rightarrow f$ \
where $f$ is defined by $f(x,y)=f_{1}(x)f_{2}(y)$ \ and $\mathcal{L}_{\delta
}$ preserves this subspace.

We now prove the existence and uniqueness of the invariant measure for this
kind of self-consistent operators for small $\delta $, applying our general
statement, Theorem \ref{existence}.

\begin{proposition}
Let $T_{1},T_{2}$ be two $C^{6}$ expanding maps and let $h_{1},h_{2}\in
C^{6}(\mathbb{S}^{1}\times \mathbb{S}^{1},\mathbb{R}).$ Let us consider a
globally coupled system as defined above. There is some $\overline{\delta }$
such that for all $\delta \in \lbrack 0,\overline{\delta }]$ there is a
unique $(f_{1,\delta },f_{2,\delta })\in B_{2}$ such that 
\begin{equation*}
\mathcal{L}_{\delta }((f_{1,\delta },f_{2,\delta }))=(f_{1,\delta
},f_{2,\delta }).
\end{equation*}%
Furthermore there is $M\geq 0$ such that for all $\delta \in \lbrack 0,%
\overline{\delta }]$%
\begin{equation*}
||(f_{1,\delta },f_{2,\delta })||_{B_{2}}\leq M.
\end{equation*}
\end{proposition}

\begin{proof}
The proof follows by the application of Theorem \ref{existence} with $%
B_{w}=B_{1}$ and $B_{s}=B_{2}$. We verify the needed assumptions; the
assumption $(Exi1)$ is trivial, indeed given $\mu =(\phi _{1},\phi _{2})\in
B_{1}$ for any $\delta $ small enough the invariant measure of $L_{\delta
,\phi _{1},\phi _{2}}$, which is \ a system which is the product of two
expanding maps, is trivially in $B_{2}$ \ and if we let $(\phi _{1},\phi
_{2})$ range in $P_{1}$, the $B_{2}$ norm of the associated invariant
measure is uniformly bounded.

The assumption $(Exi2)$ can be easily deduced by Proposition \ref{launo} as
done in the case of system made by coupling identical maps \ obtaining that
there is some $K_{1}\geq 0$ such that%
\begin{equation*}
||L_{\delta ,\phi _{1},\phi _{2}}-L_{\delta ,\phi _{3},\phi
_{4}}||_{B_{2}\rightarrow B_{1}}\leq \delta K_{1}||(\phi _{1},\phi
_{2})-(\phi _{3},\phi _{4})||_{B_{1}}
\end{equation*}%
for $\delta $ ranging in some neighborhood of the origin and $(\phi
_{1},\phi _{2}),$ $(\phi _{3},\phi _{4})\in P_{1}$. Again, applying Lemma %
\ref{lipz}, like done in Proposition \ref{existenceexp} we verify $(Exi3)$
for this product system. Then Theorem \ref{existence} can be applied, giving
the statement.
\end{proof}

It seems that it is possible to extend all the results we proved for coupled
expanding maps to this kind of systems, with the same ideas and estimates
(but longer formulas and computations, as we have two coordinates). This
work however would be quite long and outside of the scope of this paper.

During the revisions of the present paper, the preprint \cite{ST2} was
published. In this work a formalization of a system of coupled maps of
different types and its related self-consistent transfer operators similar
to the one shown in this section was used to study the synchronization of
interacting clusters of globally coupled maps.

\section{The optimal coupling\label{opt}}

In this section\ we study the problem of finding an optimal small coupling
functions $\delta \dot{h}$ in order to maximize the average of a given
observable. This is an optimal control problem in which the goal is to
change the statistical properties of the system in an certain direction, in
some optimal way. In this case we consider an initial uncoupled system and
introduce a small perturbation by a coupling function $\delta \dot{h}$ and
we look for the response of the system to this small perturbation like in
Theorem \ref{thm:linresp}. We suppose the direction of perturbation $\dot{h}$
can vary in some (infinite dimensional) set $P$, and in this set we look for
an optimal one. In the context of extended systems this kind of problems
were also defined as "management of the statistical properties of the
complex system" (\cite{MK}). In some sense this is an inverse problem
related to the linear response, in which the goal is to find the optimal
perturbation giving a certain kind of response. Related problems in which
the focus is more on the realization of a given fixed response have also
been called "linear request" problems (see \cite{GP} and \cite{Kl}).

The problem of finding an optimal \ infinitesimal perturbation, in order to
maximize the average of a given observable and other statistical properties
of dynamics was investigated in the case of finite Markov chains in \cite%
{ADF} and for a class of random dynamical systems whose transfer operators
are Hilbert Schmidt operators in \cite{AFG}.

In this section we start the investigation of these kind of problems in the
case of self-consistent transfer operators. We obtain existence and
uniqueness of the optimal solution under assumptions similar to the ones
used in \cite{AFG}. We will focus on the question of finding the best
coupling in order to optimize the behavior of a given observable. Let us
explain more precisely but still a bit informally the kind of problem we are
going to consider: given a certain system, we consider a set $P$ of allowed
infinitesimal perturbations we can put in the system. It is natural to think
of the set of allowed perturbations $P$ as a convex set because if two
different perturbations of the system are possible, then their convex
combination (applying the two perturbations with different intensities)
should also be possible. We will also consider $P$ as a subset of some
Hilbert space $\mathcal{H}$ (as it is useful for optimization purposes). Let 
$\mu _{\dot{h},\delta }$ \ be the invariant probability measure of the
system after applying a perturbation in the direction $\dot{h}\in P$ with
intensity $\delta $ (we will formalize later what we mean by direction and
intensity in our case). Let the response to this perturbation be denoted as%
\begin{equation*}
R(\dot{h})=\lim_{\delta \rightarrow 0}\frac{\mu _{\dot{h},\delta }-\mu _{0}}{%
\delta }.
\end{equation*}%
\ Let us consider $c:[0,1]\rightarrow \mathbb{R}$. We are interested in the
rate of increasing of the expectation of $c$%
\begin{equation*}
\frac{d(\int c\ d\mu _{\dot{h},\delta })}{d\delta }\bigg|_{\delta =0}
\end{equation*}%
and the element $\dot{h}\in P$ for which this is maximized, thus we are
interested in finding $\dot{h}_{opt}$ such that%
\begin{equation}
\frac{d(\int c\ d\mu _{\dot{h}_{opt},\delta })}{d\delta }\bigg|_{\delta =0}=%
\underset{\dot{h}\in P}{\max }\frac{d(\int c\ d\mu _{\dot{h},\delta })}{%
d\delta }\bigg|_{\delta =0}.  \label{maxprob}
\end{equation}%
By $(\ref{LRidea})$ and $(\ref{LRidea2})$, under the suitable assumptions,
this turns out to be equivalent to finding $\dot{h}_{opt}$ such that%
\begin{equation}
\int c~dR(\dot{h}_{opt})=\max_{\dot{h}\in P}\int c~dR(\dot{h}).  \label{ehoh}
\end{equation}

This is hence the maximization of a certain linear function on the set $P$.

\subsection{Some reminders on optimization of a linear function on a convex
set \label{recalls}}

The optimal perturbation problem we mean to consider is related to the
maximization of a continuous linear function on the set of allowed
infinitesimal perturbations $P$. The existence and uniqueness of an optimal
perturbation hence depends on the properties of the convex bounded set $P$.
\ We now recall some general results, adapted for our purposes, on
optimizing a linear\ continuous function on a convex set. \ Let $\mathcal{J}:%
\mathcal{H}\rightarrow {\mathbb{R}}$ be a continuous linear function, where $%
\mathcal{H}$ is a separable Hilbert space and $P\subset \mathcal{H}$.

The abstract problem we consider then is to find $\dot{h}_{opt}\in P$ such
that 
\begin{equation}
\mathcal{J}(\dot{h}_{opt})=\max_{\dot{h}\in P}\mathcal{J}(\dot{h}).
\label{gen-func-opt-prob}
\end{equation}

The following propositions summarizes some efficient criteria for the
existence and uniqueness of the solution of such problem (see \cite{AFG},
Section 4 for more details and the proofs).

\begin{proposition}[Existence of the optimal solution]
\label{prop:exist} Let $P$ be bounded, convex, and closed in $\mathcal{H}$.
Then, Problem \eqref{gen-func-opt-prob} has at least one solution.
\end{proposition}

Uniqueness of the optimal solution will be provided by strict convexity of $%
P $.

\begin{definition}
\label{stconv}We say that a convex closed set $A\subseteq \mathcal{H}$ is 
\emph{strictly convex} if for all pair $x,y\in A$ and for all $0<\gamma <1$,
the points $\gamma x+(1-\gamma )y\in \mathrm{int}(A)$, where the relative
interior\footnote{%
The relative interior of a closed convex set $C$ is the interior of $C$
relative to the closed affine hull of $C$.} is meant.
\end{definition}

\begin{proposition}[Uniqueness of the optimal solution]
\label{prop:uniqe} Suppose $P$ is closed, bounded, and strictly convex
subset of $\mathcal{H}$, and that $P$ contains the zero vector in its
relative interior. If $\mathcal{J}$ is not uniformly vanishing on $P$ then
the optimal solution to \eqref{gen-func-opt-prob} is unique.
\end{proposition}

We remark that in the case $\mathcal{J}$ is uniformly vanishing, all the
elements of $P$ are solutions of the problem $(\ref{ehoh}).$

\subsection{Optimizing the response of the expectation of an observable}

\label{subsubsec1} Let $c\in L^{1}$ be a given observable. We consider the
problem of finding an infinitesimal perturbation that maximizes the
expectation of $c.$ \ As motivated before, we want to solve the problem
stated in $(\ref{ehoh})$. Suppose that $P$ is a closed, bounded, convex
subset of $\mathcal{H}$ containing the zero perturbation, and that $\mathcal{%
J}$ is not uniformly vanishing on $P$. Let us consider the function $%
\mathcal{J}(\dot{h})=\int c~dR(\dot{h}).$ When this function is continuous
as a map from $(P,\Vert \cdot \Vert _{\mathcal{H}})$ to $\mathbb{R}$, we may
immediately apply Proposition \ref{prop:exist} to obtain that there exists a
solution to the problem considered in \textit{\ }$(\ref{ehoh})$\textit{.}
If, in addition, $P$ is strictly convex and $\mathcal{J}$ is nonvanishing,
then by Proposition \ref{prop:uniqe} the solution to $(\ref{ehoh})$ is
unique.

In the following subsections we hence apply these remarks to find the
existence and uniqueness of the optimal coupling in the case of coupled
expanding map and maps with additive noise.

\subsubsection{The optimal coupling for expanding maps}

We consider self-consistent transfer operators coming from a system of
coupled maps as in Section \ref{secmap}, where $\mathcal{L}_{0}$ is the
uncoupled operator and $\mathcal{L}_{\delta }$ is the self-consistent
operator with coupling driven by a function $\dot{h}:\mathbb{S}^{1}\times 
\mathbb{S}^{1}\rightarrow \mathbb{R}$ and with strength $\delta .$ We proved
in Proposition \ref{thm:linresp copy(1)} (see $($\ref{iii}$)$) that the
response of the invariant measure of the system as $\delta $ increases is
given by 
\begin{equation*}
R(\dot{h})=(Id-L_{0})^{-1}(h_{0}\int_{S^{1}}\dot{h}(x,y)h_{0}(y)dy)^{\prime
}.
\end{equation*}

Given some observable $c\in L^{1}$ and some convex set of allowed
perturbations $P$ we now apply the previous results to the problem of
finding the optimal coupling $\dot{h}_{opt}\in P$ solving the problem $(\ref%
{ehoh})$ for this response function \ $R(\dot{h})$. \ From Remark \ref{dalla}
(see also $(\ref{LRidea2})$) we know that the rate of change of the average
of $c$ can be estimated by the linear response when the convergence of the
linear response is in $W^{1,1}$.

We remark that to apply the general results of Section \ref{recalls} we need 
$P$ \ being a subset of a Hilbert space. \ Since to apply Proposition \ref%
{thm:linresp copy(1)} we need $\dot{h}\in C^{6}$ we consider a Hilbert space
of perturbations which is included in $C^{6}.$ A simple choice is $W^{7,2}$.
We hence consider a system with coupled expanding maps, the Hilbert space $%
W^{7,2}$ and a convex set $P\subseteq W^{7,2}(S^{1}\times S^{1}).$

\begin{proposition}
Under the above assumptions, supposing that $P$ is a closed bounded convex
set in $W^{7,2},$ Problem $(\ref{ehoh})$ has a solution in $P.$ If
furthermore $P$ is strictly convex either the optimal solution is unique or
every $\dot{h}\in P$ is the optimal solution.
\end{proposition}

\begin{proof}
The result directly follows applying Propositions \ref{prop:exist} \ and \ref%
{prop:uniqe}. In order to apply the propositions we have to check that $\dot{%
h}\rightarrow \int c~dR(\dot{h})$ is continuous on $P.$ Since 
\begin{equation*}
R(\dot{h})=(Id-L_{0})^{-1}(h_{0}\int_{S^{1}}\dot{h}(x,y)h_{0}(y)dy)^{\prime }
\end{equation*}%
we have 
\begin{eqnarray*}
\int c~dR(\dot{h}) &\leq &||c||_{1}||(Id-L_{0})^{-1}(h_{0}\int_{S^{1}}\dot{h}%
(x,y)h_{0}(y)dy)^{\prime }||_{\infty } \\
&\leq &||c||_{1}||(Id-L_{0})^{-1}||_{W^{1,1}\rightarrow
W^{1,1}}||(h_{0}\int_{S^{1}}\dot{h}(x,y)h_{0}(y)dy)^{\prime }||_{W^{1,1}} \\
&\leq &||c||_{1}||(Id-L_{0})^{-1}||_{V_{W^{1,1}}\rightarrow
W^{1,1}}2||h_{0}||_{C^{3}}^{2}||\dot{h}||_{W^{2,1}}.
\end{eqnarray*}

Now the result follow by a direct application of \ Propositions \ref%
{prop:exist} \ and \ref{prop:uniqe}.
\end{proof}

\subsubsection{The optimal coupling for systems with additive noise}

Now we consider the optimal coupling in order to maximize the average of one
observable $c$ in the case of the coupled maps with additive noise as
described in Section \ref{noise}. Since Proposition \ref{thm:linresp copy(2)}
gives a convergence of the linear response in the strong space $C^{k}$, by
Remark \ref{dalla} we know that we can consider very general observables.
For simplicity we will consider $c\in L^{1}$ but in fact we could consider
even weaker spaces as distribution spaces (the dual of $C^{k}$). For
simplicity we also take $P\subseteq W^{1,2}$ to let $($\ref{11121}$)$ make
sense. The response formula in this case is%
\begin{equation*}
R(\dot{h})=(Id-L_{0})^{-1}\rho \ast (L_{T_{0}}(h_{0})\int_{S^{1}}\dot{h}%
(x,y)h_{0}(y)dy)^{\prime }.
\end{equation*}%
We will hence consider the problem $(\ref{ehoh})$ with this response
function. Similarly to the expanding maps case we get the following
statement.

\begin{proposition}
Under the above assumptions, supposing that $P$ is a closed bounded convex
set in $W^{1,2},$ Problem $(\ref{ehoh})$ has a solution in $P.$ If
furthermore $P$ is strictly convex either the solution is unique or every $%
\dot{h}\in P$ is the optimal solution.
\end{proposition}

\begin{proof}
The result again directly follows applying Propositions \ref{prop:exist} \
and \ref{prop:uniqe}. In order to apply the propositions we check that $\dot{%
h}\rightarrow \int c~dR(\dot{h})$ is continuous on $P.$ Since in this case 
\begin{equation*}
R(\dot{h})=(Id-L_{0})^{-1}\rho \ast (h_{0}\int_{S^{1}}\dot{h}%
(x,y)h_{0}(y)dy)^{\prime }
\end{equation*}

we have 
\begin{eqnarray*}
\int c~dR(\dot{h}) &\leq &||c||_{1}||(Id-L_{0})^{-1}\rho \ast
(h_{0}\int_{S^{1}}\dot{h}(x,y)h_{0}(y)dy)^{\prime }||_{\infty } \\
&\leq &||c||_{1}||(Id-L_{0})^{-1}||_{V_{C^{k}}\rightarrow L^{\infty
}}2||\rho ||_{C^{k}}||h_{0}||_{W^{1,1}}^{2}||\dot{h}||_{W^{2,1}}
\end{eqnarray*}

establishing the continuity of $\dot{h}\rightarrow \int c~dR(\dot{h})$.
\end{proof}

\noindent \textbf{Acknowledgments.} S.G. is partially supported by the
research project PRIN 2017S35EHN\_004 "Regular and stochastic behavior in
dynamical systems" of the Italian Ministry of Education and Research. The
author whish to thank F. S\'{e}lley, S. Vaienti, M. Tanzi and C. Liverani
for fruitful discussions during the preparation of the work. The author
whish also to thank the anonymous reviewers for their patience and the help
in revising this long text.\newline


\begin{thebibliography}{99}
\bibitem{AFG} F. Antown, G., Froyland, S. Galatolo \emph{Optimal linear
response for Markov Hilbert-Schmidt integral operators and stochastic
dynamical systems }arXiv:2101.09411

\bibitem{ADF} F. Antown, D., Dragicevic, and G. Froyland. \emph{Optimal
linear responses for Markov chains andstochastically perturbed dynamical
systems}. J. Stat. Phys., 170(6): 1051--1087, (2018).

\bibitem{Ba2} V. Baladi, \emph{Linear response, or else}, \ ICM proceedings
Seoul 2014\emph{\ ,} vol III, pp525-545.

\bibitem{BKLZ} J.B. Bardet, G. Keller, R. Zweim\"{u}ller. \emph{%
Stochastically stable globally coupled maps with bistable thermodynamic
limit.} Comm. Math. Phys. 292 (1), :237--270, (2009).

\bibitem{Bla} M. L. Blank. \emph{Ergodic averaging with and without
invariant measures.} Nonlinearity, 30(12):4649, (2017).

\bibitem{Bl11} M. L. Blank. \emph{Self-consistent mappings and systems of
interacting particles.} Doklady Math., 83, pp 49--52. Springer, (2011).

\bibitem{bal} P. Balint, G. Keller, F. M. Selley, I.\ P. Toth, \emph{%
Synchronization versus stability of the invariant distribution for a class
of globally coupled maps. }Nonlinearity, 31(8), 3770. (2018)

\bibitem{CF} J-R. Chazottes, B. Fernandez.\emph{\ Dynamics of coupled map
lattices and of related spatially extended systems}, v. 671. Springer
Science \& Business Media, (2005).

\bibitem{tak} T. Chihara, Y. Sato, I Nisoli, S Galatolo \ \emph{Existence of
multiple noise-induced transitions in a Lasota-Mackey map }Chaos 32, 013117
(2022); https://doi.org/10.1063/5.0070198

\bibitem{C} J. Conway \emph{A course in functional analysis,} volume 96.
Springer Science \& Business Media (2013)

\bibitem{CR} J.P. Conze, A Raugi, \emph{Limit theorems for sequential
expanding dynamical systems} on , Ergodic theory and related fields,
Contemp. Math., vol. 430, AMS, Providence, RI, pp. 89--121, (2007).

\bibitem{F2} H. Dietert, B. Fernandez \emph{The mathematics of asymptotic
stability in the Kuramoto model.} Proceedings of the Royal Society A,
474(2220), 20180467. (2018)

\bibitem{Fe} B. Fernandez, \emph{Breaking of ergodicity in expanding systems
of globally coupled piecewise affine circle maps}. Journal of Statistical
Physics 154, 4 , 999-1029 (2014).

\bibitem{Fl} M. Florenzano \emph{General Equilibrium Analysis: Existence and
Optimality Properties of Equilibria } Springer, ISBN 9781402075124, (2003)

\bibitem{GS} S. Galatolo, J. Sedro \emph{Quadratic response of random and
deterministic dynamical systems }Chaos 30, 023113 (2020);
https://doi.org/10.1063/1.5122658

\bibitem{Gdisp} S. Galatolo \emph{Statistical properties of dynamics.
Introduction to the functional analytic approach \ }arXiv:1510.02615

\bibitem{GP} S. Galatolo, M. Pollicott \emph{Controlling the statistical
properties of expanding maps.} Nonlinearity 30, 7 , 2737, (2017).

\bibitem{GNS} S. Galatolo, I. Nisoli, B. Saussol. \emph{An elementary way to
rigorously estimate convergence to equilibrium and escape rates.} J. Comput.
Dyn., 2015, 2 (1) : 51-64. doi: 10.3934/jcd.2015.2.51

\bibitem{gani} S. Galatolo, I. Nisoli \emph{An Elementary Approach to
Rigorous Approximation of Invariant Measures }SIAM J. Appl. Dyn. Syst. 13-2
(2014), pp. 958-985

\bibitem{GMN} S. Galatolo, M. Monge, I. Nisoli \emph{Existence of Noise
Induced Order, a Computer Aided Proof \ }Nonlinearity, 33(9):4237--4276,
(2020).

\bibitem{GMN2} S. Galatolo, M. Monge, I. Nisoli \emph{Rigorous approximation
of stationary measures and convergence to equilibrium for iterated function
systems }J. Phys. A: Math. Theor. 49, 274001, (2016).

\bibitem{GG} S Galatolo and P Giulietti \emph{A linear response for
dynamical systems with additive noise} Nonlinearity 32, 6, pp. 2269-2301 \
(2019)

\bibitem{go} F. Golse \emph{On the dynamics of large particle systems in the
mean-field limit.} In Macroscopic and large scale phenomena: coarse
graining, mean field limits and ergodicity (eds A Muntean, J Rademacher, A
Zagaris). Lecture Notes Application Mathematics and Mechanics, vol. 3, pp.
1--144. Springer (2016).

\bibitem{JD} M. Jiang, R. de la Llave \emph{Smooth dependence of
thermodynamic limits of SRB-measures.} Communications in Mathematical
Physics 211, 2 (2000), 303-333.

\bibitem{Jd2} M. Jiang, R. de la Llave \emph{Linear response function for
coupled hyperbolic attractors}. Communications in Mathematical Physics 261,
2 (2006), 379-404.

\bibitem{JP} M. Jiang, Y. B. \ Pesin. \emph{Equilibrium measures for coupled
map lattices: Existence, uniqueness and finite-dimensional approximations.}
Comm. Math. \ Phys. 193(3):675--711 (1998). \ 

\bibitem{L2} C. Liverani, \emph{\ Invariant measures and their properties. A
functional analytic point of view}, Dynamical Systems. Part II: Topological
Geometrical and Ergodic Properties of Dynamics. Centro di Ricerca Matematica
\textquotedblleft Ennio De Giorgi\textquotedblright : Proceedings. Published
by the Scuola Normale Superiore in Pisa (2004).

\bibitem{L3} C. Liverani \emph{Rigorous numerical investigation of the
statistical properties of piecewise expanding maps. A feasibility study. }%
Nonlinearity 14, 463--490 \ (2001)

\bibitem{LM} A. Lasota and M. Mackey. \emph{Probabilistic properties of
deterministic systems.} Cambridge university press, (1985).

\bibitem{Kl} B. Kloeckner \emph{The linear request problem.} Proceedings of
the American Mathematical Society 146, 7, 2953-2962 (2018).

\bibitem{Kan} K. Kaneko. \emph{Globally coupled chaos violates the law of
large numbers but not the central limit theorem}. Physical review letters,
65(12):1391 (1990).

\bibitem{KF} A. Kolmogorov and S. Fomin\emph{. Elements of the Theory of
Functions and Functional Analysis}.Volume 2: Measure. The Lebesgue Integral.
Hilbert Space. Graylock, (1961).

\bibitem{K} G. Keller. \emph{An ergodic theoretic approach to mean field
coupled maps.} In Fractal Geometry and Stochastics II, pages 183--208.
Springer, (2000).

\bibitem{KL05} G. Keller, C. Liverani.\emph{\ A spectral gap for a
one-dimensional lattice of coupled piecewise expanding interval maps.} In
Dynamics of coupled map lattices and of related spatially extended systems,
pp 115--151. Springer, 2005.

\bibitem{KL} G. Keller ; C. Liverani \emph{Stability of the spectrum for
transfer operators \ }Annali della Scuola Normale Superiore di Pisa, 4, 28
(1999) no. 1, pp. 141-152.

\bibitem{jst} L. Marangio, J. Sedro, S. Galatolo, A. Di Garbo, M. Ghil \ 
\emph{Arnold maps with noise: differentiability and non-monotonicity of the
rotation number.} J. Stat. Phys. 179, 1594--1624, (2020)

\bibitem{MK} R. MacKay \emph{Management of complex dynamical systems.}
Nonlinearity 31, 2 (2018), R52.

\bibitem{OSY} W. Ott, M. Stenlund, L-S Young. \emph{Memory loss for
time-dependent dynamical systems.} Mathematical research letters,
16(2):463--475, 2009.

\bibitem{R} D. Ruelle, \emph{Differentiation of SRB states}, \textit{Comm.
Math. Phys.} 187 (1997) 227--241.

\bibitem{Sed} J. Sedro \emph{A regularity result for fixed points, with
applications to linear response.} Nonlinearity 31 1417 (2018).

\bibitem{ST} F. M. S\'{e}lley, M. Tanzi \ \emph{Linear Response for a Family
of Self-Consistent Transfer Operators \ }Comm. Math. Phys. 382(3), 1601-1624
(2021).

\bibitem{ST2} F. M. S\'{e}lley, M. Tanzi \ \emph{Synchronization for
Networks of Globally Coupled Maps in the Thermodynamic Limit (}%
arXiv:2110.05618).

\bibitem{Sell} F. M. S\'{e}lley. \emph{Asymptotic properties of mean field
coupled maps.} PhD thesis, (2019).

\bibitem{Se2} F. M. S\'{e}lley \emph{A self-consistent dynamical system with
multiple absolutely continuous invariant measures }\ Journal of
Computational Dynamics, 8(1), 932 \ (2021).

\bibitem{WG} C. Wormell, G. Gottwald \emph{On the validity of linear
response theory in high-dimensional deterministic dynamical systems.}
Journal of Statistical Physics 172, 6 (2018), 1479-1498.

\bibitem{WG2} C. L. Wormell, G. A. Gottwald \emph{Linear response for
macroscopic observables in high-dimensional systems.} Chaos: An
Interdisciplinary Journal of Nonlinear Science 29, 11 (2019), 113127.

\bibitem{Viana} M. Viana \emph{Lectures on Lyapunov Exponents}. Cambridge
Studies in Advanced Mathematics 145, Cambridge University Press (2014)
\end{thebibliography}
\end{document}